\documentclass[a4paper, 10pt, twoside, reqno]{amsart} 
\usepackage[a4paper,inner=2cm,outer=2cm,top=2.5cm,bottom=2.5cm]{geometry}
\usepackage{amsmath,amscd}
\usepackage{amssymb}
 
\def\diag{ {\rm diag} }

\usepackage{lmodern}
\usepackage{amsthm}
\usepackage{comment}
\usepackage{mathrsfs}
\usepackage{graphicx, xcolor}
\usepackage{xcolor}
\usepackage{mathtools}
\usepackage[dvipsnames]{xcolor}
\definecolor{myred}{RGB}{220,20,60}
\definecolor{chatakred}{RGB}{255,0,0}
\usepackage[ocgcolorlinks, linkcolor=blue,citecolor=red,urlcolor=blue]{hyperref}
\usepackage{bm}
\usepackage{bbm}
\usepackage{url}
\usepackage[utf8]{inputenc}
\usepackage{mathtools,amssymb}
\usepackage{yhmath}
    \usepackage{amssymb}
\usepackage{tikz}
\usepackage{dsfont}
\usepackage{relsize}
\usepackage{url}
\usepackage{xcolor}
\usepackage{graphicx}
\usepackage{mathrsfs}
\usepackage[shortlabels]{enumitem}
\usepackage{lineno}
\usepackage{amsmath}
\usepackage{enumitem}
\usepackage{amsthm} 
\usepackage{verbatim}
\usepackage{dsfont}
\usepackage[utf8]{inputenc}
\usepackage{tikz}
\numberwithin{equation}{section}

\allowdisplaybreaks
 
 \mathtoolsset{showonlyrefs}
\theoremstyle{plain}
\newtheorem{theorem}{Theorem}[section]
\newtheorem{lemma}[theorem]{Lemma}
\newtheorem{corollary}[theorem]{Corollary}

\newtheorem{proposition}[theorem]{Proposition}

\newtheorem{remark}[theorem]{Remark}
\allowdisplaybreaks[0]
\usepackage{blindtext}
\usepackage{cleveref}
\renewcommand{\O}{\Omega}
\renewcommand{\o}{\omega}
\title[Stability result for reaction-diffusion-convection equation] {
Stable determination of time-dependent coefficients in a
reaction-diffusion-convection system}
\date{}
\author[Bhardwaj and Kumar]{Rahul Bhardwaj{$^{*}$} and Parveen Kumar$^{\diamond}$}
\address{{$^{*}$ Department of Mathematics, Indian Institute of Technology Ropar, Rupnagar- 140001, Punjab, India.
		\newline
		\indent E-mail:{\tt\  bhardwaj161067@gmail.com}}}
        \address{$^{\diamond}$ Department of Mathematics, Indian Institute of Technology Ropar, Rupnagar- 140001, Punjab, India.
	\newline\indent E-mail:{\tt \ parveen.24maz0013@iitrpr.ac.in}}
    \thanks{{\bf Mathematics Subject Classification (2020)}: Primary: 35R30, 35K20; Secondary: 35K45, 35R25, 35B30.}
    \thanks{{\bf Key words and phrases.}  Inverse problem, reaction-diffusion-convection system, uniqueness, geometric optics solutions, stability, Carleman estimate.}

\DeclareMathOperator{\supp}{supp} %support
\begin{document}
	\begin{abstract}
In this manuscript, we investigate an inverse boundary value problem for a reaction-diffusion-convection system in a bounded domain of $\mathbb{R}^{1+n}$, $n\geq 2$. We aim to obtain a stability estimate for determining the time-dependent convection coefficient and matrix-valued potential from boundary measurements represented by the Dirichlet-to-Neumann map. We consider a partial data setting in which the measurements are available only on a subset of the lateral boundary that slightly exceeds one-half of the boundary.
We first establish the well-posedness of the associated initial-boundary value problem. Subsequently, by combining Carleman estimates with suitable geometric optics solutions, we derive stability estimates for the unknown coefficients. More precisely, we prove a double logarithmic ($\log$-$\log$) stability estimate for the time-dependent convection coefficient from the knowledge of the partial Dirichlet-to-Neumann map. This stability result is then employed to recover the matrix-valued potential, yielding a triple logarithmic ($\log$-$\log$-$\log$) stability estimate for the zeroth-order coefficient.
	\end{abstract}
	\maketitle
\section{Introduction}
\subsection{Mathematical formulation and problem of interest}\label{Mathematical formulation} 
Let $T > 0$ be fixed and let $\Omega \subset \mathbb{R}^n$ $(n \geq 2)$ be a bounded open set with smooth boundary $\partial \Omega$. We define the cylindrical space-time domain $\Omega_T := (0,T) \times \Omega$, and denote its lateral boundary by $\partial \Omega_T := (0,T) \times \partial \Omega$. Throughout the article, we use boldface notation to denote spaces of vector-valued functions, with each component belonging to the corresponding underlying function space. For instance, $\mathbf{L}^2(X)$ denotes the space of vector-valued functions on a measurable set $X$ whose components belong to $L^2(X)$. 
Also for any function $w: \O_T \rightarrow\mathbb{R}$, we denote by $\partial_t w$, the partial derivative of $w$ with respect to time variable $t$  and by $\nabla_x w$, the gradient operator with respect to space variable $x$ which is given by $\nabla_x w:=(\partial_{x_1}w,\partial_{x_2}w, \cdots , \partial_{x_n}w)$.
In the present article, we consider the following linear reaction-diffusion-convection operator 
\begin{align}
 \label{Operator L_{A,Q}}
  \hspace{-.6cm} \displaystyle     \mathcal{L}_{A,Q}\overrightarrow{u}(t,x) &:= \begin{bmatrix}
 \partial_t u_1(t,x)-  \displaystyle \sum_{j=1}^{n}(\partial_j+ A^1_j(t,x))^2u_1(t,x)+ \sum_{j=1}^k q_{1j}(t,x) u_j(t,x)
\\ 
\partial_t u_2(t,x)- \displaystyle  \sum_{j=1}^{n}(\partial_j+ A^2_j(t,x))^2u_2(t,x)+ \sum_{j=1}^k q_{2j}(t,x) u_j(t,x)\\
\vdots\\
\partial_t u_k(t,x)-  \displaystyle \sum_{j=1}^{n}(\partial_j+ A^k_j(t,x))^2u_k(t,x)+ \sum_{j=1}^k q_{kj}(t,x) u_j(t,x)
\end{bmatrix}, \quad (t,x)\in \Omega_T,
\end{align} where   $\overrightarrow{u}:= \left(u_1, u_2, \cdots, u_k\right)^T$ represents the flow field, $A$  and $Q$ represent the convection term and the matrix-valued potential respectively.
For the convection term 
\(
    A(t,x):= \operatorname{diag}\bigl(\overrightarrow{A}^1(t,x),\overrightarrow{A}^2(t,x), \ldots, \overrightarrow{A}^k(t,x)\bigr),
\)
we assume that each vector field \(\overrightarrow{A}^i(t,x):=\bigl(A^i_1(t,x), A^i_2(t,x), \ldots, A^i_n(t,x)\bigr)\), for \(1\leq i\leq k\), belongs to \(H^{\ell}(\Omega_T;\mathbb{R}^n)\). Also, for
the matrix-valued potential 
\(
    Q(t,x) := \bigl(q_{ij}(t,x)\bigr)_{1\leq i,j\leq k},
\)
we assume \(q_{ij}(t,x) \in H^{\ell-1}(\Omega_T;\mathbb{R})\) for all \(1\leq i,j\leq k\), where 
\(
    \ell > \frac{1+n}{2} + 3.
\)
Next, for two vectors  $\overrightarrow{u}(t,x):=(u_1(t,x),u_2(t,x),\cdots,u_k(t,x))^T$ and $\overrightarrow{v}(t,x):=(v_1(t,x),v_2(t,x),\cdots,v_k(t,x))^T$, $(t,x)\in\Omega_T$, we define 
\begin{align*}
    \partial_t \overrightarrow{u}:=\begin{bmatrix}
    \partial_t u_1\\\partial_t u_2 \\ \vdots \\\partial_t u_k
\end{bmatrix},\; \nabla_x\overrightarrow{u}:=\begin{bmatrix}
    \nabla_x u_1\\\nabla_x u_2 \\ \vdots \\\nabla_x u_k
\end{bmatrix} ,\;\Delta_x\overrightarrow{u}:=\begin{bmatrix}
    \Delta_x u_1\\\Delta_x u_2 \\ \vdots \\\Delta_x u_k
\end{bmatrix}\quad \text{and}\quad \overrightarrow{u}\cdot \overrightarrow{v}=\sum_{j=1}^k u_i v_i.
\end{align*}
 In view of the above notations, the operator $\mathcal{L}_{A,Q}$ in compact form is given by 
\begin{align*}
\mathcal{L}_{A,Q}\overrightarrow{u}=\partial_t \overrightarrow{u}-\Delta_x \overrightarrow{u}-2A\cdot\nabla_x\overrightarrow{u}-(\nabla_x\cdot A)\overrightarrow{u}- A^2\overrightarrow{u}+Q\overrightarrow{u},
\end{align*}
where
\(
    A\cdot \nabla_x \overrightarrow{u}:=\left(\overrightarrow{A}^{1}\cdot\nabla_x {u}_1,\overrightarrow{A}^{2}\cdot\nabla_x {u}_2,\cdots,\overrightarrow{A}^{k}\cdot\nabla_x{u}_k\right)
\text{ and }
    A^2:=\diag \left(\lvert \overrightarrow{A}^1\rvert^2, \lvert \overrightarrow{A}^2\rvert^2, \cdots, \lvert \overrightarrow{A}^k\rvert^2\right).
\)
In this article, we consider a coefficient identification  inverse problem related to the following initial-boundary value problem (IBVP) for a linear reaction-diffusion-convection equation 
 \begin{equation}{\label{PDE}}
\left\{\begin{array}{r l c}
\displaystyle \mathcal{L}_{A,Q}\overrightarrow{u}(t,x) &= \overrightarrow{0},\quad& (t,x)\in \Omega_T,\\[2pt]
\overrightarrow{u}(0,x)&= \overrightarrow{0}, \quad& x\in \Omega,  \\[2pt]
\overrightarrow{u}(t,x)&= \overrightarrow{F}(t,x),\quad&  (t,x)\in \partial\Omega_T .
\end{array}\right.
\end{equation}
Before proceeding to the main results, we recall the well-posedness of the forward problem associated with IBVP \eqref{PDE}. To formulate the existence and uniqueness theorem for the IBVP \eqref{PDE}, as well as the corresponding boundary operators, we first introduce a few  notations and function spaces, which  will be used throughout this article. 
Motivated by \cite{Bellassoued_2021, senapati2021stability}, we define, for $m>0$, the admissible set of convection coefficients $A$ and matrix-valued potentials $Q$ as follows:
 \begin{align*}
 \mathcal{A}_m := \bigg\{(A, Q) : \text{for all } 1 \leq i,j \leq k,\ \overrightarrow{A}^{\,i} \in \mathbf{H}^{\ell}(\Omega_{T}),\ q_{ij} \in H^{\ell-1}(\Omega_{T}),\text{ where } \ell > \dfrac{n+1}{2} + 3,\ \text{and}\\[4pt]
\big\lVert A \big\rVert_{\mathbf{H}^{\ell}(\Omega_{T})} + \big\lVert Q \big\rVert_{\mathbf{H}^{\ell-1}(\Omega_{T})} \leq m \bigg\}.
 \end{align*}
 Here,
 \begin{align*}
   \lVert A\rVert_{\mathbf{H}^{\ell}(\O_T)}:= \left(\sum_{i=1}^k \lVert \overrightarrow{A}^i \rVert_{\mathbf{H}^{\ell}(\O_T)}^2\right)^{1/2} \text{ and } \;\; \lVert Q \rVert_{\mathbf{H}^{\ell-1}(\O_T)}:=\left(\sum_{i,j=1}^k \lVert q_{ij} \rVert_{H^{\ell-1}(\O_T;\mathbb{R})}^2\right)^{1/2}.  
 \end{align*}
 Next, for any non-negative real number $r,s$ and for $U = \O$ or $U = \partial\O$, we define the time-dependent Sobolev spaces  by
  \begin{align*}
      H^{r,s}((0,T)\times U)&:= L^2(0,T: H^r(U)) \cap H^s(0,T; L^2(U)),
  \end{align*}
  and
  \begin{align*}
      {}_{0}H^{r,s}(\partial \O_T)&:=\{u\in H^{r,s}(\partial \O_T);u(0,x)=0, x\in \O\}. 
  \end{align*}
 Both the Sobolev spaces $H^{r,s}((0,T)\times U)$ and ${}_{0}H^{r,s}(\partial \O_T)$ defined above are  equipped with the following norm
  \begin{align}
      \lVert u \rVert_{H^{r,s}((0,T)\times U)}^2= \lVert u \rVert^2_{L^{2}(0,T;H^{r}( U))}+\lVert u \rVert^2_{H^{s}(0,T; L^2(U))}.
  \end{align}
Now it follows from Theorem \ref{existence in forward pde} (see below)  that the IBVP \eqref{PDE} admits a unique solution $\overrightarrow{u}\in \mathbf{H}^{2,1}(\Omega_T)$ whenever  $\overrightarrow{F}\in{}_{0}\mathbf{H}^{\frac{3}{2},\frac{3}{4}}(\partial \O_T)$. Moreover, the solution $\overrightarrow{u}$ satisfies the following estimate
\begin{align*}
  \lVert \overrightarrow{u}\rVert_{\mathbf{H}^{2,1}(\Omega_T)}  \leq C \lVert\overrightarrow{F}\rVert_{{}_{0}\mathbf{H}^{\frac{3}{2},\frac{3}{4}}(\partial \O_T)},  
\end{align*}
where $C$ depends only on $\Omega,\ T$ and $m$.
In addition, we analyze the well-posedness of the problem for a more general operator in which, instead of a diagonal matrix, a full matrix is employed in the convection term (see below Theorem \ref{existence in forward pde} for more details). Finally, with the help of Theorem \ref{existence in forward pde}, we observe that the operator  $\mathcal{N}_{A,Q}$ given by 
\begin{align}\label{Weak form of DN map}
\begin{aligned}
    &\Big\langle \mathcal{N}_{A,Q}\overrightarrow{u},\overrightarrow{w}\rvert_{\partial \O_T}\Big\rangle\\[4pt]& :=\displaystyle \int_{\O_T} \left(-\overrightarrow{u}\cdot\overline{\partial_t \overrightarrow{w}}+\nabla_x \overrightarrow{u}\cdot \overline{\nabla_x \overrightarrow{w}}+2 A\overrightarrow{u}\cdot \overline{\nabla_x \overrightarrow{w}}+ (\nabla_x\cdot A) \overrightarrow{u}\cdot \overline{\overrightarrow{w}}- A^2\overrightarrow{u}\cdot \overline{\overrightarrow{w}}+ Q\overrightarrow{u}\cdot \overline{\overrightarrow{w}}\right)\ dx dt,
    \end{aligned}
\end{align}
is well-defined for $\overrightarrow{w} \in \mathbf{H}^{2,1}(\Omega_T)$ such that $\overrightarrow{w}(T,x)=0$, for $x \in \O$ and  $\overrightarrow{u}\in \mathbf{H}^{2,1}(\Omega_T)$, a solution to the IBVP \eqref{PDE}. Now, using integration by parts, we have
\begin{align*}
    \Big\langle \mathcal{N}_{A,Q}\overrightarrow{u},\overrightarrow{w}\rvert_{\partial \O_T}\Big\rangle&=\displaystyle \int_{\O_T} \mathcal{L}_{A,Q}\overrightarrow{u}\cdot \overline{\overrightarrow{w}}\ dx dt-\int_{\O} \left(\overrightarrow{u}(T,x)\cdot \overline{\overrightarrow{w}}(T,x)-\overrightarrow{u}(0,x)\cdot \overline{\overrightarrow{w}}(0,x)\right)  dx\\[2pt]&\qquad +\int_{\partial \O_T} \partial_\nu \overrightarrow{u}\cdot \overline{\overrightarrow{w}} \ dS_x dt+\int_{\partial \O_T} 2\left(\nu\cdot A\right) \overrightarrow{u}\cdot \overline{\overrightarrow{w}} \
     dS_x dt\\[2pt]
     &=\int_{\partial \O_T} \left (\partial_\nu \overrightarrow{u}+ 2\left(\nu\cdot A\right) \overrightarrow{u}\right)\cdot \overline{\overrightarrow{w}} \
     dS_x dt.
\end{align*}
This implies the following identity,  whenever $\overrightarrow{u}$ is a solution to the IBVP \eqref{PDE} 
\begin{align}
    \mathcal{N}_{A,Q}\overrightarrow{u}=\left(\partial_\nu \overrightarrow{u} + 2\left(\nu\cdot A\right) \overrightarrow{u}\right)\big|_{\partial \Omega_T},
\end{align}
 where $\nu$ denotes the  outward unit normal vector on $\partial \Omega$ and $\partial_{\nu}\overrightarrow{u}:=\left(\partial_{\nu}u_1, \partial_{\nu} u_2, \cdots, \partial_{\nu}u_k\right)^T$. Inspired by  this, we define the associated Dirichlet-to-Neumann (DN) map, denoted by \(\Lambda_{A,Q}\), as a bounded linear operator
\begin{align*}
    \Lambda_{A,Q}:{}_{0}\mathbf{H}^{\frac{3}{2},\frac{3}{4}}(\partial \O_T)\rightarrow \mathbf{H}^{\frac{1}{2},\frac{1}{4}}(\partial \O_T),
\end{align*}  by 
\begin{align}{\label{DN map}}
    \Lambda_{A,Q}(\overrightarrow{F}):= \mathcal{N}_{A,Q}\overrightarrow{u},
\end{align}
where $\overrightarrow{u}$ is a unique solution to the IBVP \eqref{PDE} corresponding to the Dirichlet boundary data $\overrightarrow{F}$.
Now, to define the partial DN map, we need to specify the region of $\partial \O$ where the boundary measurements are available. To do so, we fix  $\o_0 \in \mathbb{S}^{n-1}$ and  define the front face and back face of $\partial \O$ with respect to $\o_0$, respectively, as follows (see \cite{senapati2021stability})
\begin{align}\label{front and back face}
F(\o_0) := \{x \in \partial\Omega: \o_0\cdot\nu(x) \leq 0 \} \quad \mbox{and} \quad 
B(\o_0) := \{x \in \partial\Omega: \o_0\cdot\nu(x) \geq 0 \},
\end{align}
where $\nu(x)$ is the unit normal vector pointing outside at $x\in \partial\O$. Since the boundary information is available on just slightly more than half of the boundary, we consider $\epsilon>0$ sufficiently small and  define $\widetilde{F}$ and $\widetilde{B}$  by
\begin{align}\label{front and back face partial}
\widetilde{F}(\o_0) := \left\{x \in \partial\Omega: \o_0\cdot\nu(x) < \frac{\epsilon}{2}  \right\} \quad \mbox{and} \quad 
\widetilde{B}(\o_0) := \left\{x \in \partial\Omega: \o_0\cdot\nu(x) > \frac{\epsilon}{2}  \right\},   
\end{align}
 where $\widetilde{F}(\o_0)$ is a sufficiently small open neighborhood of $F(\o_0)$. Let us denote the partial DN map by $\Lambda^{\sharp}$ and is given by \begin{align}\label{partial dn map}
 \Lambda^{\sharp}(\overrightarrow{F}):=\left(\partial_\nu \overrightarrow{u} + 2\left(\nu\cdot A\right) \overrightarrow{u}\right)\big|_{(0,T)\times \widetilde{F}(\o_0)}.
 \end{align}
 In this article, we consider the following inverse problem.
\\[2pt]
\textbf{Problem of interest.}
To establish a stability estimate for determining the convection coefficient $A$ and the matrix-valued potential $Q$ appearing in the IBVP \eqref{PDE} from knowledge of the aforementioned partial DN map.

\subsection{Applications and physical significance.}
Systems of reaction-diffusion-convection equations arise naturally in a wide range of scientific and engineering applications, where they are used to describe the transport, diffusion, and interaction of multiple interacting species or physical quantities. Such models are widely used in chemical engineering, ecology, environmental sciences, and biomedical engineering. Representative applications include predator-prey dynamics \cite{doi:10.1142/S0219525904000056,XuWang2018}, multispecies transport in porous media \cite{KolevaVulkov2023,OliveiraBluntBijeljic2019}, river water quality modelling \cite{EfendievEberl2004,Hao2025}, and the transport and interaction of chemical substances in reactive media \cite{XuSamperAyoraManzanoCustodio1999,YuanZhao2025}. In these applications, the convection term describes the underlying transport mechanism, while the matrix-valued reaction term accounts for the coupling and interactions among different species.

From the perspective of inverse problems, recovering these coefficients from boundary measurements provides valuable information about the internal properties of the underlying medium. This has important applications in environmental monitoring, contaminant transport, chemical process control, and biomedical imaging, highlighting the practical relevance of the mathematical model studied in this work.
%%%%%%%%%%%%%%%%%%%%%%%%%%%%%%%%%%%%%%%%%%%%%%%%%%%%%%%%%%
\subsection{Existing works}
In this subsection, we briefly review some existing results related to the inverse problems considered in the present work. While the inverse coefficient problem for a single parabolic equation has been extensively investigated over the past few decades, comparatively little is known for coupled systems of parabolic equations, particularly when multiple lower-order coefficients are recovered simultaneously from partial boundary measurements. We therefore begin with the literature on single parabolic equations before discussing the available results for coupled systems.
 Consider the following IBVP
\begin{equation}\label{Convection diffusion equation}
\left\{
\begin{array}{r l c}
\big( \partial_t - \displaystyle\sum_{j=1}^{n} (\partial_j + A_j(t,x))^2 + q(t,x)\big) u(t,x) =& 0, \quad& (t,x) \in \Omega_T, \\
u(0,x) =& 0, \quad& x \in \Omega, \\
u(t,x) =& f(t,x), \quad& (t,x) \in \partial\Omega_T.
\end{array}
\right.
\end{equation}
\\
The inverse problem of recovering the coefficients in IBVP
\eqref{Convection diffusion equation} has attracted considerable attention under various assumptions on the unknown coefficients and the available measurements. We first summarize the results concerning time-independent coefficients. In the case where the convection term is absent, that is, $A=0$, the uniqueness of determining the potential $q$ from the DN map was established by Isakov \cite{isakov1991completeness}. Subsequently, Choulli \cite{Choulli_Book} derived logarithmic-type stability estimates corresponding to this uniqueness result. The inverse problem of identifying first-order coefficients in a one-dimensional parabolic equation from final-time observations was investigated by Deng et al \cite{deng2008identifying}. In two spatial dimensions, Cheng \emph{et al.} \cite{cheng2002identification} proved the unique determination of the convection coefficient from a single boundary measurement when the zeroth-order term is absent. Later, Bellassoued and Rassas \cite{Bellassoued_2020} studied a stationary convection-diffusion equation and established logarithmic stability estimates for the unknown coefficients from the associated DN map.

Inverse problems involving final-time observations have also been studied extensively. In the absence of a convection term, existence and uniqueness results for recovering coefficients from final-time measurements were obtained in
\cite{Isakov1991,PrilepkoSolovev1987,Yu1991}, where the conclusions depend on the sign of the density coefficient. These results were later extended by Choulli \cite{Choulli1994} to semilinear parabolic equations without imposing sign restrictions on the density coefficient. Furthermore, Choulli and Yamamoto \cite{ChoulliYamamoto1997} established local stability estimates for the recovery of the density coefficient in dimensions $n\leq3$.

We next discuss inverse problems for parabolic equations with space-time dependent coefficients. Choulli and Kian \cite{choulli2012stability} investigated a semilinear parabolic equation and established the recovery of a time-dependent coefficient from Neumann boundary measurements. For the case $A=0$ in IBVP \eqref{Convection diffusion equation},  Choulli and Kian  in \cite{choulli2018logarithmic} derived a logarithmic stability estimate for the potential $q$ using partial DN data measured on a suitable open subset of the boundary. Subsequently, Sahoo and Vashisth \cite{sahoo2019partial} considered the simultaneous recovery of the time-independent convection coefficient and the time-dependent potential, proving uniqueness up to a gauge transformation from partial boundary measurements. Bellassoued and Fraj in  \cite{Bellassoued_2021}  established stability estimates under the more general assumption that the measurements are available on an arbitrary open subset of boundary. Under the divergence-free condition $\nabla_x\cdot A=0$, Senapati \emph{et al.} \cite{senapati2021stability} obtained double-logarithmic and triple-logarithmic stability estimates for the convection coefficient and the potential, respectively. More recently, Mishra \emph{et al.} \cite{MishraPurohitVashisth2025} investigated partial data inverse problems for evolution equations on Riemannian manifolds and established the recovery of time-dependent first- and zeroth-order perturbations.

We also mention several works employing logarithmic Carleman weights to recover coefficients from measurements available on relatively small portions of the boundary. Fan and Duan \cite{fan2021determining} first established the unique determination of the time-dependent potential in the case $A=0$, from boundary data prescribed on a sufficiently small open subset of the boundary. This result was subsequently generalized by Purohit \cite{Purohit2024DeterminingTC}, who proved the unique recovery of both the time-dependent convection coefficient and the potential from the corresponding partial boundary measurements.
For additional results on inverse coefficient problems for parabolic equations, we refer the reader to
\cite{caro2018determination,NakamuraSasayama2013,prilepko1992inverse}
 and the references therein. 

%%%%%%%%%%%%%%%%%%%%%%%%%%%%%
We conclude this subsection by reviewing some representative results on inverse problems for coupled systems with matrix-valued coefficients. In contrast to the extensive literature available for scalar equations, inverse coefficient problems for coupled parabolic and hyperbolic systems have received comparatively less attention and remain an active area of research.
An early contribution in this direction was made by Bube and Burridge \cite{bube1983one}, who investigated the recovery of $2\times2$ matrix-valued coefficients in a coupled first-order parabolic as well as hyperbolic system using one-dimensional observations. Later, Cristofol \emph{et al.} \cite{cristofol2006inverse} studied a coupled reaction-diffusion system and established simultaneous stability results for one of the coefficients together with the initial condition by combining a global Carleman estimate with observations available on a subdomain. Subsequently, Allal \emph{et al.} \cite{allal2020lipschitz} considered an inverse source problem for a cascade system of coupled degenerate parabolic equations and proved uniqueness and stability for the recovery of the source terms from observations of a single component over a time interval together with final-time measurements of all components.
Inverse coefficient and source problems for coupled parabolic systems have since been investigated in a variety of settings. We refer to
\cite{bellassoued2017carleman, montoya2024source,ren2026uniqueness}
for representative developments, and to
\cite{ainseba2015stability,benabdallah2009inverse,cristofol2013inverse, cristofol2006inverse,cristofol2008identification, dinakar2017identification}
for further results on coupled reaction-diffusion equations.

We also refer to \cite{BellassouedChoulli2008, BukhgeimUhlmann2002, HeckWang2006, KenigSjostrandUhlmann2007, Kian2016SIAM, Kian2016, Krishnan2020} for earlier contributions concerning uniqueness and stability estimates for other classes of PDEs. 
Inverse problems for coupled systems associated with different types of PDEs have likewise attracted substantial attention. In particular, uniqueness and stability results derived from boundary measurements have been established in \cite{ben2025stable, bhardwaj2026inverse, DouYamamoto2019, filippas2025recoverymatrixvaluedpotential, khanfer2019inverse, Kumar2024, LiuTriggiani2011, Mishra10122021, ZhuDou2023} under a variety of geometric settings and measurement configurations.

Motivated by these developments, the present article investigates the inverse problem of simultaneously recovering a time-dependent matrix-valued convection coefficient and a matrix-valued potential in a coupled reaction-diffusion-convection system from partial DN data. To the best of our knowledge, the stability results established here are new for this class of coupled parabolic systems.

\subsection{Statement of the main results}
 In this subsection, we present the main results of the article, which provide stability estimates for the time-dependent convection coefficient and the matrix-valued potential from partial boundary measurements.
 
\begin{theorem}\label{main theorem}{(Stability result)} 
Let $\Omega \subset \mathbb{R}^n$ $(n \geq 2)$ be an open, bounded, and simply connected domain with smooth boundary $\partial \Omega$. Assume that $T> \operatorname{diam}(\Omega)$ and that, for $i=1,2$, the pairs $(A_{(i)},Q_{(i)})\in \mathcal{A}(m)$ satisfy
\begin{align}\label{extra condition imposed on A}
\lVert\overrightarrow{A}^j_{(2)}-\overrightarrow{A}^\ell_{(2)}\rVert_{L^2(\Omega_T)} \leq \lVert A_{(1)}-A_{(2)}\rVert_{\mathbf{L}^2(\Omega_T)} \quad \text{for all } j,\ \ell\in \{1,2,\dots,k\}.
\end{align}
Moreover, let $\Lambda^{\sharp}_i$ denote the partial DN map associated with the operator $\mathcal{L}_{A_{(i)},Q_{(i)}}$, as defined in Equation \eqref{partial dn map}. Then there exist positive constants $C(\Omega_T,m)$ and $a_i(\Omega,m)$ $(1\leq i\leq 4)$ such that the following stability estimates hold:
\begin{align}\label{stable estimate of A and Q}
\begin{aligned}
    \lVert A_{(1)}-A_{(2)}\rVert_{\mathbf{L}^2(\Omega_T)}
    &\leq C\left( \lVert \Lambda^{\sharp}_1- \Lambda^{\sharp}_2\rVert^{a_1}+\bigl\lvert \log \bigl\lvert \log  \lVert \Lambda^{\sharp}_1- \Lambda^{\sharp}_2\rVert\bigr\rvert \bigr\rvert^{-a_2}\right), \quad\text{and}\\
    \lVert Q_{(1)}-Q_{(2)}\rVert_{\mathbf{L}^2(\Omega_T)}
    &\leq C\left( \lVert \Lambda^{\sharp}_1- \Lambda^{\sharp}_2\rVert^{a_3}+\bigl\lvert \log\bigl\lvert\log \bigl\lvert \log  \lVert \Lambda^{\sharp}_1- \Lambda^{\sharp}_2\rVert\bigr\rvert \bigr\rvert\bigr\rvert^{-a_4}\right),
    \end{aligned}
\end{align}
provided that $A_{(1)}=A_{(2)}$ on $\partial \Omega_T$ and 
\(
    \nabla_x\cdot \overrightarrow{A}^i_{(1)}=\nabla_x\cdot \overrightarrow{A}^i_{(2)} 
\) in $\Omega_T$,
 for any $i\in \{1,\cdots,k\}$.
\end{theorem}
\begin{remark}
\begin{enumerate}
\item We impose an extra condition on $A$ given by Equation \eqref{extra condition imposed on A} required to establish the stability estimate for the matrix-valued potential (see Equation \eqref{condition on matrix A required for estimating Q}).
    \item 
    If one considers the operator $\mathcal{L}_{A,Q_{(i)}}$, or the operator $\mathcal{L}_{A_{(i)},Q_{(i)}}$ with $A_{(i)}$ a scalar matrix, then Equation \eqref{extra condition imposed on A} is satisfied in a straightforward manner. Moreover, if we define
\[
A_{(2)}(t,x) := \operatorname{diag}\bigl(\overrightarrow{A}^1(t,x),0,\ldots,0\bigr)
\quad\text{and}\quad
A_{(1)}(t,x) := \operatorname{diag}\bigl(0,\overrightarrow{A}^2(t,x),\ldots,\overrightarrow{A}^k(t,x)\bigr),
\]
then Equation \eqref{extra condition imposed on A} is again automatically fulfilled.  Consequently, there are multiple choices for the convection coefficients which satisfy \eqref{extra condition imposed on A}.

    \item In the scalar (single-equation) framework, as previously discussed, one can in general guarantee uniqueness in the recovery of the coefficients only modulo a gauge transformation; see, for example, \cite{Bellassoued_2021,caro2018determination,  MishraPurohitVashisth2025, sahoo2019partial} and the references therein. In contrast, under the divergence-free constraint, the coefficients become fully identifiable. This observation provides the primary motivation for imposing the divergence-free condition in our analysis.
\end{enumerate}
\end{remark}
Finally, as a corollary of Theorem \ref{main theorem}, we have the following uniqueness result, which is established without using the assumption \eqref{extra condition imposed on A} on convection term $A$.
\begin{corollary}\label{Uniqueness for main PDE}{(Uniqueness result)}
   Let $\Omega \subset \mathbb{R}^n$ $(n \geq 2)$ be an open, bounded, and simply connected domain with smooth boundary $\partial \Omega$. Assume that $T> \operatorname{diam}(\Omega)$ and that, for $i=1,2$, the pairs $(A_{(i)},Q_{(i)})\in \mathcal{A}(m)$. Moreover, let $\Lambda^{\sharp}_i$ denote the partial DN map associated with the operator $\mathcal{L}_{A_{(i)},Q_{(i)}}$, as defined in Equation \eqref{partial dn map}. Then 
$\Lambda^\sharp_1=\Lambda^\sharp_2$ implies $A_{(1)}=A_{(2)}$ in $\O_T$, provided that $A_{(1)}=A_{(2)}$ on $\partial \Omega_T$ and 
    $\nabla_x\cdot \overrightarrow{A}^i_{(1)}=\nabla_x\cdot \overrightarrow{A}^i_{(2)} \quad$ 
in $\Omega_T$, for each $1\leq i \leq k$.
\end{corollary}
\noindent The proof of Theorem~\ref{main theorem} is based on the derivation of a suitable integral identity involving solutions of the IBVP \eqref{PDE} and its adjoint problem. Combining this identity with Carleman estimates and appropriately constructed geometric optics solutions, we establish a double logarithmic ($\log$-$\log$) stability estimate for the time-dependent convection coefficient from partial boundary measurements. This result is then utilized to derive a triple logarithmic ($\log$-$\log$-$\log$) stability estimate for the matrix-valued potential, thereby completing the recovery of the lower-order coefficients.

 \subsection{Outline} 
The remainder of the paper is organized as follows. In Section~\ref{section 2}, we establish the existence and uniqueness of solutions to a more general IBVP given by \eqref{more general forward pde}. Section~\ref{Section 3} is devoted to the derivation of interior and boundary Carleman estimates; in particular, we prove an interior Carleman estimate in negative-order Sobolev spaces. In Section~\ref{Construction of solutions}, we construct geometric optics solutions for the operator $\mathcal{L}_{A,Q}$ and its adjoint operator $\mathcal{L}_{A,Q}^{*}$. Finally, Section~\ref{Proof of main theorem} is dedicated to the proof of the main theorem and the derivation of the corresponding stability estimates.

\section{Well-posedness of forward problem}\label{section 2}
In this section, we establish the well-posedness of a more general IBVP, namely the system \eqref{more general forward pde}. This result guarantees that the DN map given in Equation~\eqref{DN map} is well defined. We begin by introducing the notation and preliminary definitions used throughout this article.
Let
$M(t,x):=\bigl(\overrightarrow{M}_{ij}(t,x)\bigr)_{1\le i,j\le k}$
be a $k\times k$ matrix whose entries are vector-valued functions in $\mathbb{R}^n$, where
\(
   \overrightarrow{M}_{ij}(t,x)
:=
\bigl(
M_{ij}^{1}(t,x),
M_{ij}^{2}(t,x),
\ldots,
M_{ij}^{n}(t,x)
\bigr),
1\le i,j\le k. 
\)
For such a matrix-valued function $M$, we define
$\nabla_x\cdot M
:=
\bigl(
\nabla_x\cdot\overrightarrow{M}_{ij}
\bigr)_{1\le i,j\le k},$
and, for a scalar-valued function $v$,
$M\cdot\nabla_x v
:=
\bigl(
\overrightarrow{M}_{ij}\cdot\nabla_x v
\bigr)_{1\le i,j\le k}$.
Moreover, for a vector-valued function
$\overrightarrow{u}
=
(u_1,\ldots,u_k)^T,$
we define
\begin{align*}
    M\cdot\nabla_x\overrightarrow{u}
:=
\left(
\displaystyle\sum_{j=1}^{k}
\overrightarrow{M}_{1j}\cdot\nabla_x u_j, 
\sum_{j=1}^{k}
\overrightarrow{M}_{2j}\cdot\nabla_x u_j,
\cdots,
\displaystyle\sum_{j=1}^{k}
\overrightarrow{M}_{kj}\cdot\nabla_x u_j
\right)^T.
\end{align*}
Furthermore, for an open set $\Omega\subset\mathbb{R}^n$ and $m\in\mathbb{N}$, we denote by $H^m(\Omega)$ and $W^{m,\infty}(\Omega)$ the usual Sobolev spaces, equipped with their standard norms
\begin{equation*} \displaystyle\|u\|_{H^m(\Omega)}: =\left(\sum_{\lvert\alpha\rvert\leq m} \lVert\partial ^{\alpha} u\rVert^2_{L^2(\Omega)}\right)^{1/2}\,\text{ and }\quad \,\|u\|_{W^{m,\infty}(\Omega)}:=\max_{\lvert\alpha\rvert\leq m}\|\partial^{\alpha}u\|_{L^{\infty}(\Omega)},
\end{equation*} 
respectively, where, for a multi-index
$\alpha=(\alpha_1,\alpha_2,\ldots,\alpha_n)\in(\mathbb{N}\cup\{0\})^n$,
we use the notation
$\partial^{\alpha}
=
\partial_{x_1}^{\alpha_1}
\partial_{x_2}^{\alpha_2}
\cdots
\partial_{x_n}^{\alpha_n}$
to denote the partial differential operator of order
$|\alpha|
=
\alpha_1+\cdots+\alpha_n$.
We also introduce the vector-valued Sobolev spaces
$\mathbf{H}^m(\Omega)$ and
$\mathbf{W}^{m,\infty}(\Omega)$, endowed with the norms
\begin{align*}
\displaystyle\|\overrightarrow{u}\|_{\mathbf{H}^m(\Omega)}=\left(\sum_{i=1}^k\|u_i\|^2_{H^m(\Omega)}\right)^{1/2} \text{ and } \quad \,\|\overrightarrow{u}\|_{\mathbf{W}^{m,\infty}(\Omega)}=\max_{1 \leq i \leq k}\|u_i\|_{W^{m,\infty}(\Omega)}.
\end{align*}
The Sobolev space $H_0^{m}(\Omega)$ denotes the completion of $C_c^\infty(\Omega)$ with respect to  $\lVert \cdot\rVert_{H^m(\Omega)}$ norm and $H^{-m}(\O)$ denotes the dual space of $H_0^m(\O)$. 
In what follows, we fix $M(t,x):= \left(\overrightarrow{M}_{ij}(t,x)\right)_{{1\leq i,j\leq k}}$ with each $\overrightarrow{M}_{ij} :=\left(M_{ij}^1, M_{ij}^2, \cdots , M_{ij}^n\right) \in \textbf{L}^{\infty}(\Omega_T),$ $N(t,x):= \left(\overrightarrow{N}_{ij}(t,x)\right)_{{1\leq i,j\leq k}}$ with each $\overrightarrow{N}_{ij}:=\left(N_{ij}^1, N_{ij}^2, \cdots , N_{ij}^n\right) \in \mathbf{W}^{1,\infty}(\Omega_T)$, and, $P(t,x) := \left(p_{ij}(t,x)\right)_{{1\leq i,j\leq k}}$ with each $p_{ij} \in L^\infty(\Omega_T)$. 
We consider the following general linear reaction-diffusion-convection operator on the cylindrical domain $\Omega_T$, denoted by $\mathcal{K}_{M,N,P}$, and defined by
\begin{align}
 \label{Operator K_{M,N,P}}
   \displaystyle     \mathcal{K}_{M,N,P}\overrightarrow{u} &:= \begin{bmatrix}
 \partial_t u_1- \Delta u_1+  \displaystyle \sum_{j=1}^{k}\overrightarrow{M}_{1j}\cdot \nabla_x u_j+ \sum_{j=1}^{k} (\nabla_x\cdot \overrightarrow{N}_{1j})u_j +\sum_{j=1}^k p_{1j} u_j
\\ 
 \partial_t u_2- \Delta u_2+   \displaystyle\sum_{j=1}^{k}\overrightarrow{M}_{2j}\cdot \nabla_x u_j+ \sum_{j=1}^{k} (\nabla_x\cdot \overrightarrow{N}_{2j})u_j +\sum_{j=1}^k p_{2j} u_j\\
\vdots\\
 \partial_t u_k- \Delta u_k+   \displaystyle\sum_{j=1}^{k}\overrightarrow{M}_{kj}\cdot \nabla_x u_j+ \sum_{j=1}^{k} (\nabla_x\cdot \overrightarrow{N}_{kj})u_j +\sum_{j=1}^k p_{kj} u_j
\end{bmatrix}.
\end{align}
We now prove the well-posedness of the operator $\mathcal{K}_{M,N,P}$ in the following theorem.
\begin{theorem}\label{existence in forward pde}
Assume that $M\in \textbf{L}^{\infty}(\O_T),\ N \in \mathbf{W}^{1,\infty}(\O_T)$ and $P\in \textbf{L}^{\infty}(\O_T)$. Then for $\overrightarrow{F}\in\mathbf{H}^{\frac{3}{2},\frac{3}{4}}(\partial \O_T)$
the IBVP
\begin{align}\label{more general forward pde}
 \left\{  \begin{array}{r l c}
\mathcal{K}_{M,N,P}\overrightarrow{u}(t,x) &= \overrightarrow{0},\quad& (t,x) \in \Omega_T,\\
\overrightarrow{u}(0,x)&= \overrightarrow{0}, \quad& x \in \Omega  ,\\
\overrightarrow{u}(t,x)&= \overrightarrow{F}(t,x),\quad& (t,x) \in \partial \Omega_T,
\end{array} \right.
\end{align}
admits a unique solution $\overrightarrow{u} \in \mathbf{H}^{2,1}(\Omega_T)$ satisfying
   \begin{align}\label{main estimate in forward pde 1.1}
       \lVert \overrightarrow{u}\rVert_{\mathbf{H}^{2,1}(\Omega_T)}  \leq C \lVert\overrightarrow{F}\rVert_{{}_{0}\mathbf{H}^{\frac{3}{2},\frac{3}{4}}(\partial \O_T)},
   \end{align}
   where $C$ depends on $\Omega,\ T$ and  the coefficients of operator $\mathcal{K}_{M,N,P}$. In particular,  if we define
\begin{align}\label{Value of M,N,P in term of A}
\begin{aligned}
 \overrightarrow{M}_{ij}:=\left\{  \begin{array}{c l }
-2\overrightarrow{A}^{i} \quad& \text{ if }\quad i= j\\0 \quad& \text{ if }\quad i\neq j
\end{array} \right. ,\quad  \overrightarrow{N}_{ij}:=\left\{  \begin{array}{c l }
-\overrightarrow{A}^{i} \quad& \text{ if }\quad i= j\\0 \quad& \text{ if }\quad i\neq j
\end{array} \right. \text{ and}\\  {p}_{ij}:=\left\{  \begin{array}{c l }
q_{ii}-\lvert\overrightarrow{A}^{i}\rvert^2 \quad& \text{ if }\quad i= j\\q_{ij} \quad& \text{ if }\quad i\neq j
\end{array} \right.,
\end{aligned}
\end{align}
where $(A,Q) \in \mathcal{A}_m$ then the operator $\mathcal{K}_{M,N,P}$ reduces to $\mathcal{L}_{A,Q}$. Thus  there exists a unique solution $\overrightarrow{u} \in \mathbf{H}^{2,1}(\Omega_T)$ to IBVP \eqref{PDE}  which satisfies
   \begin{align}\label{main estimate in forward pde}
   \lVert \overrightarrow{u}\rVert_{\mathbf{H}^{2,1}(\Omega_T)}  \leq C \lVert\overrightarrow{F}\rVert_{{}_{0}\mathbf{H}^{\frac{3}{2},\frac{3}{4}}(\partial \O_T)},
   \end{align}   where $C$ depends on $\Omega,\ T$ and $m$.
\end{theorem}
\begin{proof}
The proof is based on the classical results presented in
\cite{lionsmagenes1972vol2}. By the trace theorem, the trace operator
$\Gamma:\mathbf{H}^{2,1}(\Omega_T)\longrightarrow
\mathbf{H}^{\frac32,\frac34}(\partial\Omega_T)$
is a continuous surjection map. Consequently, for every
$\overrightarrow{F}\in
{}_{0}\mathbf{H}^{\frac32,\frac34}(\partial\Omega_T)$,
there exists a function
$\overrightarrow{v}\in
\mathbf{H}^{2,1}(\Omega_T)$
such that
$\Gamma(\overrightarrow{v})
=
\overrightarrow{F}$.
Moreover, there exists a constant $C>0$, depending only on $\Omega_T$, such that
\begin{align}\label{eq; estimate trace map}
\|\overrightarrow{v}\|_{\mathbf{H}^{2,1}(\Omega_T)}
\le
C
\|\overrightarrow{F}\|_{{}_{0}\mathbf{H}^{\frac32,\frac34}(\partial\Omega_T)}.
\end{align}
We emphasize that the lifting function $\overrightarrow{v}$ is, in general,
not unique. Fix any such function $\overrightarrow{v}$, and define
$\overrightarrow{u}
:=
\overrightarrow{w}
+
\overrightarrow{v}$.
Substituting this decomposition into IBVP
\eqref{more general forward pde}, we find that
$\overrightarrow{w}$ satisfies the following IBVP:
 \begin{align}\label{for which we required bilinear form}
     \left\{   \begin{array}{r l c}
\displaystyle \left(\partial_t-\Delta_x+ M\cdot \nabla_x+ \nabla_x \cdot N+P\right)\overrightarrow{w}(t,x)&= F(\overrightarrow{v}), \ & (t,x) \in \Omega_T,\\
\overrightarrow{w}(0,x)&= \overrightarrow{0}, \ & x \in \Omega  ,\\
\overrightarrow{w}(t,x)&= \overrightarrow{0},\ & (t,x) \in \partial \Omega_T, 
\end{array}\right.
    \end{align}
    where $F(\overrightarrow{v})
:=
-
\left(
\partial_t
-\Delta_x
+
M\cdot\nabla_x
+
\nabla_x\cdot N
+
P
\right)
\overrightarrow{v}$.
    First, we prove that Equation~\eqref{for which we required bilinear form} admits a unique solution $\overrightarrow{w}\in\mathbf{H}^{2,1}(\Omega_T)$. We then show that the uniqueness of the solution to Equation~\eqref{for which we required bilinear form} implies the uniqueness of the solution to the IBVP~\eqref{more general forward pde}. Consequently, establishing uniqueness for the former problem is sufficient to deduce the uniqueness of the latter. 
Since $\overrightarrow{v}\in \mathbf{H}^{2,1}(\Omega_T)$, it follows from the regularity of $\overrightarrow{v}$ and the assumptions on the coefficients $M,N$ and $P$ that
\begin{align*}
    F(\overrightarrow{v})
:=
-
\left(
\partial_t
-\Delta_x
+
M\cdot\nabla_x
+
\nabla_x\cdot N
+
P
\right)
\overrightarrow{v}
\in
L^{2}\bigl(0,T;\mathbf{L}^{2}(\Omega)\bigr).
\end{align*}
Moreover, by the boundedness of the coefficients $M$, $N$, and $P$, we obtain
\begin{align}\label{right hand bound in inhomogenous operator}
\begin{aligned}
&
\left\|
-\left(
\partial_t
-\Delta_x
+
M\cdot\nabla_x
+
(\nabla_x\cdot N)
+
P
\right)
\overrightarrow{v}
\right\|_{\mathbf{L}^{2}(\Omega_T)}
\\
&\qquad
\le
C\Bigl(
\|\partial_t\overrightarrow{v}\|_{\mathbf{L}^{2}(\Omega_T)}
+
\|\Delta_x\overrightarrow{v}\|_{\mathbf{L}^{2}(\Omega_T)}
+
C_1\|\nabla_x\overrightarrow{v}\|_{\mathbf{L}^{2}(\Omega_T)}
+
(C_2+C_3)
\|\overrightarrow{v}\|_{\mathbf{L}^{2}(\Omega_T)}
\Bigr)
\\
&\qquad
\le
C
\|\overrightarrow{v}\|_{\mathbf{H}^{2,1}(\Omega_T)},
\end{aligned}
\end{align}
where
\begin{align*}
    C_1
=
\max_{1\le i,j\le k}
\|\overrightarrow{M}_{i,j}\|_{\mathbf{L}^{\infty}(\Omega_T)},
\quad
C_2
=
\max_{1\le i,j\le k}
\|\overrightarrow{N}_{i,j}\|_{\mathbf{W}^{1,\infty}(\Omega_T)}, \quad\text{ and }\quad
    C_3
=
\max_{1\le i,j\le k}
\|p_{i,j}\|_{L^{\infty}(\Omega_T)}.
\end{align*}
Motivated by the estimate \eqref{right hand bound in inhomogenous operator}, we now introduce the sesquilinear form
$B(\cdot,\cdot;t)$ associated with the IBVP
\eqref{for which we required bilinear form}. This form will play a key role in establishing the existence and uniqueness of weak solutions. It is defined on the Hilbert space
$\mathbf{H}_0^{1}(\Omega)$ by
    \begin{align}\label{bilinear operator}
\begin{aligned}
B(\overrightarrow{\mathrm{u}}_1,\overrightarrow{\mathrm{u}}_2;t)
:=
\int_{\Omega}
\nabla_x\overrightarrow{\mathrm{u}}_1
\cdot
\overline{\nabla_x\overrightarrow{\mathrm{u}}_2}\,dx
+
\int_{\Omega}
\left(M\cdot\nabla_x\overrightarrow{\mathrm{u}}_1\right)
\cdot
\overline{\overrightarrow{\mathrm{u}}_2}\,dx
+
\int_{\Omega}
(\nabla_x\cdot N)\overrightarrow{\mathrm{u}}_1
\cdot
\overline{\overrightarrow{\mathrm{u}}_2}\,dx
+
\int_{\Omega}
P\overrightarrow{\mathrm{u}}_1
\cdot
\overline{\overrightarrow{\mathrm{u}}_2}\,dx,
\end{aligned}
\end{align}
for every
$\overrightarrow{\mathrm{u}}_1,
\overrightarrow{\mathrm{u}}_2
\in
\mathbf{H}_0^1(\Omega)$
and $t\in(0,T)$.
By H\"older's inequality, together with the boundedness of the coefficients
$M$, $N$, and $P$, the form
$B(\cdot,\cdot;t)$ is a continuous sesquilinear form on
$\mathbf{H}_0^1(\Omega)\times \mathbf{H}_0^1(\Omega)$. More precisely,
\begin{align}\label{continuity bilinear form}
\left|
B(\overrightarrow{\mathrm{u}}_1,\overrightarrow{\mathrm{u}}_2;t)
\right|
\le
C
\|
\overrightarrow{\mathrm{u}}_1
\|_{\mathbf{H}^1(\Omega)}
\|
\overrightarrow{\mathrm{u}}_2
\|_{\mathbf{H}^1(\Omega)},
\end{align}
for every
$\overrightarrow{\mathrm{u}}_1,
\overrightarrow{\mathrm{u}}_2
\in
\mathbf{H}_0^1(\Omega)$
,
where the constant $C>0$ depends only on
$\Omega$, $M$, $N$, and $P$.
We next verify that $B(\cdot,\cdot;t)$ satisfies a G\aa rding inequality.
Applying H\"older's inequality followed by Young's inequality, we obtain
\begin{align}\label{Garding estimate}
\begin{aligned}
\mathfrak{Re}
B(\overrightarrow{\mathrm{u}}_1,\overrightarrow{\mathrm{u}}_1;t)
\ge\;&
\|
\nabla_x\overrightarrow{\mathrm{u}}_1
\|_{\mathbf{L}^{2}(\Omega)}^{2}
-
2(C_2+C_3)
\|
\overrightarrow{\mathrm{u}}_1
\|_{\mathbf{L}^{2}(\Omega)}^{2}
-
2C_1
\|
\nabla_x\overrightarrow{\mathrm{u}}_1
\|_{\mathbf{L}^{2}(\Omega)}
\|
\overrightarrow{\mathrm{u}}_1
\|_{\mathbf{L}^{2}(\Omega)}
\\
\ge\;&
(1-\varepsilon C_1)
\|
\nabla_x\overrightarrow{\mathrm{u}}_1
\|_{\mathbf{L}^{2}(\Omega)}^{2}
-
2\left(
C_2+C_3+\frac{C_1}{2\varepsilon}
\right)
\|
\overrightarrow{\mathrm{u}}_1
\|_{\mathbf{L}^{2}(\Omega)}^{2},
\end{aligned}
\end{align}
for every
$\overrightarrow{\mathrm{u}}_1
\in
\mathbf{H}_0^1(\Omega)$
and $t\in(0,T)$.
Choosing
$\varepsilon=\frac{1}{2C_1}$,
the estimate \eqref{Garding estimate} simplifies to
\begin{align}\label{Garding inequality}
\mathfrak{Re}
B(\overrightarrow{\mathrm{u}}_1,\overrightarrow{\mathrm{u}}_1;t)
+
\gamma
\|
\overrightarrow{\mathrm{u}}_1
\|_{\mathbf{L}^{2}(\Omega)}^{2}
\ge
\frac{1}{2}
\|
\overrightarrow{\mathrm{u}}_1
\|_{\mathbf{H}_0^{1}(\Omega)}^{2},
\end{align}
where \( \gamma
=
2(C_2+C_3+C_1^{2})\ge0.\)
Hence, the sesquilinear form $B(\cdot,\cdot;t)$ is continuous and satisfies the G\aa rding inequality uniformly with respect to $t\in(0,T)$. Therefore, an application of \cite[Theorem~6.1, Chapter~4]{lionsmagenes1972vol2} implies that the IBVP \eqref{for which we required bilinear form} admits a unique solution
$\overrightarrow{w}\in\mathbf{H}^{2,1}(\Omega_T).$
Moreover, there exists a constant $C>0$, depending only on
$\Omega_T$, $M$, $N$, and $P$, such that
\begin{align}\label{estimate w}
\begin{aligned}
\|\overrightarrow{w}\|_{\mathbf{H}^{2,1}(\Omega_T)}
&\le
C
\left\|
-\left(
\partial_t
-\Delta_x
+
M\cdot\nabla_x
+
(\nabla_x\cdot N)
+
P
\right)
\overrightarrow{v}
\right\|_{\mathbf{L}^{2}(\Omega_T)}
\\
&\le
C
\|\overrightarrow{v}\|_{\mathbf{H}^{2,1}(\Omega_T)}
\le
C
\|\overrightarrow{F}\|_{{}_{0}\mathbf{H}^{\frac32,\frac34}(\partial\Omega_T)},
\end{aligned}
\end{align}
where the last inequality follows from the trace estimate
\eqref{eq; estimate trace map}.
We are now in a position to establish the existence of a solution to IBVP
\eqref{more general forward pde}. Indeed, defining
\(
    \overrightarrow{u}
:=
\overrightarrow{w}
+
\overrightarrow{v},
\)
where
$\overrightarrow{w},
\overrightarrow{v}
\in
\mathbf{H}^{2,1}(\Omega_T),$
it follows immediately that
$\overrightarrow{u}
\in
\mathbf{H}^{2,1}(\Omega_T)$
satisfies the IBVP
\eqref{more general forward pde}.
It remains to establish the uniqueness of the solution. More precisely, we
show that the solution is independent of the choice of the lifting function
$\overrightarrow{v}$.
Suppose, on the contrary, that
\(
    \overrightarrow{u}_1
=
\overrightarrow{w}_1
+
\overrightarrow{v}_1
\) and \(
\overrightarrow{u}_2
=
\overrightarrow{w}_2
+
\overrightarrow{v}_2
\)
are two solutions of IBVP \eqref{more general forward pde}, where
$\overrightarrow{v}_1,\,
\overrightarrow{v}_2
\in
\mathbf{H}^{2,1}(\Omega_T)$
satisfy
$\Gamma(\overrightarrow{v}_1)
=
\Gamma(\overrightarrow{v}_2)
=
\overrightarrow{F}$,
together with the estimate
\eqref{eq; estimate trace map}. Let
$\overrightarrow{w}_1,
\overrightarrow{w}_2
\in
\mathbf{H}^{2,1}(\Omega_T)$
be the corresponding solutions of IBVP
\eqref{for which we required bilinear form}
associated with
$\overrightarrow{v}_1$ and $\overrightarrow{v}_2$, respectively.
Define
$\overrightarrow{U}
:=
\overrightarrow{u}_1
-
\overrightarrow{u}_2$.
Then $\overrightarrow{U}$ satisfies
\begin{align}
\label{homogeneous equation uniqueness}
\left\{
\begin{aligned}
\mathcal{K}_{M,N,P}\overrightarrow{U}
&=
\overrightarrow{0},
&&
(t,x)\in\Omega_T,
\\
\overrightarrow{U}(0,x)
&=
\overrightarrow{0},
&&
x\in\Omega,
\\
\overrightarrow{U}(t,x)
&=
\overrightarrow{0},
&&
(t,x)\in\partial\Omega_T.
\end{aligned}
\right.
\end{align}
The well-posedness result established above implies that the homogeneous
problem \eqref{homogeneous equation uniqueness} admits only the trivial
solution
which immediately yields
$\overrightarrow{u}_1
=
\overrightarrow{u}_2$.
Therefore, the solution of IBVP
\eqref{more general forward pde}
is independent of the choice of the lifting function
$\overrightarrow{v}$.
Consequently, the IBVP~\eqref{more general forward pde}
admits a unique solution
$\overrightarrow{u}
\in
\mathbf{H}^{2,1}(\Omega_T)$.
Finally, taking
$M$, $N$, and $P$
as defined in Equation~\eqref{Value of M,N,P in term of A}
and assuming
$(A,Q)\in\mathcal{A}_m$, following Sobolev embedding theorem, 
we conclude that the IBVP
\eqref{PDE}
admits a unique solution satisfying the stability estimate
\eqref{main estimate in forward pde}.
    \end{proof}
\section{Carleman estimate}\label{Section 3} 
In this section, we derive a Carleman estimate that will play an important role in the proof of the main theorem. The section is structured into two subsections. The first subsection is devoted to boundary Carleman estimates for the operator $\mathcal{K}_{M,N,P}$, defined by
\begin{align}\label{operator K_M,N,P}
\mathcal{K}_{M,N,P}\overrightarrow{u}
:=
\partial_t\overrightarrow{u}
-\Delta_x\overrightarrow{u}
+
M(t,x)\cdot\nabla_x\overrightarrow{u}
+
\bigl(\nabla_x\cdot N(t,x)\bigr)\overrightarrow{u}
+
P(t,x)\overrightarrow{u},
\end{align}
where the matrix-valued coefficients $M$, $N$, and $P$ are as introduced in Section~\ref{section 2}, whereas the second subsection focuses on deriving interior Carleman estimates in negative-order Sobolev spaces for the operator $\mathcal{K}_{M,N,P}^*$, defined in \eqref{Operator K^{*}_{A,Q}}.
The boundary Carleman estimate will be employed to control the unknown boundary contributions arising in the integral identity, by expressing them in terms of the available boundary measurements. On the other hand, the interior Carleman estimate will play a crucial role in constructing the special solutions, commonly referred to as geometric-optics solutions. Our approach follows the framework developed in \cite{choulli2018logarithmic, FerreiraKenigSaloUhlmann2009, KenigSjostrandUhlmann2007, MishraPurohitVashisth2025,sahoo2019partial,senapati2021stability}, where analogous Carleman estimates were established for scalar elliptic and parabolic equations. The main novelty here is the extension of these estimates to a coupled system of parabolic equations with matrix-valued coefficients.
To derive the desired Carleman estimates, we introduce the linear Carleman weight function
\begin{align}\label{carleman weight}
\phi(t,x)
=
\lambda^2 t+\lambda\,x\cdot\omega,
\end{align}
where $\lambda>0$ is a large parameter and $\o \in \mathbb{S}^{n-1}$.
\par
For scalar parabolic equations, Carleman estimates associated with the weight \eqref{carleman weight} have been established in
\cite{choulli2018logarithmic, MishraPurohitVashisth2025,
sahoo2019partial}. In the present work, we prove analogous Carleman estimates for the coupled parabolic system governed by the operator $\mathcal{K}_{M,N,P}$. 

%%%%%%%%%%%%%%%%%%%
\subsection{Boundary Carleman estimate} 

In this subsection, we derive a weighted $H^1$--$L^2$ boundary Carleman estimate for the operator $\mathcal{K}_{M,N,P}$. The result is formulated in the following theorem.

\begin{theorem}[Boundary Carleman estimate]\label{BCE}
Let $\mathcal{K}_{M,N,P}$ and $\phi$ be given by \eqref{operator K_M,N,P} and \eqref{carleman weight}, respectively. Then, for any $\overrightarrow{u}\in {\textbf{C}}^{\infty}(\overline{\Omega_T})$ such that $\overrightarrow{u}(t,x)|_{\partial\Omega_T} = \overrightarrow{0}$ and $\overrightarrow{u}(0,x) = \overrightarrow{0}$ for $x\in\O$, there exists a constant $C>0$, depending only on $\O, T, M, N$ and $P$, for which the estimate
\begin{align}
   &\lambda \left\langle (\o\cdot \nu(x))e^{-\phi}\partial_\nu \overrightarrow{u}, e^{-\phi}\partial_\nu\overrightarrow{u}\right\rangle_{(0,T)\times B(\o)} 
   +\left(
\lambda^2\|e^{-\phi}\overrightarrow{u}\|^2_{\textbf{L}^2(\Omega_T)} + \|e^{-\phi}(\nabla_x \cdot\overrightarrow{u})\|^2_{\textbf{L}^2(\Omega_T)}\right)\label{boundary carleman estimate1}\\
&\quad +\lambda
\|e^{-\phi(T,x)}\overrightarrow{u}(T,x)\|^2_{\textbf{L}^2(\Omega)} \leq C\|e^{-\phi}\mathcal{K}_{M,N,P}\overrightarrow{u}\|^2_{\textbf{L}^2(\Omega_T)} +C\lambda \left\langle (\o\cdot \nu(x))e^{-\phi}\partial_\nu \overrightarrow{u},e^{-\phi}\partial_\nu\overrightarrow{u}\right\rangle_{(0,T)\times F(\o)} \nonumber
\end{align}
holds for all $\lambda$ sufficiently large.
\end{theorem} 

\begin{remark}
    \begin{enumerate}
    \item In particular, consider $(A,Q)\in \mathcal{A}_m$ and let $\mathcal{L}_{A,Q}$ be defined by \eqref{Operator L_{A,Q}}. Then, for any $\overrightarrow{u}\in {\textbf{C}}^{\infty}(\overline{\Omega_T})$ satisfying $\overrightarrow{u}(t,x)|_{\partial\Omega_T} = \overrightarrow{0}$ and $\overrightarrow{u}(0,x) = \overrightarrow{0}$ for $x\in\O$, there exists a constant $C>0$, depending only on $\O, T$ and $m$, such that
\begin{align}
   &\lambda \left\langle (\o\cdot \nu(x))e^{-\phi}\partial_\nu \overrightarrow{u},e^{-\phi}\partial_\nu\overrightarrow{u}\right\rangle_{(0,T)\times B(\o)} 
   +\left(
\lambda^2\|e^{-\phi}\overrightarrow{u}\|^2_{\textbf{L}^2(\Omega_T)} + \|e^{-\phi}(\nabla_x \cdot\overrightarrow{u})\|^2_{\textbf{L}^2(\Omega_T)}\right)\label{boundary carleman estimate2}\\
&\quad +\lambda
\|e^{-\phi(T,x)}\overrightarrow{u}(T,x)\|^2_{\textbf{L}^2(\Omega)} \leq C\|e^{-\phi}\mathcal{L}_{A,Q}\overrightarrow{u}\|^2_{\textbf{L}^2(\Omega_T)} +C\lambda \left\langle (\o\cdot \nu(x))e^{-\phi}\partial_\nu \overrightarrow{u},e^{-\phi}\partial_\nu\overrightarrow{u}\right\rangle_{(0,T)\times F(\o)} \nonumber
\end{align}
holds for all $\lambda$ sufficiently large.
        \item Furthermore, if one takes $\overrightarrow{u}\in {\textbf{C}}^{\infty}_c(\Omega_T)$, then all boundary terms in \eqref{boundary carleman estimate1} and \eqref{boundary carleman estimate2} vanish, yielding the following interior Carleman estimates. For any $\overrightarrow{u}\in {\textbf{C}}^{\infty}_c(\Omega_T)$, there exists a constant $C>0$ (independent of $\lambda$) such that
\begin{align}
\label{interior carleman estimate1}
\begin{aligned}
 \left(
\lambda^2\|e^{-\phi}\overrightarrow{u}\|^2_{\textbf{L}^2(\Omega_T)} + \|e^{-\phi}(\nabla_x \cdot\overrightarrow{u})\|^2_{\textbf{L}^2(\Omega_T)}\right) &\leq C\|e^{-\phi}\mathcal{K}_{M,N,P}\overrightarrow{u}\|^2_{\textbf{L}^2(\Omega_T)}, \\ 
 \left(
\lambda^2\|e^{-\phi}\overrightarrow{u}\|^2_{\textbf{L}^2(\Omega_T)} + \|e^{-\phi}(\nabla_x \cdot\overrightarrow{u})\|^2_{\textbf{L}^2(\Omega_T)}\right) &\leq C\|e^{-\phi}\mathcal{L}_{A,Q}\overrightarrow{u}\|^2_{\textbf{L}^2(\Omega_T)} 
    \end{aligned}
\end{align}
for all $\lambda$ sufficiently large.
    \end{enumerate}
\end{remark}
\begin{proof}[Proof of Theorem \ref{BCE}] 
To absorb the first-order terms, we need to convexify the limiting Carleman weight $\phi$. Inspired by  \cite{caro2018determination, MishraPurohitVashisth2025, senapati2021stability}, we consider a convexified weight function $\phi_{\sigma}$,  as follows 
\begin{align}\label{modified Carleman weight}
    \phi_{\sigma}= \phi-\dfrac{\sigma((x+x_0)\cdot \o)^2}{2}, \qquad \sigma>1. 
\end{align}
where a fixed vector $x_0 \in \mathbb{R}^n$ is chosen satisfying
\(
x_0\cdot\omega
=
2+\sup_{x\in\Omega}|x|.
\)
Next, we start with defining  the conjugated operator $\mathcal{K}_{\phi_\sigma}$ of $\displaystyle \mathcal{K}_{M,N,P}$   by 
\begin{align}
\mathcal{K}_{\phi_\sigma}\overrightarrow{u}:&= e^{-\phi_\sigma} \mathcal{K}_{M,N,P}  \left (e^{\phi_\sigma}\overrightarrow{u}\right)=   e^{-\phi_\sigma}  \left( \partial_t -\Delta_x + M\cdot \nabla_x +\nabla_x\cdot N+P\right) \left (e^{\phi_\sigma}\overrightarrow{u}\right)\nonumber \\ &=  \partial_t\overrightarrow{u} - 2\left[\lambda-\sigma((x+x_0)\cdot\o)\right]\o\cdot \nabla_x \overrightarrow{u}+2\lambda \sigma ((x+x_0)\cdot \o)\overrightarrow{u}
-\sigma^2 ((x+x_0)\cdot \o)^2\overrightarrow{u} \nonumber\\
&\quad - \Delta_x \overrightarrow{u} +\sigma \overrightarrow{u} + M\cdot \nabla_x  \overrightarrow{u}+\left[\lambda-\sigma((x+x_0)\cdot\o)\right](\o\cdot M)\overrightarrow{u}+(\nabla_x\cdot N+P)\overrightarrow{u}\nonumber
\\
 &:= (\mathcal{K}_1+\mathcal{K}_2+\mathcal{K}_3)\overrightarrow{u}\label{conjugate operator in terms of P, Q, R, S}
\end{align}
where 
\begin{align}\label{eq for P}
\begin{aligned}
\mathcal{K}_1 & := \partial_t  - 2\left[\lambda-\sigma((x+x_0)\cdot\o)\right]\o\cdot \nabla_x+4 \sigma, \\\mathcal{K}_2 & :=
- \Delta_x+2\lambda \sigma ((x+x_0)\cdot \o)-\sigma^2 ((x+x_0)\cdot \o)^2-3\sigma,
\\
\mathcal{K}_3 & :=  M\cdot \nabla_x + \left[\lambda-\sigma((x+x_0)\cdot\o)\right](\o\cdot M) +\nabla_x\cdot N+P.
\end{aligned}
\end{align}
 Now, using the triangle inequality in Equation \eqref{conjugate operator in terms of P, Q, R, S}, we have that 
\begin{align}\label{v dagger lower bound}
\begin{aligned}
\lVert\mathcal{K}_{\phi_\sigma}\overrightarrow{u}\rVert^2_{\textbf{L}^2(\Omega_T)}&=\int_{\O_T} \lvert(\mathcal{K}_1+\mathcal{K}_2+\mathcal{K}_3)\overrightarrow{u}\rvert^2~dxdt \\&\geq \dfrac{1}{2}\lVert(\mathcal{K}_1+\mathcal{K}_2) \overrightarrow{u}\rVert^2_{\textbf{L}^2(\Omega_T)}- \lVert \mathcal{K}_3 \overrightarrow{u}\rVert^2_{\textbf{L}^2(\Omega_T)}\\&\geq \underbrace{ \int_{\O_T} \mathfrak{Re}\left(\left(\mathcal{K}_1 \overrightarrow{u}\right)\cdot \overline{\left(
\mathcal{K}_2 \overrightarrow{u}\right)}\right)~dxdt}_{I_1}- \underbrace{\lVert \mathcal{K}_3 \overrightarrow{u}\rVert^2_{\textbf{L}^2(\Omega_T)}}_{I_2}.
\end{aligned}
\end{align}
We now estimate the two terms $I_1$ and $I_2$ appearing in
\eqref{v dagger lower bound}. More precisely, our goal is to derive a suitable lower bound for $I_1$ and an upper bound for $I_2$, and then substitute these estimates into \eqref{v dagger lower bound}. This will enable us to obtain the desired lower bound for
$\|\mathcal{K}_{\phi_\sigma}\overrightarrow{u}\|_{\mathbf{L}^{2}(\Omega_T)}$.
To this end, we first observe that  
\begin{align}\label{I_1 expansion}
\begin{aligned}I_1&=\displaystyle\int_{\O_T} \mathfrak{Re}\left(\left(\mathcal{K}_1 \overrightarrow{u}\right)\cdot \overline{\left(
\mathcal{K}_2 \overrightarrow{u}\right)}\right)~dxdt\\&=-\int_{\O_T}\mathfrak{Re}\left((\partial_t \overrightarrow{u})\cdot\overline{(\Delta_x \overrightarrow{u})}\right)~dxdt+4\sigma\int_{\O_T}\mathcal{S}(x)\lvert \overrightarrow{u} \rvert^2~dxdt+\int_{\O_T}\mathfrak{Re}\left(\partial_t \overrightarrow{u} \cdot \overline{\mathcal{S}(x) \overrightarrow{u}} \right)~dxdt\\&\quad -4 \sigma\int_{\O_T}  \mathfrak{Re}\left(\overrightarrow{u} \cdot \overline{\Delta_x \overrightarrow{u}}\right)~dxdt+2\int_{\O_T} \left[\lambda-\sigma((x+x_0)\cdot\o)\right]\mathfrak{Re}\left((\o\cdot \nabla_x \overrightarrow{u})\cdot \overline{\Delta_x \overrightarrow{u}}\right)~dxdt\\&\quad -2\int_{\O_T} \mathcal{S}(x)\left[\lambda-\sigma((x+x_0)\cdot\o)\right]\mathfrak{Re}\left((\o\cdot \nabla_x \overrightarrow{u}) \cdot \overline{\overrightarrow{u}}\right)~dxdt\\:&\triangleq I_{1,1}+I_{1,2}+I_{1,3}+I_{1,4}+I_{1,5}+I_{1,6}
    \end{aligned}
\end{align} 
where $\mathcal{S}(x)=2\lambda \sigma ((x+x_0)\cdot \o)-\sigma^2 ((x+x_0)\cdot \o)^2-3\sigma$.
 First, we simplify the terms in $I_{1,1}$.
Since the boundary condition $\overrightarrow{u}|_{\partial\Omega_T} = \overrightarrow{0}$ entails $\partial_t \overrightarrow{u}|_{\partial\Omega_T} = \overrightarrow{0}$ and the initial condition $\overrightarrow{u}(0,x) = \overrightarrow{0}$ for all $x \in \Omega$ implies $\nabla_x \overrightarrow{u}(0,x) = \overrightarrow{0}$ for all $x \in \Omega$, we may apply the integration-by-parts formula to obtain
 \begin{align}\label{I11}
 \begin{aligned}
I_{1,1} &= -\mathfrak{Re}\int_{\O_T}(\partial_t \overrightarrow{u})\cdot\overline{(\Delta_x \overrightarrow{u})}~dxdt=\mathfrak{Re}\int_{\O_T}(\partial_t \nabla_x \overrightarrow{u})\cdot \overline{( \nabla_x \overrightarrow{u})}~dxdt=\dfrac{1}{2}\int_{\O_T}\partial_t (\lvert \nabla_x \overrightarrow{u}\rvert^2)~dxdt\\&=\dfrac{1}{2}\int_\O \lvert \nabla_x \overrightarrow{u}\rvert^2\bigg|_0^T~dx=\dfrac{1}{2}\int_\O \lvert \nabla_x \overrightarrow{u}\rvert^2(T,x)~dx\geq0.
\end{aligned}
     \end{align}
From \eqref{I_1 expansion}, $I_{1,2}$ is given by
\begin{align}\label{I12}
    \begin{aligned}
        I_{1,2} &=4\sigma\int_{\O_T}\mathcal{S}(x)\lvert \overrightarrow{u} \rvert^2~dxdt=  4\sigma\int_{\O_T}\left(
        2\lambda \sigma ((x+x_0)\cdot \o)-\sigma^2 ((x+x_0)\cdot \o)^2-3\sigma \right)\lvert \overrightarrow{u} \rvert^2~dxdt.
    \end{aligned}
\end{align}
Since  $x_0\cdot \o=2+\sup_{x\in \O}\lvert x\rvert$, we observe that
\begin{align}
    \begin{aligned}
        (x+x_0)\cdot \o \leq 2\left(1+\sup_{x\in\O}\lvert x\rvert\right)\ \mbox{ and}\ (x+x_0)\cdot \o\geq 2.
    \end{aligned}
\end{align}
Throughout the proof of the Carleman estimate, the above inequalities will be used repeatedly without explicit reference at each step. We now establish the following lower bound for the term $I_{1,2}$, when $\sigma>1$:
\begin{align}
    \begin{aligned}
        I_{1,2}\geq 16\lambda \sigma^2\int_{\O_T}\lvert \overrightarrow{u} \rvert^2~dxdt-8\sigma^3\int_{\O_T} \left(2+2\sup_{x\in\O}\lvert x\rvert\right)^2 \lvert \overrightarrow{u} \rvert^2~dxdt.
    \end{aligned}
\end{align}
Further, we choose $\lambda\geq \sigma \left(2+2\displaystyle\sup_{x\in\O}\lvert x\rvert\right)^2$. With this choice of $\lambda$, the above expression reduces to
\begin{align}
    \begin{aligned}
            I_{1,2}\geq 8\lambda \sigma^2\int_{\O_T}\lvert \overrightarrow{u} \rvert^2~dxdt.
    \end{aligned}
\end{align}
Next, applying integration by parts to the term $I_{1,3}$, we obtain
\begin{align}\label{I13}
    \begin{aligned}
     I_{1,3}&= \int_{\O_T}\mathfrak{Re}\left(\partial_t \overrightarrow{u} \cdot \overline{\mathcal{S}(x) \overrightarrow{u}} \right)~dxdt\\&=  \dfrac{1}{2}\int_{\O_T}\mathcal{S}(x)\partial_t\lvert \overrightarrow{u} \rvert^2~dxdt=\dfrac{1}{2}\int_{\O}\left(\mathcal{S}(x)\lvert \overrightarrow{u} \rvert^2\right)\bigg|_0^T~dx=\dfrac{1}{2}\int_{\O}\mathcal{S}(x)\lvert \overrightarrow{u} \rvert^2(T,x)~dx\\&=\dfrac{1}{2}\int_{\O} \left(2\lambda \sigma ((x+x_0)\cdot \o)-\sigma^2 ((x+x_0)\cdot \o)^2-3\sigma\right) \lvert \overrightarrow{u} \rvert^2(T,x)~dx.
    \end{aligned}
\end{align}
Now, for $\sigma>1$, the lower bound for the above expression is given by
\begin{align}
    \begin{aligned}
      I_{1,3}  &\geq \lambda \sigma \int_{\O} ((x+x_0)\cdot \o)\lvert \overrightarrow{u} \rvert^2(T,x)~dx -\sigma^2\int_{\O} \left(2+2\sup_{x\in\O}\lvert x\rvert\right)^2\lvert \overrightarrow{u} \rvert^2(T,x)~dx \\&\geq 2\lambda \sigma \int_{\O} \lvert \overrightarrow{u} \rvert^2(T,x)~dx -\sigma^2\int_{\O} \left(2+2\sup_{x\in\O}\lvert x\rvert\right)^2\lvert \overrightarrow{u} \rvert^2(T,x)~dx.
    \end{aligned}
\end{align}
Now if we choose $\lambda\geq \sigma \left(2+2\displaystyle\sup_{x\in\O}\lvert x\rvert\right)^2$ then the above inequality reduces to
\begin{align}
    \begin{aligned}
         I_{1,3}\geq \lambda \sigma \int_{\O} \lvert \overrightarrow{u} \rvert^2(T,x)~dx .
    \end{aligned}
\end{align}
Next, applying the integration-by-parts formula and using the boundary condition
$\overrightarrow{u}\big|_{\partial\Omega_T}=\overrightarrow{0}$,
we simplify the term $I_{1,4}$ to obtain
\begin{align}\label{I14}
    \begin{aligned}
        I_{1,4}= -4\sigma\mathfrak{Re} \int_{\O_T}  \overrightarrow{u} \cdot \overline{\Delta_x \overrightarrow{u}}~dxdt=4 \sigma\int_{\O_T} \lvert \nabla_x \overrightarrow{u}\rvert^2~dxdt.
    \end{aligned}
\end{align}
Since the boundary $\partial\Omega$ is smooth and
$\overrightarrow{u}|_{\partial \Omega}=0$ for all $t\in(0,T)$,
the tangential derivatives of $\overrightarrow{u}$ vanish on $\partial\Omega$. Consequently, the gradient $\nabla_x\overrightarrow{u}$ is normal to the boundary, that is,
$\nabla_x\overrightarrow{u}
=
\nu\left(\partial_\nu\overrightarrow{u}\right)$
on $\partial\Omega$,
where $\nu$ denotes the outward unit normal vector to $\partial\Omega$. We will use this observation, together with the integration-by-parts formula, in the subsequent calculations to obtain
\begin{align}\label{I15}
    \begin{aligned}
        I_{1,5}&=2\int_{\O_T} \left[\lambda-\sigma((x+x_0)\cdot\o)\right]\mathfrak{Re}\left((\o\cdot \nabla_x \overrightarrow{u})\cdot \overline{\Delta_x \overrightarrow{u}}\right)~dxdt\\&=2\mathfrak{Re}\int_{\partial\O_T} \left[\lambda-\sigma((x+x_0)\cdot\o)\right](\o\cdot \nabla_x \overrightarrow{u})\cdot \overline{\partial_\nu \overrightarrow{u}}~dS_xdt+2\sigma 
        \int_{\O_T}\lvert\o\cdot \nabla_x \overrightarrow{u}\rvert^2~dxdt
        \\&\qquad -2\mathfrak{Re}\int_{\O_T} \left[\lambda-\sigma((x+x_0)\cdot\o)\right]\left(\nabla_x(\o\cdot \nabla_x \overrightarrow{u})\cdot \overline{\nabla_x \overrightarrow{u} }\right)~dxdt
\\&=2\mathfrak{Re}\int_{\partial\O_T} \left[\lambda-\sigma((x+x_0)\cdot\o)\right](\o\cdot \nu)(\partial_\nu  \overrightarrow{u})\cdot \overline{\partial_\nu \overrightarrow{u}}~dS_xdt+2\sigma 
        \int_{\O_T}\lvert\o\cdot \nabla_x \overrightarrow{u}\rvert^2~dxdt
        \\&\qquad -2\mathfrak{Re}\int_{\O_T} \left[\lambda-\sigma((x+x_0)\cdot\o)\right]\left(\omega \cdot \nabla_x(\lvert \nabla_x \overrightarrow{u}\rvert^2)\right)~dxdt
        \\& =\int_{\partial\O_T} \left[\lambda-\sigma((x+x_0)\cdot\o)\right](\o\cdot \nu) \lvert\partial_\nu \overrightarrow{u}\rvert^2~dS_xdt+2\sigma \int_{\O_T} \lvert \o\cdot \nabla_x \overrightarrow{u}\rvert^2~dxdt-\sigma \int_{\O_T} \lvert \nabla_x \overrightarrow{u}\rvert^2~dxdt.
    \end{aligned}
\end{align}
Next, we choose $\lambda \geq 4\sigma\left(1+\displaystyle\sup_{x\in \O}\lvert x\rvert\right)$ so that the above expression boils down to
\begin{align}
    \begin{aligned}
        I_{1,5}\geq & \dfrac{\lambda}{2}\int_{(0,T)\times B(\o)} \lvert\o\cdot \nu\rvert \lvert\partial_\nu \overrightarrow{u}\rvert^2~dS_xdt-\lambda\int_{(0,T)\times F(\o)} \lvert\o\cdot \nu\rvert \lvert\partial_\nu \overrightarrow{u}\rvert^2~dS_xdt\\&+2\sigma \int_{\O_T} \lvert \o\cdot \nabla_x \overrightarrow{u}\rvert^2~dxdt-\sigma \int_{\O_T} \lvert \nabla_x \overrightarrow{u}\rvert^2~dxdt.
    \end{aligned}
\end{align}
From \eqref{I_1 expansion}, $I_{1,6}$ is given by
\begin{align}
    \begin{aligned}
        I_{1,6}&=-2\int_{\O_T} \mathcal{S}(x)\left[\lambda-\sigma((x+x_0)\cdot\o)\right]\mathfrak{Re}\left((\o\cdot \nabla_x \overrightarrow{u}) \cdot \overline{\overrightarrow{u}}\right)~dxdt\\&=-\int_{\O_T} \mathcal{S}(x)\left[\lambda-\sigma((x+x_0)\cdot\o)\right]\o\cdot \nabla_x(\lvert \overrightarrow{u}\rvert^2)~dxdt\\&=
-\int_{\O_T}
       \left[ 2\lambda \sigma ((x+x_0)\cdot \o)-\sigma^2 ((x+x_0)\cdot \o)^2-3\sigma\right]\left[\lambda-\sigma((x+x_0)\cdot\o)\right]\left(\o\cdot \nabla_x(\lvert \overrightarrow{u}\rvert^2)\right)~dxdt.
    \end{aligned}
\end{align}
Now using integration-by-parts, $I_{1,6}$ simplifies to
\begin{align}\label{I16}
    \begin{aligned}
        I_{1,6}&=\int_{\O_T}\left[
     2\lambda^2\sigma - 6\lambda \sigma^2 (x+x_0)\cdot\o+3\sigma^3((x+x_0)\cdot\o)^2+3\sigma^2
       \right]
       \lvert \overrightarrow{u}\rvert^2~dxdt\\&\geq
       \int_{\O_T}\left[
     2\lambda^2\sigma - 12\lambda \sigma^2 \left(1+\sup_{x\in\O}\lvert x\rvert\right)+12\sigma^3+3\sigma^2
       \right]
       \lvert \overrightarrow{u}\rvert^2~dxdt\\&\geq
      2\sigma \int_{\O_T} \left(\lambda - \frac{\sigma \left(1+\displaystyle\sup_{x\in \O}\lvert x\rvert \right)}{2}\right)^2\lvert \overrightarrow{u}\rvert^2~dxdt.
    \end{aligned}
\end{align}
For $\lambda\geq \sigma \left(1+\displaystyle\sup_{x\in \O}\lvert x\rvert \right)$, a lower-bound for $I_{1,6}$ is given by
\begin{align}
    \begin{aligned}
        I_{1,6}\geq  \dfrac{\lambda^2\sigma}{2}\int_{\O_T}\lvert \overrightarrow{u}\rvert^2~dxdt.
    \end{aligned}
\end{align}
Using the above calculations in Equation \eqref{I_1 expansion}, we get
\begin{align}\label{estimate for I1}
    \begin{aligned}
        I_1&\geq  8\lambda \sigma^2\int_{\O_T}\lvert \overrightarrow{u} \rvert^2~dxdt+\lambda \sigma \int_{\O} \lvert \overrightarrow{u} \rvert^2(T,x)~dx +4 \sigma\int_{\O_T} \lvert \nabla_x \overrightarrow{u}\rvert^2~dxdt\\&\qquad+\dfrac{\lambda}{2}\int_{(0,T)\times B(\o)} \lvert\o\cdot \nu\rvert \lvert\partial_\nu \overrightarrow{u}\rvert^2~dS_xdt-\lambda\int_{(0,T)\times F(\o)} \lvert\o\cdot \nu\rvert \lvert\partial_\nu \overrightarrow{u}\rvert^2~dS_xdt\\&\qquad+2\sigma \int_{\O_T} \lvert \o\cdot \nabla_x \overrightarrow{u}\rvert^2~dxdt-\sigma \int_{\O_T} \lvert \nabla_x \overrightarrow{u}\rvert^2~dxdt+\dfrac{\lambda^2\sigma}{2}\int_{\O_T}\lvert \overrightarrow{u}\rvert^2~dxdt.
    \end{aligned}
\end{align}
Next, the lower bound for $I_{2}$ is estimated as follows
\begin{align}\label{estimate for I2}
    \begin{aligned}
        I_2&=\lVert \mathcal{K}_3 \overrightarrow{u}\rVert^2_{\textbf{L}^2(\Omega_T)}\\&=\int_{\O_T}\left\lvert  M\cdot \nabla_x \overrightarrow{u}+ \left[\lambda-\sigma((x+x_0)\cdot\o)\right](\o\cdot M) \overrightarrow{u}+(\nabla_x\cdot N) \overrightarrow{u}+P\overrightarrow{u}\right\rvert^2~dxdt\\&\leq 8\lVert \overrightarrow{u}\rVert\rVert^2_{\textbf{L}^2(\O_T)}\left((\lambda-2\sigma)^2\lVert M\rVert^2_{\textbf{L}^2(\O_T)}+\lVert \nabla_x\cdot N\rVert^2_{\textbf{L}^2(\O_T)}+\lVert P \rVert^2_{\textbf{L}^2(\O_T)}
        \right)\\&\qquad+8\lVert M \rVert^2_{\textbf{L}^2(\O_T)}\lVert \nabla_x\cdot \overrightarrow{u} \rVert^2_{\textbf{L}^2(\O_T)}.
    \end{aligned}
\end{align}
Using Equations \eqref{estimate for I1} and \eqref{estimate for I2} in \eqref{v dagger lower bound}, we have
\begin{align}
    \begin{aligned}
\lVert\mathcal{K}_{\phi_\sigma}\overrightarrow{u}\rVert^2_{\textbf{L}^2(\Omega_T)}&\geq 8\lambda \sigma^2\int_{\O_T}\lvert \overrightarrow{u} \rvert^2~dxdt+\lambda \sigma \int_{\O} \lvert \overrightarrow{u} \rvert^2(T,x)~dx +4 \sigma\int_{\O_T} \lvert \nabla_x \overrightarrow{u}\rvert^2~dxdt\\&\qquad+\dfrac{\lambda}{2}\int_{(0,T)\times B(\o)} \lvert\o\cdot \nu\rvert \lvert\partial_\nu \overrightarrow{u}\rvert^2~dS_xdt-\lambda\int_{(0,T)\times F(\o)} \lvert\o\cdot \nu\rvert \lvert\partial_\nu \overrightarrow{u}\rvert^2~dS_xdt\\&\qquad+2\sigma \int_{\O_T} \lvert \o\cdot \nabla_x \overrightarrow{u}\rvert^2~dxdt-\sigma \int_{\O_T} \lvert \nabla_x \overrightarrow{u}\rvert^2~dxdt+\dfrac{\lambda^2\sigma}{2}\int_{\O_T}\lvert \overrightarrow{u}\rvert^2~dxdt\\&\qquad-8\lVert \overrightarrow{u}\rVert\rVert^2_{\textbf{L}^2(\O_T)}\left((\lambda-2\sigma)^2\lVert M\rVert^2_{\textbf{L}^2(\O_T)}+\lVert \nabla_x\cdot N\rVert^2_{\textbf{L}^2(\O_T)}+\lVert P \rVert^2_{\textbf{L}^2(\O_T)}
        \right)\\&\qquad -8\lVert M \rVert^2_{\textbf{L}^2(\O_T)}\lVert \nabla_x\cdot \overrightarrow{u} \rVert^2_{\textbf{L}^2(\O_T)} .
    \end{aligned}
\end{align}
The above expression can be simplified to
\begin{align}
    \begin{aligned}
\lVert\mathcal{K}_{\phi_\sigma}\overrightarrow{u}\rVert^2_{\textbf{L}^2(\Omega_T)}  &+\lambda\int_{(0,T)\times F(\o)} \lvert\o\cdot \nu\rvert \lvert\partial_\nu \overrightarrow{u}\rvert^2~dS_xdt\\&\geq \bigg[\lambda^2\left(\dfrac{\sigma}{2}-8\lVert M\rVert^2_{\textbf{L}^2(\O_T)}\right)+\lambda\left(8\sigma^2+32\sigma \lVert M\rVert^2_{\textbf{L}^2(\O_T)}\right)+32\sigma^2\lVert M\rVert^2_{\textbf{L}^2(\O_T)}\\&\qquad-8\lVert \nabla_x\cdot N\rVert^2_{\textbf{L}^2(\O_T)}-8\lVert P \rVert^2_{\textbf{L}^2(\O_T)}\bigg]\lVert \overrightarrow{u}\rVert\rVert^2_{\textbf{L}^2(\O_T)}+\lambda \sigma \int_{\O} \lvert \overrightarrow{u} \rvert^2(T,x)~dx\\&\qquad +\left(\sigma-8\lVert M \rVert^2_{\textbf{L}^2(\O_T)}\right)\lVert \nabla_x\cdot \overrightarrow{u} \rVert^2_{\textbf{L}^2(\O_T)} +\dfrac{\lambda}{2}\int_{(0,T)\times B(\o)} \lvert\o\cdot \nu\rvert \lvert\partial_\nu \overrightarrow{u}\rvert^2~dS_xdt.
    \end{aligned}
\end{align}
Now for $\lambda$ and $\sigma$ large enough, we have
\begin{align}\label{estimate for K in term of sigma}
    \begin{aligned}
\lambda^2\sigma \lVert \overrightarrow{u}\rVert^2_{\textbf{L}^2(\O_T)}&+\sigma\lVert \nabla_x\cdot \overrightarrow{u} \rVert^2_{\textbf{L}^2(\O_T)}+\lambda \sigma\int_{\O} \lvert \overrightarrow{u} \rvert^2(T,x)~dx+\lambda\int_{(0,T)\times B(\o)} \lvert\o\cdot \nu\rvert \lvert\partial_\nu \overrightarrow{u}\rvert^2~dS_xdt\\&\leq C
\lVert\mathcal{K}_{\phi_\sigma}\overrightarrow{u}\rVert^2_{\textbf{L}^2(\Omega_T)}  +C\lambda\int_{(0,T)\times F(\o)} \lvert\o\cdot \nu\rvert \lvert\partial_\nu \overrightarrow{u}\rvert^2~dS_xdt.
    \end{aligned}
\end{align}
Replacing $\overrightarrow{u}$ by $e^{-\phi_\sigma}\overrightarrow{u}$, the estimate in \eqref{estimate for K in term of sigma} becomes
\begin{align}
\begin{aligned}
   &\lambda\int_{(0,T)\times B(\o)} e^{-2\phi_\sigma}\lvert\o\cdot \nu\rvert \lvert\partial_\nu \overrightarrow{u}\rvert^2~dS_xdt 
   +\left(
\lambda^2\sigma\|e^{-\phi_\sigma}\overrightarrow{u}\|^2_{\textbf{L}^2(\Omega_T)} + \sigma\|e^{-\phi_\sigma}(\nabla_x \cdot\overrightarrow{u})\|^2_{\textbf{L}^2(\Omega_T)}\right)\\&\qquad+\lambda \sigma 
\|e^{-\phi_\sigma(T,x)}\overrightarrow{u}(T,x)\|^2_{\textbf{L}^2(\Omega)} \leq C\|e^{-\phi_\sigma}\mathcal{K}_{M,N,P}\overrightarrow{u}\|^2_{\textbf{L}^2(\Omega_T)} +C\lambda\int_{(0,T)\times F(\o)} e^{-2\phi_\sigma}\lvert\o\cdot \nu\rvert \lvert\partial_\nu \overrightarrow{u}\rvert^2~dS_xdt
    \end{aligned}
\end{align}
for $\lambda$ large enough. Since $\exp{\left(\dfrac{\sigma((x+x_0)\cdot \o)^2}{2}\right)}$ has strictly upper and lower bounds, the above inequality reduces to 
\begin{align}
\begin{aligned}
   &\lambda\int_{(0,T)\times B(\o)} e^{-2\phi}\lvert\o\cdot \nu\rvert \lvert\partial_\nu \overrightarrow{u}\rvert^2~dS_xdt 
   +\left(
\lambda^2\|e^{-\phi}\overrightarrow{u}\|^2_{\textbf{L}^2(\Omega_T)} +\|e^{-\phi}(\nabla_x \cdot\overrightarrow{u})\|^2_{\textbf{L}^2(\Omega_T)}\right)\\&+\lambda  
\|e^{-\phi(T,x)}\overrightarrow{u}(T,x)\|^2_{\textbf{L}^2(\Omega)} \leq C\|e^{-\phi}\mathcal{K}_{M,N,P}\overrightarrow{u}\|^2_{\textbf{L}^2(\Omega_T)} +C\lambda\int_{(0,T)\times F(\o)} e^{-2\phi}\lvert\o\cdot \nu\rvert \lvert\partial_\nu \overrightarrow{u}\rvert^2~dS_xdt.
    \end{aligned}
\end{align}
This completes the proof.
\end{proof}
\subsection{Interior Carleman estimates in negative-order Sobolev spaces} 
 As mentioned earlier, the goal of this subsection is to establish interior Carleman estimates in negative-order Sobolev spaces. These estimates are essential for proving the solvability of the weighted operator associated with $\mathcal{K}_{M,N,P}$ (see Proposition~\ref{Proposition 1}). In particular, we derive an $H^{-1}-H^{-2}$ Carleman estimate, which is a key ingredient in constructing geometric optics solutions. First, we introduce several function spaces and auxiliary operators that will be used throughout the subsequent analysis. For $s\in \mathbb{R}$, let us define $\lambda$-dependent Sobolev spaces ${L}^2(0,T;{H}^s_\lambda (\mathbb{R}^n))$ by 
\begin{align*}
L^2(0,T;H^s_{\lambda}(\mathbb{R}^n)):= \left\{w(t,\cdot)\in \mathcal{S}'(\mathbb{R}^n): \langle \lambda,\xi \rangle^{s/2}\widehat{w}(t,\xi) \in L^2((0,T)\times\mathbb{R}^n)\right\}
\end{align*}
endowed with the norm 
\begin{align*}
\|w\|^2_{L^2(0,T;H^s_{\lambda}(\mathbb{R}^n))}:=\int_{0}^{T}\int_{\mathbb{R}^n} \langle  \lambda,\xi \rangle^{s}|\widehat{w}(t,\xi)|^2\ d\xi dt
\end{align*}
 where $\mathcal{S}'(\mathbb{R}^n)$ denotes the space of tempered distributions on $\mathbb{R}^n$ and $\langle \lambda ,\xi\rangle:=\sqrt{\lambda^2+\lvert \xi\rvert^2}$ denotes the \textit{Japanese bracket}. In particular, for any non-negative integer $m$, we have that
 \begin{align}\label{vector semiclassical sobolev space order s} \|\overrightarrow{u}\|^2_{L^2(0,T;\mathbf{H}^{m}_\lambda(\O))}&:= \sum_{i=1}^{k} \sum_{j=0}^m  \lambda^{2(m-j)}\|u_i\|^2_{L^2(0,T;{H}^j(\O))}, \\
 \|\overrightarrow{u}\|^2_{L^2(0,T;\mathbf{H}^{-m}_\lambda(\mathbb{R}^n))}&:= \sum_{i=1}^{k} \lVert (\lambda^2-\Delta_x)^{-m/2} u_i\rVert^2_{L^2(0,T;L^2(\mathbb{R}^n))}.
 \end{align}
From the above definition, we observe that
 \begin{align}
     \begin{aligned}
\|\overrightarrow{u}\|^2_{L^2(0,T;\mathbf{H}^{-2}_\lambda(\mathbb{R}^n))} \leq \dfrac{C}{\lambda}\|\overrightarrow{u}\|^2_{L^2(0,T;\mathbf{H}^{-1}_\lambda(\mathbb{R}^n))}.
     \end{aligned}
 \end{align}
 Also, we introduce the class of symbols of order $m$ for $\lambda\geq \lambda_0>0$ and $\xi\in \mathbb{R}^n$ denoted by $   \mathcal{S}_\lambda^m$ and given by
 \begin{align}
     \begin{aligned}
         \mathcal{S}_\lambda^m:=\left\{a(\cdot,\cdot,\lambda) \in C^{\infty}(\mathbb{R}^n\times\mathbb{R}^n) : \lvert \partial_x^\alpha \partial_\xi^\beta a(x,\xi,\lambda)\rvert\leq C(\alpha,\beta)\langle \lambda,\xi \rangle^{m-\lvert \beta \rvert}\right\}
     \end{aligned}
 \end{align}
 for all $\alpha, \beta \in (\mathbb{N}\cup\{0\})^n$.
Now, let $\mathcal{K}^*_{M,N,P}$ denotes the $L^2$-adjoint operator of $\mathcal{K}_{M,N,P}$ given by
\begin{align}
 \label{Operator K^{*}_{A,Q}}
   \displaystyle \mathcal{K}^*_{M,N,P}\overrightarrow{v} := -\partial_t\overrightarrow{v}-\Delta_x \overrightarrow{v}-[M]^* \cdot\nabla_x \overrightarrow{v}-(\nabla_x\cdot [M]^*)\overrightarrow{v} +\nabla_x\cdot [N]^*\overrightarrow{v}+[P]^*\overrightarrow{v}
\end{align} where $\overrightarrow{v}:= (v_1, v_2, \cdots, v_k)^T$ and $[B]^*$ is the  conjugate transpose of matrix $B$.
In order to obtain the solvability result for $\mathcal{K}_{\phi}\overrightarrow{u}= \overrightarrow{W}$, we require the following $H^{-1}-H^{-2}$ interior Carleman estimate:
\begin{lemma}
Let $\phi$ and $\mathcal{K}^*_{M,N,P}$ are as defined by Equations \eqref{carleman weight} and \eqref{Operator K^{*}_{A,Q}}, respectively. Then for any $\overrightarrow{v} \in {C}^1([0,T];\mathbf{C}^{\infty}_c(\Omega))$ with $\overrightarrow{v}(T,x)=\overrightarrow{0}$, there exists a constant $C>0$, such that the estimate 
    \begin{align}\label{shifted by -1 carleman estimate}
\|\overrightarrow{v}\|_{L^2(0,T;\mathbf{H}^{-1}_\lambda(\mathbb{R}^n))} \leq C  \| \mathcal{K}^*_{\phi} \overrightarrow{v}\|_{L^2(0,T;\mathbf{H}^{-2}_{\lambda}(\mathbb{R}^n))}
\end{align}
holds $\lambda$ large enough.
\end{lemma}
\begin{proof}
 The idea of the proof is inspired from \cite{Bellassoued_2021, Bellassoued01102006, sahoo2019partial}. Consider the modified Carleman weight
\begin{align}\label{modified Carleman weight in negative order sobolev space}
    \phi_{\sigma}^-=- \phi-\dfrac{\sigma((x+x_0)\cdot \o)^2}{2}, \qquad \sigma>1, 
\end{align}
and the conjugate adjoint operator
\begin{align*}
e^{-\phi_{\sigma}^-}\mathcal{K}^*_{M,N,P} e^{\phi_{\sigma}^-}\overrightarrow{v} = &  (-\partial_t \overrightarrow{v} - \Delta_x \overrightarrow{v}) + 2\left[\lambda+\sigma((x+x_0)\cdot\o)\right]\o\cdot \nabla_x \overrightarrow{v} + \sigma\overrightarrow{v}\\
    & - \left[2\lambda \sigma ((x+x_0)\cdot \o)+\sigma^2 ((x+x_0)\cdot \o)^2\right]\overrightarrow{v}\\&+ \left[\lambda+\sigma((x+x_0)\cdot\o)\right](\o\cdot [M]^*)\overrightarrow{v} \\&- [M]^*\cdot \nabla_x\overrightarrow{v}-(\nabla_x\cdot [M]^*) \overrightarrow{v}+\nabla_x\cdot [N]^*\overrightarrow{v}+[P]^*\overrightarrow{v}
    \\
    := &\mathcal{K}_1^-\overrightarrow{v} + \mathcal{K}_2^-\overrightarrow{v} + \mathcal{K}_3^-\overrightarrow{v}
\end{align*}
where 
\begin{align}\label{equations for k1-k2-k3-}
\begin{aligned}
\mathcal{K}_1^- & := -\partial_t  + 2\left[\lambda+\sigma((x+x_0)\cdot\o)\right]\o\cdot \nabla_x+4 \sigma, \\\mathcal{K}_2^- & :=
- \Delta_x-2\lambda \sigma ((x+x_0)\cdot \o)-\sigma^2 ((x+x_0)\cdot \o)^2-3\sigma,
\\
\mathcal{K}_3^- & :=  -[M]^*\cdot \nabla_x + \left[\lambda+\sigma((x+x_0)\cdot\o)\right](\o\cdot [M]^*) -(\nabla_x\cdot [M]^*) +\nabla_x\cdot [N]^*+[P]^*.
\end{aligned}
\end{align}
We denote the symbols of $\mathcal{K}_1^-$, $\mathcal{K}_2^-$ by $\mathtt{k_1^-}$ and $\mathtt{k_2^-}$ respectively, and are given by
\begin{align}\label{symbols for k1 and k2}
    \begin{aligned}
       \mathtt{k_1^-}&=-\mathrm{i}\tau+2\mathrm{i}\left[\lambda+\sigma((x+x_0)\cdot\o)\right]\o\cdot\xi+4\sigma,\\
         \mathtt{k_2^-}&=\lvert\xi\rvert^2-2\lambda \sigma ((x+x_0)\cdot \o)-\sigma^2 ((x+x_0)\cdot \o)^2-3\sigma.
    \end{aligned}
\end{align}
Let $\widetilde{\Omega} \subset \mathbb{R}^n$ be a bounded, open and smooth set such that $\Omega$ is relatively compact in $\widetilde{\Omega}$. Then, 
from \eqref{v dagger lower bound} and \eqref{estimate for K in term of sigma}, we observe that for any $\overrightarrow{u} \in {C}^1([0,T];{\mathbf{C}}^{\infty}_c(\widetilde{\Omega}))$ with $\overrightarrow{u}\lvert_{t=0}=\overrightarrow{0}$, there exist $\sigma>1$ such that 
\begin{align}
    \begin{aligned}
\sigma\|\overrightarrow{u}\|^2_{L^2(0,T;\mathbf{H}^1_{\lambda}(\mathbb{R}^n))}\leq C\|(\mathcal{K}_{1}+\mathcal{K}_2)\overrightarrow{u}\|^2_{L^2(0,T;\textbf{L}^2(\mathbb{R}^n))} 
    \end{aligned}
\end{align}
holds for $\lambda$ large enough, where $\mathcal{K}_1$ and $\mathcal{K}_2$ are given by Equation \eqref{eq for P}. Now, from an analogue calculation used to derive \eqref{estimate for K in term of sigma}, we obtain a similar estimate for the conjugate operator $e^{-\phi_{\sigma}^-}\mathcal{K}^*_{0,0,0} e^{\phi_{\sigma}^-}$. For any $\overrightarrow{w}\in {C}^1([0,T];{\mathbf{C}}^{\infty}_c(\widetilde{\Omega}))$ with $\overrightarrow{w}\lvert_{t=T}=\overrightarrow{0}$, there exist $\sigma>1$ such that 
\begin{align}\label{interior carleman for k star in term of sigma}
    \begin{aligned}
\sigma\|\overrightarrow{w}\|^2_{L^2(0,T;\mathbf{H}^1_{\lambda}(\mathbb{R}^n))}\leq C\|(\mathcal{K}_{1}^-+\mathcal{K}_2^-)\overrightarrow{w}\|^2_{L^2(0,T;\textbf{L}^2(\mathbb{R}^n))} 
    \end{aligned}
\end{align}
holds for $\lambda$ large enough, where $\mathcal{K}_1^-$ and $\mathcal{K}_2^-$ are given by Equation \eqref{equations for k1-k2-k3-}. Next, we consider an intermediate open subset $\O_0$, relatively compact in $\widetilde{\O}$ such that $\overline{\O}\subset \O_0\subset \widetilde{\O}$ and cut-off functions $\varphi_1,\varphi_2 \in C_c^\infty(\mathbb{R}^n;[0,1])$ with 
$\supp~ \varphi_1\subset\O_0$, $\supp~ \varphi_2\subset\widetilde{\O}$
and 
\begin{align}\label{cutoff functions}
    \begin{aligned}
        \varphi_1&=1, \text{ on } \overline{\O},\\
        \varphi_2&=1, \text{ on } \overline{\O_0}.
    \end{aligned}
\end{align}
Note that for any $\overrightarrow{v}\in {C}^1([0,T];{\mathbf{C}}^{\infty}_c(\Omega))$, the operator 
$\varphi_2 (\lambda^2-\Delta_x)^{-1}\overrightarrow{v}\in  {C}^1([0,T];{\mathbf{C}}^{\infty}_c(\widetilde{\Omega}))$. Also, $\varphi_2 (\lambda^2-\Delta_x)^{-1}\overrightarrow{v}(T,\cdot)=\overrightarrow{0}$ with $\supp(\varphi_2 (\lambda^2-\Delta_x)^{-1}v_i)\subset \supp(\varphi_2)\subset \widetilde{\O}$ for each $1\leq i\leq k$. 
Next, we derive suitable upper and lower bounds for $\varphi_2 (\lambda^2-\Delta_x)^{-1}\overrightarrow{v}$ for further analysis. 
We start with introducing the commutator bracket $[\cdot,\cdot]$, defined as 
\(
    [A,B]:=AB-BA.
\)
For $(\lambda^2-\Delta_x)^{-1}\in \mathcal{S}_\lambda^{-2}$ and $\varphi_2 \in \mathcal{S}_\lambda^{0}$, the commutator has the following estimate
\begin{align}\label{commutator bracket norm}
    \begin{aligned}
        \left\lVert [(\lambda^2-\Delta_x)^{-1}, \varphi_2]\overrightarrow{v}\right\rVert_{\mathbf{H}^1_{\lambda}(\mathbb{R}^n)}\leq C  \lVert \overrightarrow{v}\rVert_{\mathbf{H}^{-2}_{\lambda}(\mathbb{R}^n)}\leq \dfrac{C}{\lambda}  \lVert \overrightarrow{v}\rVert_{\mathbf{H}^{-1}_{\lambda}(\mathbb{R}^n)}.
    \end{aligned}
\end{align}
Using
$[(\lambda^2-\Delta_x)^{-1}, \varphi_2]=(\lambda^2-\Delta_x)^{-1} \varphi_2-\varphi_2 (\lambda^2-\Delta_x)^{-1}$ and \eqref{commutator bracket norm} along with the second triangle inequality, we obtain
\begin{align}\label{upper bound estimate for varphi2}
    \begin{aligned}
     \lVert   \varphi_2 (\lambda^2-\Delta_x)^{-1}\overrightarrow{v}\rVert_{L^2(0,T;\mathbf{H}^1_{\lambda}(\mathbb{R}^n))}&= \lVert 
(\lambda^2-\Delta_x)^{-1} \varphi_2 \overrightarrow{v}-
     [(\lambda^2-\Delta_x)^{-1}, \varphi_2]\overrightarrow{v}\rVert_{L^2(0,T;\mathbf{H}^1_{\lambda}(\mathbb{R}^n))}\\&\geq \lVert\overrightarrow{v}\rVert_{L^2(0,T;\mathbf{H}^{-1}_{\lambda}(\mathbb{R}^n))}-\dfrac{C}{\lambda}  \lVert \overrightarrow{v}\rVert_{\mathbf{H}^{-1}_{\lambda}(\mathbb{R}^n)}\\&\geq \lVert\overrightarrow{v}\rVert_{L^2(0,T;\mathbf{H}^{-1}_{\lambda}(\mathbb{R}^n))}
    \end{aligned}
\end{align}
for $\lambda$ large enough. Now to derive the  upper bound for $\varphi_2 (\lambda^2-\Delta_x)^{-1}$, we insert $\varphi_2 (\lambda^2-\Delta_x)^{-1}\overrightarrow{v}$ in-place of $\overrightarrow{w}$ in Carleman estimate given by Equation \eqref{interior carleman for k star in term of sigma} to arrive at  
\begin{align}\label{3.15}
    \begin{aligned}
&\sigma^{1/2}\|\varphi_2 (\lambda^2-\Delta_x)^{-1}\overrightarrow{v}\|_{L^2(0,T;\mathbf{H}^1_{\lambda}(\mathbb{R}^n))}\\&\leq C\|(\mathcal{K}_{1}^-+\mathcal{K}_2^-)\varphi_2 (\lambda^2-\Delta_x)^{-1}\overrightarrow{v}\|_{L^2(0,T;\mathbf{L}^2(\mathbb{R}^n))} \\&\leq C\|\left([\mathcal{K}_{1}^-+\mathcal{K}_{2}^-,\varphi_2]+\varphi_2(\mathcal{K}_{1}^-+\mathcal{K}_2^-) \right)(\lambda^2-\Delta_x)^{-1}\overrightarrow{v}\|_{L^2(0,T;\mathbf{L}^2(\mathbb{R}^n))} \\&\leq C\lVert [\mathcal{K}_{1}^-+\mathcal{K}_{2}^-,\varphi_2](\lambda^2-\Delta_x)^{-1}\overrightarrow{v}\|_{L^2(0,T;\mathbf{L}^2(\mathbb{R}^n))} +C\lVert (\mathcal{K}_{1}^-+\mathcal{K}_2^-) (\lambda^2-\Delta_x)^{-1}\overrightarrow{v}\|_{L^2(0,T;\mathbf{L}^2(\mathbb{R}^n))} 
    \end{aligned}
\end{align}
From a simple calculation, we observe that
\begin{align}
    \begin{aligned}
        [\mathcal{K}_{1}^-+\mathcal{K}_{2}^- ,\varphi_2]=2\left[\lambda+\sigma((x+x_0)\cdot\o)\right](\o\cdot\nabla_x \varphi_2)-(\Delta_x\varphi_2)-2 \nabla_x\varphi_2\cdot \nabla_x
    \end{aligned}
\end{align}
is an operator of order $1$. Thus, the upper bound of the first term in the last inequality \eqref{3.15} is given by
\begin{align}
    \begin{aligned}
\lVert [\mathcal{K}_{1}^-+\mathcal{K}_{2}^-,\varphi_2](\lambda^2-\Delta_x)^{-1}\overrightarrow{v}\|_{L^2(0,T;\mathbf{L}^2(\mathbb{R}^n))}&= \lVert [\mathcal{K}_{1}^-+\mathcal{K}_{2}^-,\varphi_2](\lambda^2-\Delta_x)^{-1}(\varphi_1\overrightarrow{v})\|_{L^2(0,T;\mathbf{L}^2(\mathbb{R}^n))}\\&\leq\lVert [\mathcal{K}_{1}^-+\mathcal{K}_{2}^-,\varphi_2]\varphi_1(\lambda^2-\Delta_x)^{-1}\overrightarrow{v}\|_{L^2(0,T;\mathbf{L}^2(\mathbb{R}^n))}\\&\qquad+\lVert [\mathcal{K}_{1}^-+\mathcal{K}_{2}^-,\varphi_2][(\lambda^2-\Delta_x)^{-1},\varphi_1]\overrightarrow{v}\|_{L^2(0,T;\mathbf{L}^2(\mathbb{R}^n))}
    \end{aligned}
\end{align}
From \eqref{cutoff functions}, we know that $\varphi_2=1$ on $\overline{\O_0}$ therefore $[\mathcal{K}_{1}^-+\mathcal{K}_{2}^-,\varphi_2]=0$ in $(0,T)\times \O_0$. Also, due to the support condition on $\varphi_1$, we can see that 
$\supp ~\varphi_1(\lambda^2-\Delta_x)^{-1}{v_i} \subset \O_0$
 for $1\leq i\leq k$. Using these observations, we have
 \begin{align*}
      [\mathcal{K}_{1}^-+\mathcal{K}_{2}^-,\varphi_2]\varphi_1(\lambda^2-\Delta_x)^{-1}\overrightarrow{v}=\overrightarrow{0} \text{ on } (0,T)\times \mathbb{R}^n.
 \end{align*}
Next, using the fact that $[\mathcal{K}_{1}^-+\mathcal{K}_{2}^- ,\varphi_2]$ and $[(\lambda^2-\Delta_x)^{-1},\varphi_1]$ having symbols of order $1$ and $-3$ respectively, we have 
\begin{align}
    \lVert [\mathcal{K}_{1}^-+\mathcal{K}_{2}^-,\varphi_2][(\lambda^2-\Delta_x)^{-1},\varphi_1]\overrightarrow{v}\|_{L^2(0,T;\mathbf{L}^2(\mathbb{R}^n))}\leq C\lVert\overrightarrow{v}\rVert_{L^2(0,T;\mathbf{H}^{-2}_\lambda(\mathbb{R}^n))}\leq \dfrac{C}{\lambda}  \lVert \overrightarrow{v}\rVert_{L^2(0,T;\mathbf{H}^{-1}_{\lambda}(\mathbb{R}^n)}.
\end{align}
Therefore, we can conclude that
\begin{align}
    \begin{aligned}
        \lVert [\mathcal{K}_{1}^-+\mathcal{K}_{2}^-,\varphi_2](\lambda^2-\Delta_x)^{-1}\overrightarrow{v}\|_{L^2(0,T;\mathbf{L}^2(\mathbb{R}^n))}\leq \dfrac{C}{\lambda}\lVert \overrightarrow{v}\rVert_{L^2(0,T;\mathbf{H}^{-1}_{\lambda}(\mathbb{R}^n)}.
    \end{aligned}
\end{align}
Using the above estimate in \eqref{3.15} will leads to
\begin{align}
    \begin{aligned}
&\sigma^{1/2}\|\varphi_2 (\lambda^2-\Delta_x)^{-1}\overrightarrow{v}\|_{L^2(0,T;\mathbf{H}^1_{\lambda}(\mathbb{R}^n))}\leq C\lVert (\mathcal{K}_{1}^-+\mathcal{K}_2^-) (\lambda^2-\Delta_x)^{-1}\overrightarrow{v}\|_{L^2(0,T;\mathbf{L}^2(\mathbb{R}^n))}. 
    \end{aligned}
\end{align}
Combining the above estimate with \eqref{upper bound estimate for varphi2}, we get
\begin{align}\label{3.21}
    \begin{aligned}
\sigma^{1/2}\lVert\overrightarrow{v}\rVert_{L^2(0,T;\mathbf{H}^{-1}_{\lambda}(\mathbb{R}^n))}    \leq C\lVert (\mathcal{K}_{1}^-+\mathcal{K}_2^-) (\lambda^2-\Delta_x)^{-1}\overrightarrow{v}\|_{L^2(0,T;\mathbf{L}^2(\mathbb{R}^n))} 
    \end{aligned}
\end{align}
To obtain the desired bound, we further simplify the above estimate using the composition of pseudo-differential operators. We observe that 
\begin{align}\label{3.20}
    \begin{aligned}
    (\lambda^2-\Delta_x)^{-1}  (\mathcal{K}_{1}^-+\mathcal{K}_2^-) (\lambda^2-\Delta_x)=(\mathcal{K}_{1}^-+\mathcal{K}_2^-)  + \mathcal{R^-}(x,D,\lambda),
    \end{aligned}
\end{align}
where the symbol of $\mathcal{R^-}$, say  $\mathtt{r^-}$, is given by
\begin{align}
    \begin{aligned}
        \mathtt{r^-}(x,\xi,\lambda)&=\dfrac{1}{\mathrm{i}}\nabla_\xi(\lambda^2+\lvert \xi\rvert^2)^{-1}\cdot \nabla_x(\mathtt{k_1^-}+\mathtt{k_2^-}) (\lambda^2+\lvert \xi\rvert^2) \text{ mod}(\sigma^2\times\mathcal{S}_{\lambda}^{-1})\\
        &=-4\mathrm{i}\dfrac{\left( \lambda\sigma+\sigma^2(x+x_0)\cdot\o-\mathrm{i}\sigma \o\cdot\xi\right)}{(\lambda^2+\lvert \xi\rvert^2)}(\o\cdot\xi)\text{ mod}(\sigma^2\times\mathcal{S}_{\lambda}^{-1})
    \end{aligned}
\end{align}
where $\mathtt{k_1^-}$ and $\mathtt{k_2^-}$ are given by Equation \eqref{symbols for k1 and k2}. Thus, we have 
\begin{align}
    \begin{aligned}
        \lVert \mathcal{R^-}(\lambda^2-\Delta_x)^{-1}\overrightarrow{v}\rVert_{L^2(0,T;\mathbf{L}^2(\mathbb{R}^n)}&\leq C\sigma^2\lVert (\lambda^2-\Delta_x)^{-1}\overrightarrow{v}\rVert_{L^2(0,T;\mathbf{L}^2(\mathbb{R}^n)}\\&\leq C\sigma^2\lVert \overrightarrow{v}\rVert_{L^2(0,T;\mathbf{H}^{-2}_\lambda(\mathbb{R}^n)}\leq C\dfrac{\sigma^2}{\lambda}\lVert \overrightarrow{v}\rVert_{L^2(0,T;\mathbf{H}^{-1}_\lambda(\mathbb{R}^n)}.
    \end{aligned}
\end{align}
Using above estimate along with \eqref{3.20}, we have 
\begin{align}
    \begin{aligned}
    \lVert    (\mathcal{K}_{1}^-+&\mathcal{K}_2^-)(\lambda^2-\Delta_x)^{-1}\overrightarrow{v}\rVert_{L^2(0,T;\mathbf{L}^2(\mathbb{R}^n)}\\&\leq \lVert (\lambda^2-\Delta_x)^{-1}(\mathcal{K}_{1}^-+\mathcal{K}_2^-)\overrightarrow{v}\rVert_{L^2(0,T;\mathbf{L}^2(\mathbb{R}^n)}+   \lVert \mathcal{R^-}(\lambda^2-\Delta_x)^{-1}\overrightarrow{v}\rVert_{L^2(0,T;\mathbf{L}^2(\mathbb{R}^n)}\\&\leq \lVert (\mathcal{K}_{1}^-+\mathcal{K}_2^-)\overrightarrow{v}\rVert_{L^2(0,T;\mathbf{H}^{-2}_\lambda(\mathbb{R}^n)}+C\dfrac{\sigma^2}{\lambda}\lVert \overrightarrow{v}\rVert_{L^2(0,T;\mathbf{H}^{-1}_\lambda(\mathbb{R}^n)}.
    \end{aligned}
\end{align}
Using the above estimate in \eqref{3.21}, we arrive at
\begin{align}\label{3.22}
    \begin{aligned}
\sigma^{1/2}\lVert\overrightarrow{v}\rVert_{L^2(0,T;\mathbf{H}^{-1}_{\lambda}(\mathbb{R}^n))}&\leq C \left(\lVert (\mathcal{K}_{1}^-+\mathcal{K}_2^-)\overrightarrow{v}\rVert_{L^2(0,T;\mathbf{H}^{-2}_\lambda(\mathbb{R}^n)}+\dfrac{\sigma^2}{\lambda}\lVert \overrightarrow{v}\rVert_{L^2(0,T;\mathbf{H}^{-1}_\lambda(\mathbb{R}^n)}\right)\\&\leq C\lVert (\mathcal{K}_{1}^-+\mathcal{K}_2^-)\overrightarrow{v}\rVert_{L^2(0,T;\mathbf{H}^{-2}_\lambda(\mathbb{R}^n)}\\& \leq C\bigg(\lVert \underbrace{(\mathcal{K}_{1}^-+\mathcal{K}_2^-+
\mathcal{K}_{3}^-)}_{:=\mathcal{K}^*_{\phi_{\sigma}^-}}\overrightarrow{v}\rVert_{L^2(0,T;\mathbf{H}^{-2}_\lambda(\mathbb{R}^n)}+\lVert 
\mathcal{K}_{3}^-\overrightarrow{v}\rVert_{L^2(0,T;\mathbf{H}^{-2}_\lambda(\mathbb{R}^n)}\bigg)
    \end{aligned}
\end{align}
where $\dfrac{\sigma^2}{\lambda}\lVert \overrightarrow{v}\rVert_{L^2(0,T;\mathbf{H}^{-1}(\mathbb{R}^n)}$ absorbs into the left-hand-side of above inequality for $\lambda$ large enough. Next, we simplify the last term in the above expression. Recall that
\begin{align*}
    \mathcal{K}_3^-  :=  -[M]^*\cdot \nabla_x + \left[\lambda+\sigma((x+x_0)\cdot\o)\right](\o\cdot [M]^*) -(\nabla_x\cdot [M]^*) +\nabla_x\cdot [N]^*+[P]^*.
\end{align*}
Thus, we have
\begin{align}\label{3.23}
    \begin{aligned}
        \lVert 
\mathcal{K}_{3}^-\overrightarrow{v}\rVert_{L^2(0,T;\mathbf{H}^{-2}_\lambda(\mathbb{R}^n)}^2 &\leq C\sum_{i,j=1}^k\Bigg(  \big\lVert \overline{\overrightarrow{M}_{ij}}\cdot \nabla_x v_i\big\rVert^2_{L^2(0,T;{H}^{-2}_\lambda(\mathbb{R}^n)}+\big\lVert(\nabla_x\cdot \overline{\overrightarrow{M}_{ij}})v_i\big\rVert^2_{L^2(0,T;{H}^{-2}_\lambda(\mathbb{R}^n)}\\&\qquad\qquad+\big\lVert(\nabla_x\cdot \overline{\overrightarrow{N}_{ij}})v_i\big\rVert^2_{L^2(0,T;{H}^{-2}_\lambda(\mathbb{R}^n)} + \big\lVert \overline{p_{ij}}v_i\big\rVert^2_{L^2(0,T;{H}^{-2}_\lambda(\mathbb{R}^n)}\\&\qquad\qquad+\big\lVert\lambda(\overline{\o\cdot \overrightarrow{M}_{ij}})v_i\big\rVert^2_{L^2(0,T;{H}^{-2}_\lambda(\mathbb{R}^n)}\Bigg)
    \end{aligned}
\end{align}
for $\lambda \geq 2\sigma \left(1+\displaystyle\sup_{x\in \O}\lvert x \rvert\right)$.
Next, it is required to derive the upper bound for each term on the right-hand side of the above expression to obtain the desired bound. We derive an upper bound for the general term. Let  $X \in H^r(\O_T,\mathbb{C})$ for  $r>\dfrac{n+1}{2}+2$. From the extension theorem, there exists a bounded linear operator $\mathcal{B}:  H^r(\O_T) \longrightarrow H^r(\mathbb{R}^{1+n})$ such that
\begin{align*}
    \mathcal{B}( X)=X\text{ in } \O_T, \text{ and } \lVert \mathcal{B}( X)\rVert_{H^r(\mathbb{R}^{1+n})}\leq \lVert X\rVert_{H^r(\O_T)}.
\end{align*}
 For $w\in C_c^{\infty}(\O_T)$, we have
\begin{align}
    \begin{aligned}
        \widehat{\mathcal{B}( X)w}(\tau,\xi)&= \widehat{\mathcal{B}( X)}*\widehat{w}(\tau,\xi).
    \end{aligned}
\end{align}
After multiplying the above expression by $\langle\lambda, (\tau,\xi)\rangle^{-2}$ and using Peetre's inequality, we arrive at the following estimate
\begin{align}
    \begin{aligned}
   &\left\lvert \langle\lambda, (\tau,\xi)\rangle^{-2}    \widehat{\mathcal{B}( X)w}(\tau,\xi)\right\rvert\\&\qquad=  \left\lvert\langle\lambda, (\tau,\xi)\rangle^{-2} \int_{\mathbb{R}^{1+n}}\widehat{\mathcal{B}( X)}(\tau-t,\xi-x) \widehat{w}(t,x)~dxdt\right\rvert\\&\qquad\leq \int_{\mathbb{R}^{1+n}}\langle\lambda, (\tau,\xi)\rangle^{-2} \langle\lambda,(t,x)\rangle^2 \langle\lambda,(t,x)\rangle^{-2} \left\lvert \widehat{\mathcal{B}( X)}\right\rvert(\tau-t,\xi-x) \left\lvert\widehat{w}\right\rvert(t,x)~dxdt\\&\qquad\leq 4\int_{\mathbb{R}^{1+n}} \left(1+\lvert (\tau-t,\xi-x)\rvert^2\right)\left\lvert \widehat{\mathcal{B}( X)}\right\rvert(\tau-t,\xi-x) \langle\lambda,(t,x)\rangle^{-2}\left\lvert\widehat{w}\right\rvert(t,x)~dxdt
    \end{aligned}
\end{align}
where $\langle\cdot,\cdot\rangle$ denotes the Japanese bracket.
Taking the $L^2(\mathbb{R}^{1+n})$ norm in the above expression and using Young's inequality, we obtain
\begin{align}
    \begin{aligned}
   \left\lVert     \widehat{\mathcal{B}( X)w}\right\rVert_{L^2(0,T;{H}^{-2}_\lambda(\mathbb{R}^n)}&\leq \left\lVert \left(1+\lvert (\tau-t,\xi-x)\rvert^2\right)\widehat{\mathcal{B}( X)}(\tau-t,\xi-x)  \right\rVert_{L^1(\mathbb{R}^{1+n})} \times\\&\qquad
   \left\lVert \langle\lambda,(t,x)\rangle^{-2}\widehat{w}(t,x)
   \right\rVert_{L^2(\mathbb{R}^{1+n})} 
  \\ &\leq C \lVert X\rVert_{H^r(\O_T)} \left(\int_{\mathbb{R}^{1+n}} \dfrac{1}{(1+\lvert (\tau,\xi)\rvert^2)^{r-2}}\right)^{1/2}\lVert 
  w
  \rVert_{L^2(0,T;{H}^{-2}_\lambda(\mathbb{R}^n)}.
    \end{aligned}
\end{align}
Using the above calculation in \eqref{3.23}, we obtain 
\begin{align}
    \begin{aligned}
        \lVert 
\mathcal{K}_{3}^-\overrightarrow{v}\rVert_{L^2(0,T;\mathbf{H}^{-2}_\lambda(\mathbb{R}^n)}^2 &\leq C \sum_{i,j=1}^k\bigg(\lVert \overrightarrow{M}_{ij}\rVert^2_{H^{r-1}(\O_T)}\lVert 
  \nabla_x v_i
  \rVert^2_{L^2(0,T;{H}^{-2}_\lambda(\mathbb{R}^n)}+\lVert p_{ij}\rVert^2_{H^{r-1}(\O_T)}\lVert 
 v_i
  \rVert^2_{L^2(0,T;{H}^{-2}_\lambda(\mathbb{R}^n)}\\&\qquad \qquad+\lambda^2\lVert \overrightarrow{M}_{ij}\rVert^2_{H^{r}(\O_T)}\lVert 
v_i
  \rVert^2_{L^2(0,T;{H}^{-2}_\lambda(\mathbb{R}^n)}\bigg)\\&\leq C\sum_{i,j=1}^k\lVert 
v_i
  \rVert^2_{L^2(0,T;{H}^{-1}_\lambda(\mathbb{R}^n)}.
    \end{aligned}
\end{align}
Now, we use the above estimate in Equation \eqref{3.22} to obtain
\begin{align}
    \begin{aligned}
\sigma^{1/2}\lVert\overrightarrow{v}\rVert_{L^2(0,T;\mathbf{H}^{-1}_{\lambda}(\mathbb{R}^n))} \leq C\left(\lVert \mathcal{K}^*_{\phi_{\sigma}^-}\overrightarrow{v}\rVert_{L^2(0,T;\mathbf{H}^{-2}_\lambda(\mathbb{R}^n)}+\lVert 
\overrightarrow{v}
\rVert_{L^2(0,T;\mathbf{H}^{-1}_\lambda(\mathbb{R}^n)}\right).
    \end{aligned}
\end{align}
For $\sigma$ large enough, the above estimate boils down to
\begin{align}
    \begin{aligned}
\sigma^{1/2}\lVert\overrightarrow{v}\rVert_{L^2(0,T;\mathbf{H}^{-1}_{\lambda}(\mathbb{R}^n))} \leq C\lVert \mathcal{K}^*_{\phi_{\sigma}^-}\overrightarrow{v}\rVert_{L^2(0,T;\mathbf{H}^{-2}_\lambda(\mathbb{R}^n)}.
    \end{aligned}
\end{align}
Next, by replacing $\overrightarrow{v}$ by $e^{\frac{\sigma((x+x_0)\cdot\o)^2}{2}}\overrightarrow{v}$ in the above expression and choose $\sigma$ large enough, we obtain
 \begin{align}
    \|\overrightarrow{v}\|_{L^2(0,T;{\mathbf{H}}^{-1}_\lambda(\mathbb{R}^n))} \leq C  \| \mathcal{K}^*_{\phi} \overrightarrow{v}\|_{L^2(0,T;\mathbf{H}^{-2}_{\lambda}(\mathbb{R}^n))}.
\end{align}
This completes the proof.
\end{proof}
\begin{remark}
In particular, for $(A,Q)\in \mathcal{A}_m$ and $\mathcal{L}^*_{A,Q}$ be given by Equation \eqref{Operator L^{*}_{A,Q}}.  Then for any $\overrightarrow{v} \in {C}^1([0,T];\mathbf{C}^{\infty}_c(\Omega))$ with $\overrightarrow{v}(T,x)=\overrightarrow{0}$, there exists a constant $C>0$, such that the estimate 
    \begin{align}\label{shifted by -1 carleman estimate for l phi}
\|\overrightarrow{v}\|_{L^2(0,T;\mathbf{H}^{-1}_\lambda(\mathbb{R}^n))} \leq C  \| \mathcal{L}^*_{\phi} \overrightarrow{v}\|_{L^2(0,T;\mathbf{H}^{-2}_{\lambda}(\mathbb{R}^n))}
\end{align}
holds $\lambda$ large enough.
\end{remark}
The weighted $H^{-1}-H^{-2}$ Carleman estimate derived above, combined with the Hahn-Banach extension theorem and the Riesz representation theorem, yields the following solvability result for the weighted operator associated with $\mathcal{K}_{M,N,P}$.
\begin{proposition}\label{Proposition 1}
Let $\mathcal{K}_{M,N,P}$ and $\phi$ be defined by
\eqref{Operator K_{M,N,P}} and \eqref{carleman weight},
respectively. Then, for every
$\overrightarrow{W}\in
L^{2}\bigl(0,T;\mathbf{H}^{1}(\mathbb{R}^{n})\bigr)$,
the problem
\begin{equation}\label{proposition 5.1}
\left\{
\begin{aligned}
e^{-\phi}\mathcal{K}_{M,N,P}e^{\phi}\overrightarrow{u}
&=
\overrightarrow{W},
&& (t,x)\in\Omega_T,\\
\overrightarrow{u}(0,x)
&=
\overrightarrow{0},
&& x\in\Omega,
\end{aligned}
\right.
\end{equation}
admits a solution
$\overrightarrow{u}\in \mathbf{H}^{2,1}(\Omega_T)$
satisfying
\begin{align*}   \|\overrightarrow{u}\|_{L^{2}
(0,T;\mathbf{H}^{2}_{\lambda}(\Omega))}
\le
C
\|\overrightarrow{W}\|_{L^{2}
(0,T;\mathbf{H}^{1}_{\lambda}(\Omega))},
\end{align*}
where $C>0$ is a constant depending only on
$M$, $N$, $P$, and $\Omega_T$.
\end{proposition}

\begin{proof}
Let $Y$ denotes the subspace  of $L^2(0,T;\mathbf{H}^{-2}_{\lambda}(\mathbb{R}^n))$, given by
\begin{align*}
Y:=\{\mathcal{K}_{\phi}^{*}\overrightarrow{v}(t,x): \overrightarrow{v} \in {C}^1([0,T];{\textbf{C}}^{\infty}_c(\Omega)) \text{ with }\overrightarrow{v}(T,x) = \overrightarrow{0}\}
\end{align*}
 and $\mathcal{K}_{\phi} := e^{\phi}\mathcal{K}_{M,N,P}e^{\phi}$. Next, we consider a linear operator $\mathcal{T}$ on $Y$, given by
\begin{align}\label{linear operator T}
 T(\mathcal{K}_{\phi}^{*}\overrightarrow{v}) = \int_{\O_T} \overrightarrow{v}(t,x)\cdot \overline{\overrightarrow{W}(t,x)}\ dx dt, 
\end{align}
where $\overrightarrow{W}\in L^2(0,T;\mathbf{H}^1(\mathbb{R}^n))$.
Now we use Equation \eqref{shifted by -1 carleman estimate} in the above operator to arrive at
\begin{align*}
|\mathcal{T}(\mathcal{K}_{\phi}^{*}\overrightarrow{v})| &\leq \lVert \overrightarrow{W}\rVert_{L^2(0,T;\mathbf{H}^1_\lambda(\O))}\lVert \overrightarrow{v}\rVert_{L^2(0,T;\mathbf{H}^{-1}_\lambda(\mathbb{R}^n))}\\&\leq C\lVert \overrightarrow{W}\rVert_{L^2(0,T;\mathbf{H}^1_\lambda(\O))}\| \mathcal{K}^*_{\phi} \overrightarrow{v}\|_{L^2(0,T;\mathbf{H}^{-2}_{\lambda}(\mathbb{R}^n))}
\end{align*}
holds for $\overrightarrow{v} \in {C}^1([0,T];{\mathbf{C}}^{\infty}_c(\Omega))$ satisfying $\overrightarrow{v}(T,x)=\overrightarrow{0}$. Thus, $\mathcal{T}$ is a continuous linear operator. Next, using the Hahn-Banach theorem, we extend  $\mathcal{T}$ to the full space $L^2(0,T;\mathbf{H}^{-2}_{\lambda}(\mathbb{R}^n))$ ( denoted by $\mathcal{T}$ again) and satisfies
\begin{align*}
  \|\mathcal{T}\| \leq C\lVert \overrightarrow{W}\rVert_{L^2(0,T;\mathbf{H}^1_\lambda(\O))}.
\end{align*}
Moreover, the Riesz representation theorem ensures the unique existence  $\overrightarrow{u}\in L^2(0,T;\mathbf{H}^{2}_{\lambda}(\mathbb{R}^n))$, such that
\begin{align}\label{5.23}
    \mathcal{T}(\overrightarrow{f}) = \langle \overrightarrow{f}, \overrightarrow{u}\rangle_{L^2(0,T;\mathbf{H}^{-2}_{\lambda}(\mathbb{R}^n)),L^2(0,T;\mathbf{H}^{2}_{\lambda}(\mathbb{R}^n))}, \ \text{for} \ \overrightarrow{f}\in L^2(0,T;\mathbf{H}^{-2}_{\lambda}(\mathbb{R}^n))
\end{align}
with
\begin{align*}
\|\overrightarrow{u}\|_{L^2(0,T;\mathbf{H}^{2}_{\lambda}(\mathbb{R}^n))} \leq C\lVert \overrightarrow{W}\rVert_{L^2(0,T;\mathbf{H}^1_\lambda(\O))}.
\end{align*}
Choosing $\overrightarrow{f} = \mathcal{K}_{\phi}^{*}\overrightarrow{v}$ for $\overrightarrow{v} \in {C}^1([0, T];{\mathbf{C}}^{\infty}_c(\Omega))$ with $\overrightarrow{v}(0,x)=\overrightarrow{v}(T,x)=\overrightarrow{0}$, in the above equation, and comparing with the linear operator \eqref{linear operator T}, we get $\mathcal{K}_{\phi}\overrightarrow{u} = \overrightarrow{W}$ in $\O_T$. 
Since $\overrightarrow{u}\in L^2(0,T;\mathbf{H}^{2}_{\lambda}(\mathbb{R}^n))$ and $\overrightarrow{W}\in \mathbf{L}^2(\O_T)$, from the expression for $\mathcal{K}_{\phi}$, we conclude that $\partial_t \overrightarrow{u}  \in \textbf{L}^2(\O_T)$. Therefore, we have $ \overrightarrow{u}\in \mathbf{H}^{2,1}(\O_T)$. Our next claim is that  $\overrightarrow{u}(0,x)=0$. Substituting $\overrightarrow{f} = \mathcal{K}_{\phi}^{*}\overrightarrow{v}$ into Equation \eqref{5.23} for $\overrightarrow{v} \in {C}^1([0, T];{\mathbf{C}}^{\infty}_c(\Omega))$ with $\overrightarrow{v}(T,x)=\overrightarrow{0}$, we have
\begin{align*}
\int_{\O_T}\overrightarrow{v}(t,x)\cdot\overrightarrow{W}(t,x)\ dx dt=\int_{\Omega_T} \mathcal{K}_{\phi}^{*}\overrightarrow{v}(t,x)\cdot \overrightarrow{u}(t,x)\ dx dt .
\end{align*}
Using integration by parts formula together with  $\mathcal{K}_{\phi}\overrightarrow{u} = \overrightarrow{W}$, we obtain
\begin{align*}
\int_{\O}\overrightarrow{u}(0,x)\cdot \overrightarrow{v}(0,x)\ dx = 0
\end{align*}
 for all $\overrightarrow{v}\in{C}^1([0, T];\textbf{C}^{\infty}_c(\Omega))$ satisfying $\overrightarrow{v}(T,x) = \overrightarrow{0}$. Since $\overrightarrow{v}(0,\cdot)\in \textbf{C}^{\infty}_c(\Omega)$, 
by choosing $\overrightarrow{v}= (0,  \cdots , v_i, \cdots, 0)^T$ in the above equation,  we get $u_i(0,x)=0$ for $x\in\O$ and $1\leq i \leq k$. Thus, we conclude that $\overrightarrow{u}(0,x) = \overrightarrow{0}$ for $x\in \O$.
\end{proof}
\begin{corollary}\label{remark for Proposition 1}
As an immediate consequence, for every $(A,Q)\in \mathcal{A}_m$, with $\mathcal{L}_{A,Q}$ defined by Equation \eqref{Operator L_{A,Q}}, and for every
$\overrightarrow{W}\in
L^{2}\bigl(0,T;\mathbf{H}^{1}(\mathbb{R}^{n})\bigr)$,
the problem
\begin{equation}\label{remark for proposition 5.1}
\left\{
\begin{aligned}
e^{-\phi}\mathcal{L}_{A,Q}e^{\phi}\overrightarrow{u}
&=
\overrightarrow{W},
&& (t,x)\in\Omega_T,\\
\overrightarrow{u}(0,x)
&=
\overrightarrow{0},
&& x\in\Omega,
\end{aligned}
\right.
\end{equation}
admits a solution
$\overrightarrow{u}\in \mathbf{H}^{2,1}(\Omega_T)$
satisfying
\begin{align*}   \|\overrightarrow{u}\|_{L^{2}
(0,T;\mathbf{H}^{2}_{\lambda}(\Omega))}
\le
C
\|\overrightarrow{W}\|_{L^{2}
(0,T;\mathbf{H}^{1}_{\lambda}(\Omega))},
\end{align*}
where $C>0$ is a positive constant depending only on
$m$ and $\Omega_T$.
\end{corollary}
An analogous solvability result holds for the adjoint operators
$\mathcal{K}_{M,N,P}^{*}$ and $\mathcal{L}^*_{A,Q}$. Indeed, by applying the weighted
$H^{-1}-H^{-2}$ Carleman estimate for the adjoint operator, together with the Hahn-Banach extension theorem and the Riesz representation theorem, one obtains the following solvability result.
\begin{proposition}
\label{Proposition 2}
Let $\mathcal{K}_{M,N,P}^{*}$ and $\phi$ be defined by
\eqref{Operator K^{*}_{A,Q}} and \eqref{carleman weight}, respectively.
Then, for any
$\overrightarrow{W}\in L^{2}(0,T;\mathbf{H}^{1}(\mathbb{R}^{n}))$,
the problem
\begin{align*}
\left\{
\begin{array}{rll}
e^{\phi}\mathcal{K}_{M,N,P}^{*}e^{-\phi}\overrightarrow{v}
&=
\overrightarrow{W},
& (t,x)\in \Omega_{T},\\
\overrightarrow{v}(T,x)
&=
\overrightarrow{0},
& x\in\Omega,
\end{array}
\right.
\end{align*}
admits a solution
$\overrightarrow{v}\in \mathbf{H}^{2,1}(\Omega_{T})$
satisfying
\begin{align}  \|\overrightarrow{v}\|_{L^{2}(0,T;
\mathbf{H}_{\lambda}^{2}(\Omega))}
\le
C
\|\overrightarrow{W}\|_{L^{2}(0,T;
\mathbf{H}_{\lambda}^{1}(\Omega))},
\end{align}
where $C>0$ is a constant depending only on
$M$, $N$, $P$, and $\Omega_{T}$. In particular, for $(A,Q)\in \mathcal{A}_m$  with $\mathcal{L}^*_{A,Q}$ defined by Equation \eqref{Operator L^{*}_{A,Q}}, and for every
$\overrightarrow{W}\in
L^{2}\bigl(0,T;\mathbf{H}^{1}(\mathbb{R}^{n})\bigr)$,
the problem \begin{align*}
\left\{
\begin{array}{rll}
e^{\phi}\mathcal{L}_{A,Q}^{*}e^{-\phi}\overrightarrow{v}
&=
\overrightarrow{W},
& (t,x)\in \Omega_{T},\\
\overrightarrow{v}(T,x)
&=
\overrightarrow{0},
& x\in\Omega,
\end{array}
\right.
\end{align*}
admits a solution
$\overrightarrow{v}\in \mathbf{H}^{2,1}(\Omega_{T})$
satisfying
\begin{align}  \|\overrightarrow{v}\|_{L^{2}(0,T;
\mathbf{H}_{\lambda}^{2}(\Omega))}
\le
C
\|\overrightarrow{W}\|_{L^{2}(0,T;
\mathbf{H}_{\lambda}^{1}(\Omega))},
\end{align}
where $C>0$, depends only on
$m$ and $\Omega_{T}.$
\end{proposition}

\section{Geometric optics solutions}\label{Construction of solutions}
This section is devoted to the construction of the geometric solutions for the operators $\mathcal{L}_{A,Q}$ and  $\mathcal{L}^*_{A,Q}$, where $\mathcal{L}^*_{A,Q}$ denote the
$L^2$-adjoint of $\mathcal{L}_{A,Q}$ and is given by
\begin{align}
 \label{Operator L^{*}_{A,Q}}
   \displaystyle \mathcal{L}^{*}_{A,Q}\overrightarrow{v} := \begin{bmatrix}
 -\partial_t v_1-  \sum_{j=1}^{n}(\partial_j- A^1_j)^2v_1+ \sum_{j=1}^k {q}_{j1} v_j
\\ 
-\partial_t v_2-  \sum_{j=1}^{n}(\partial_j- A^2_j)^2v_2+ \sum_{j=1}^k {q}_{j2} v_j\\
\vdots\\
-\partial_t v_k-  \sum_{j=1}^{n}(\partial_j- A^k_j)^2 v_k+ \sum_{j=1}^k {q}_{jk} v_j
\end{bmatrix}, \quad &(t,x)\in \Omega_T 
\end{align}  where $(A, Q)\in \mathcal{A}_m$ and $\overrightarrow{v}:= (v_1(t,x), v_2(t,x), \cdots, v_k(t,x))^T$. 
In compact form, $\mathcal{L}^*_{A,Q}$ can be written as 
\begin{align*}
    \mathcal{L}^*_{A,Q}\overrightarrow{v}=-\partial_t \overrightarrow{v}-\Delta_x \overrightarrow{v}+2A\cdot\nabla_x\overrightarrow{v}+(\nabla_x\cdot A)\overrightarrow{v}- A^2\overrightarrow{v}+Q\overrightarrow{v}.
\end{align*}
This section is divided into two subsections. The first subsection,
Subsection~\ref{subsec;exponentially growing solution}, is devoted to the construction of exponentially growing solutions for the operator $\mathcal{L}_{A,Q}$. The second subsection,
Subsection~\ref{subsec;exponentially decaying solution}, is concerned with the construction of exponentially decaying solutions for the adjoint operator $\mathcal{L}_{A,Q}^{*}$. 
The construction of exponentially growing solutions is a key ingredient in our analysis. Together with the corresponding exponentially decaying solutions for the adjoint operator, these solutions will be substituted into the integral identity established later to derive stability estimates for the convection coefficient and the matrix-valued potential.
\subsection{Construction of exponentially growing solution} \label{subsec;exponentially growing solution}
In this subsection, we construct exponentially growing solutions for the operator $\mathcal{L}_{A,Q}$. Let $\omega\in\mathbb{S}^{n-1}$ be fixed and consider the Carleman weight
\begin{align*}
    \phi(t,x)=\lambda^{2}t+\lambda x\cdot\omega,
\end{align*}
where $\lambda>0$ is a sufficiently large parameter. Our objective is to construct a solution
$\overrightarrow{u}\in\mathbf{H}^{2,1}(\Omega_T)$
of the equation
\begin{align*}
\mathcal{L}_{A,Q}\overrightarrow{u}=\overrightarrow{0}
\quad \text{in } \Omega_T,
\end{align*}
having the form
\begin{align*}
    \overrightarrow{u}(t,x)
=
e^{\phi(t,x)}
\left(
\overrightarrow{T_u}(t,x)
+
\overrightarrow{R_u}(t,x;\lambda)
\right),
\end{align*}
where $\overrightarrow{T_u}$ denotes the principal amplitude, while
$\overrightarrow{R_u}$ is a correction term that vanishes asymptotically as
$\lambda\to\infty$.
Substituting the above ansatz into the equation
$\mathcal{L}_{A,Q}\overrightarrow{u}=\overrightarrow{0}$
and collecting the terms of the highest order in the parameter $\lambda$, we obtain
\begin{align}\label{4.2}
\mathcal{L}_{A,Q}e^{\phi}\overrightarrow{R_u}
=
-e^{\phi}
\bigl(\mathcal{L}_{A,Q}-2\lambda L_u\bigr)
\overrightarrow{T_u},
\end{align}
where $L_u$ denotes the transport operator defined by
\begin{align}\label{transport equation}
L_u
:=
\omega\cdot\nabla_x
+\omega\cdot A.
\end{align}
To eliminate the leading-order term on the right-hand side of
Equation~\eqref{4.2}, we choose the principal amplitude
$\overrightarrow{T_u}$ to satisfy the transport equation
\(
   L_u\overrightarrow{T_u}
=
\overrightarrow{0}. 
\)
The construction of such an amplitude requires a suitable cut-off
near the initial and terminal times. For this purpose, let
$k\in\mathbb N$ and $0<\varrho\ll1$. We choose a family of smooth
functions 
   $ \{\Xi_\varrho\}_{\varrho>0}
\subset C_c^\infty((0,T))$
such that
\begin{align}\label{Xi estimate}
\Xi_\varrho(t)
\equiv1,
\quad
t\in[2\varrho,T-2\varrho], \quad \text{and} \quad
\|\Xi_\varrho\|_{W^{k,\infty}(\mathbb R)}
\le
C\varrho^{-k},
\end{align}
where the constant $C>0$ is independent of $\varrho$.
%%%%%%%%%%%%%%%%%%%%%%%%%%%%%%%%%%
Next, we construct the principal amplitude $\overrightarrow{T_u}$ as a solution of the transport equation \eqref{transport equation}. For our subsequent analysis, we consider two distinct solutions of the transport equation. These solutions play different roles in the proof of the main theorem: the first is used to establish the stability estimate for the convection coefficient $A$, while the second is employed to derive the stability estimate for the matrix-valued potential $Q$. 
\begin{enumerate}
    \item Let $\varrho\in (0,T/4)$ and $D_i\in W^{3,\infty}(\O_T,\mathbb{R}^n)$ satisfying $\lVert D_i\rVert_{W^{3,\infty}(\mathbb{R}^n)}\leq C$, for $1\leq i \leq k$. Then for $\tau\in \mathbb{R}$ and $\xi \in \mathbb{R}^n$ such that $\o\cdot \xi=0$, we define
\begin{align}\label{tu in stability of A}
    \begin{aligned}
        \overrightarrow{T_u}=\begin{bmatrix}
            \alpha_1\Xi_\varrho(t)\dfrac{\xi}{\lvert \xi\rvert}\cdot \nabla_x\left(e^{-\mathrm{i}(t\tau+x\cdot\xi)}e^{\int_{\mathbb{R}}\o\cdot D_1(t,x+s\o)ds}\right)e^{\int_{0}^{\infty} \o\cdot \overrightarrow{A}^1(t,x+s\o)ds}\\\vdots   \\ \alpha_k\Xi_\varrho(t)\dfrac{\xi}{\lvert \xi\rvert}\cdot \nabla_x\left(e^{-\mathrm{i}(t\tau+x\cdot\xi)}e^{\int_{\mathbb{R}}\o\cdot D_k(t,x+s\o)ds}\right)e^{\int_{0}^{\infty} \o\cdot \overrightarrow{A}^k(t,x+s\o)ds}
        \end{bmatrix}
    \end{aligned}
\end{align}
for any $\alpha_i\in \mathbb{R}$ for $1\leq i\leq k$. 
\item  Let $\varrho\in (0,T/4)$ and  $(\tau,\xi) \in \mathbb{R}^{1+n}$ with $\o\cdot\xi=0$.  Define
\begin{align}\label{tu in stability of Q}
    \begin{aligned}
        \overrightarrow{T_u}=\begin{bmatrix}
            \alpha_1\Xi_\varrho(t)e^{-\mathrm{i}(t\tau+x\cdot\xi)}e^{\int_{0}^{\infty} \o\cdot \overrightarrow{A}^1(t,x+s\o)ds}\\\vdots   \\ \alpha_k\Xi_\varrho(t)e^{-\mathrm{i}(t\tau+x\cdot\xi)}e^{\int_{0}^{\infty} \o\cdot \overrightarrow{A}^k(t,x+s\o)ds}
        \end{bmatrix}
    \end{aligned}
\end{align}
for any $\alpha_i\in \mathbb{R}$ for $1\leq i\leq k$.
\end{enumerate}
Using $L_u\overrightarrow{T_u}=\overrightarrow{0}$ in Equation \eqref{4.2}, we obtain
\begin{align}
   e^{-\phi} \mathcal{L}_{A,Q}e^{\phi}\overrightarrow{R_u}  = -\mathcal{L}_{A,Q}\overrightarrow{T_u}.
\end{align}
Now from Corollary \ref{remark for Proposition 1}, it follows that
 the problem
 \begin{align*}
\left\{
	\begin{array}{ l l l }
  e^{-\phi} \mathcal{L}_{A,Q}e^{\phi}\overrightarrow{R_u}  = -\mathcal{L}_{A,Q}\overrightarrow{T_u}, \ (t,x)\in\Omega_T,\\
     \overrightarrow{R_u}(0,x)=\overrightarrow{0}, \ x\in \Omega, 
\end{array}
\right.
\end{align*}
admits a solution $\overrightarrow{R_u} \in \mathbf{H}^{2,1}(\O_T)$ satisfying 
\begin{align*} \|\overrightarrow{R_u}\|_{L^2(0,T;\mathbf{H}^2_{\lambda}(\Omega))}&\leq  \|\mathcal{L}_{A,Q}\overrightarrow{T_u}\|_{L^2(0,T;\mathbf{H}^1_{\lambda}(\Omega))}.
\end{align*}
By a direct computation, one can see that if $\overrightarrow{T_u}$ is given by \eqref{tu in stability of A}, we have 
\begin{align*} \|\overrightarrow{R_u}\|_{L^2(0,T;\mathbf{H}^2_{\lambda}(\Omega))}\leq  \|\mathcal{L}_{A,Q}\overrightarrow{T_u}\|_{L^2(0,T;\mathbf{H}^1_{\lambda}(\Omega))}\leq C\lambda \varrho^{-1}\langle \tau, \xi\rangle_*^4
\end{align*}
where $\langle \tau, \xi\rangle_*$ is a notational convention for $\sqrt{1+\tau^2+|\xi|^2}$. 
We conclude this subsection with the following lemma.  
\begin{lemma}\label{lemma exponential growing solution}{(Exponentially growing solution)}
 Let $\mathcal{L}_{A,Q}$ and $\phi$ be as defined in Equations \eqref{Operator L_{A,Q}} and \eqref{carleman weight}, respectively. Assume that $\varrho\in (0,T/4)$ and that $D_i\in W^{3,\infty}(\Omega_T,\mathbb{R}^n)$ satisfy
 \begin{align*}
     \|D_i\|_{W^{3,\infty}(\mathbb{R}^n)}\leq C, \qquad 1\leq i\leq k.
 \end{align*}
Then, for any $\tau\in \mathbb{R}$ and $\xi\in \mathbb{R}^n$ such that $\omega\cdot \xi=0$, there exists a sufficiently large parameter $\lambda>0$ such that the initial value problem
\begin{align}
\left\{
\begin{array}{rll}
\mathcal{L}_{A,Q}\overrightarrow{u} & = \overrightarrow{0}, & (t,x)\in \Omega_T,\\[0.2em]
\overrightarrow{u}(0,x) & = \overrightarrow{0}, & x\in \Omega,
\end{array}
\right.
\end{align}
admits a solution $\overrightarrow{u}\in \mathbf{H}^{2,1}(\Omega_T)$ of the form
\begin{align}\label{u solution form}
\overrightarrow{u}(t,x) = e^{\phi(t,x)}\big(\overrightarrow{T_u}(t,x) + \overrightarrow{R_u}(t,x;\lambda)\big).
\end{align}
We specify below two possible choices for $\overrightarrow{T_u}$ along with corresponding bounds for the correction terms.

\begin{enumerate}
    \item First, define
    \begin{align}\label{solution of transport equation}
        \overrightarrow{T_u}(t,x)=
        \begin{bmatrix}
        \displaystyle
        \alpha_1\,\Xi_\varrho(t)\,\frac{\xi}{\lvert \xi\rvert}\cdot \nabla_x\!\left(e^{-\mathrm{i}(t\tau+x\cdot\xi)}
        e^{\int_{\mathbb{R}}\omega\cdot D_1(t,x+s\omega)\,ds}\right)
        e^{\int_{0}^{\infty} \omega\cdot \overrightarrow{A}^1(t,x+s\omega)\,ds}\\[0.9em]
        \vdots\\[0.3em]
        \displaystyle
        \alpha_k\,\Xi_\varrho(t)\,\frac{\xi}{\lvert \xi\rvert}\cdot \nabla_x\!\left(e^{-\mathrm{i}(t\tau+x\cdot\xi)}
        e^{\int_{\mathbb{R}}\omega\cdot D_k(t,x+s\omega)\,ds}\right)
        e^{\int_{0}^{\infty} \omega\cdot \overrightarrow{A}^k(t,x+s\omega)\,ds}
        \end{bmatrix},
    \end{align}
    and the corresponding correction term $\overrightarrow{R_u}(t,x;\lambda)$ satisfy
    \begin{align*}
        \|\overrightarrow{R_u}\|_{L^2\big(0,T;\mathbf{H}^{k}(\Omega)\big)}
        \leq C\,\lambda^{-1+k}\,\varrho^{-1}\,\langle \tau,\xi\rangle_*^{4},\ \text{ for }k\in\{0,1,2\}.
    \end{align*}

    \item Alternatively, define
    \begin{align}
        \overrightarrow{T_u}(t,x)=
        \begin{bmatrix}
        \alpha_1\,\Xi_\varrho(t)\,e^{-\mathrm{i}(t\tau+x\cdot\xi)}
        e^{\int_{0}^{\infty} \omega\cdot \overrightarrow{A}^1(t,x+s\omega)\,ds}\\
        \vdots\\
        \alpha_k\,\Xi_\varrho(t)\,e^{-\mathrm{i}(t\tau+x\cdot\xi)}
        e^{\int_{0}^{\infty} \omega\cdot \overrightarrow{A}^k(t,x+s\omega)\,ds}
        \end{bmatrix},
    \end{align}
    and the corresponding correction term $\overrightarrow{R_v}(t,x;\lambda)$ satisfy
    \begin{align}
\|\overrightarrow{R_v}\|_{L^2\big(0,T;\mathbf{H}^{k}(\Omega)\big)}
        \leq C\,\lambda^{-1+k}\,\varrho^{-1}\,\langle \tau,\xi\rangle_*^{3},\ \text{ for }k\in\{0,1,2\}.
    \end{align}
\end{enumerate}
These constructions hold for each choice of coefficients $\alpha_i\in \mathbb{R}$, $1\leq i\leq k$.
\end{lemma}
\subsection{Construction of exponentially decaying solution}\label{subsec;exponentially decaying solution}
In this subsection, we construct exponentially decaying solutions for the operator $\mathcal{L}^*_{A,Q}$.
Let $\phi$ and $\mathcal{L}^*_{A,Q}$  be as defined by Equations \eqref{carleman weight} and \eqref{Operator L^{*}_{A,Q}}  respectively. Our aim is to find a solution $\overrightarrow{v}$ to $\mathcal{L}^*_{A,Q}\overrightarrow{v}=\overrightarrow{0}$, such that $\overrightarrow{v}\in \mathbf{H}^{2,1}(\O_T)$ and is of the form 
\begin{align*}
     \overrightarrow{v} = e^{-\phi(t,x)}(\overrightarrow{T_v}(t,x) + \overrightarrow{R_v}(t,x;\lambda)),
\end{align*}
where $\overrightarrow{R_v}$ stands for the correction term. The substitution will give
\begin{align}\label{4.4}
\mathcal{L}^*_{A,Q}e^{-\phi}\overrightarrow{R_v}  = -e^{-\phi}(\mathcal{L}^*_{A,Q}+2\lambda L)\overrightarrow{T_v},
\end{align}
where $L$ denotes the transport operator and is given by
\(
    L_v = \o\cdot \nabla_x-\o\cdot A.
\)
Now, we look for $\overrightarrow{T_v}$ satisfying $L_v\overrightarrow{T_v} = \overrightarrow{0}$ and is of the form 
\begin{align}
\label{T tilde vector}
\overrightarrow{T_v} = \begin{pmatrix}
            \widetilde{\alpha}_1\Xi_\varrho(t)e^{-\int_{0}^{\infty} \o\cdot \overrightarrow{A}^1(t,x+s\o)ds}, \cdots  , \widetilde{\alpha}_k\Xi_\varrho(t)e^{-\int_{0}^{\infty} \o\cdot \overrightarrow{A}^k(t,x+s\o)ds}
        \end{pmatrix}
\end{align}
 for any $\widetilde{\alpha}_i \in \mathbb{R}$ for $\ 1\leq i \leq k$.
Using $L_v \overrightarrow{T_v}=\overrightarrow{0}$ in  Equation~\eqref{4.4}, we have 
\begin{align}
   e^{\phi} \mathcal{L}^*_{A,Q}e^{-\phi}\overrightarrow{R_v}  = -\mathcal{L}^*_{A,Q}\overrightarrow{T_v}.
\end{align}
Now from Proposition \ref{Proposition 2}, it follows that
 the problem
 \begin{align*}
\left\{
	\begin{array}{ r l l }
  e^{\phi} \mathcal{L}^*_{A,Q}e^{-\phi}\overrightarrow{R_v} & = -\mathcal{L}^*_{A,Q}\overrightarrow{T_v},& \ (t,x)\in\Omega_T,\\
     \overrightarrow{R_v}(T,x)&=\overrightarrow{0}, &\ x\in \Omega, 
\end{array}
\right.
\end{align*}
admits a solution $\overrightarrow{R_v} \in \mathbf{H}^{2,1}(\O_T)$ satisfying 
\begin{align*} \|\overrightarrow{R_v}\|_{L^2(0,T;\mathbf{H}^2_{\lambda}(\Omega))}\leq  \|\mathcal{L}^*_{A,Q}\overrightarrow{T_v}\|_{L^2(0,T;\mathbf{H}^1_{\lambda}(\Omega))}\leq C\lambda \varrho^{-1}.
\end{align*}
We conclude this subsection with the following lemma.  
\begin{lemma}\label{lemma exponential decaying solution}{(Exponentially decaying solution)}
  Let $\phi$ and
  $\mathcal{L}^*_{A,Q}$ are as defined by Equations \eqref{carleman weight} and \eqref{Operator L^{*}_{A,Q}} respectively. Then for $\varrho\in (0,T/4)$ there exists a
sufficient large parameter $\lambda>0$ such that the backward problem
 \begin{align}
     \left\{
	\begin{array}{ r l l }
\mathcal{L}^*_{A,Q}\overrightarrow{v}& = \overrightarrow{0},& \ (t,x)\in\Omega_T,\\
    \overrightarrow{v}(T,x)& = \overrightarrow{0},& \ x\in \Omega, 
\end{array}
\right.
 \end{align}
admits a solution $\overrightarrow{v}\in {{\mathbf{H}}^{2,1}(\O_T)}$ of the form 
\begin{align}
\label{solution form for adjoint operator}
   \overrightarrow{v} = e^{-\phi(t,x)}(\overrightarrow{T_v}(t,x) + \overrightarrow{R_v}(t,x;\lambda)),
\end{align}
 where \begin{align}
 \begin{aligned}
\overrightarrow{T_v} = \begin{bmatrix}
        \widetilde{\alpha}_1\Xi_\varrho(t)e^{-\int_{0}^{\infty} \o\cdot \overrightarrow{A}^1(t,x+s\o)ds}\\ \vdots \\  \widetilde{\alpha}_k\Xi_\varrho(t)e^{-\int_{0}^{\infty} \o\cdot \overrightarrow{A}^k(t,x+s\o)ds}
        \end{bmatrix}, 
        \end{aligned}
 \end{align}
 and $\overrightarrow{R_v}(t,x;\lambda)$ satisfies
\begin{align*} \|\overrightarrow{R_v}\|_{L^2(0,T;\mathbf{H}^k(\Omega))}\leq C\lambda^{-1+k}\varrho^{-1},\text{ for } k\in \{0,1,2\},
\end{align*}
 for any $\widetilde{\alpha}_i\in \mathbb{R}$, $1\leq i\leq k$. 
\end{lemma}
\section{Proof of Theorem \ref{main theorem}}\label{Proof of main theorem}
In this section, we present the proof of Theorem~\ref{main theorem}, which constitutes the main result of this article. The proof relies on an integral identity obtained by combining solutions of the IBVP with those of its corresponding adjoint problem. This identity relates the unknown coefficients to the available partial boundary measurements through the partial DN map.
After establishing the integral identity, we employ the boundary Carleman estimate together with the geometric optics solutions constructed in Section~\ref{Construction of solutions} to recover the unknown coefficients. More precisely, we first derive a stability estimate for the time-dependent convection coefficient. This estimate is then used to establish the stability estimate for the matrix-valued potential, thereby completing the proof of Theorem~\ref{main theorem}.
To this end, let
$\overrightarrow{u}^{(\ell)}\in \mathbf{H}^{2,1}(\Omega_T),
\,
\ell=1,2,$
be the solution of the following IBVP:
    \begin{equation}{\label{PDE corresponding to u_l}}
\left\{\begin{array}{r l c}
\displaystyle \mathcal{L}_{A_{(\ell)},Q_{(\ell)}}\overrightarrow{u}^{(\ell)}(t,x) &= \overrightarrow{0},\quad& (t,x) \in \Omega_T,\\
\overrightarrow{u}^{(\ell)}(0,x)&= \overrightarrow{0}, \quad& x \in \Omega  ,\\
\overrightarrow{u}^{(\ell)}(t,x)&= \overrightarrow{F}(t,x),\quad& (t,x) \in \partial \Omega_T ,
\end{array}\right.
\end{equation}
 where $A_{(\ell)}$ and $Q_{(\ell)}$ are as in the statement of  Theorem \ref{main theorem}. Also, we denote the partial DN map correspond to IBVP~\eqref{PDE corresponding to u_l} by $\Lambda^{\sharp}_{\ell}$ and are given by
 \begin{align}\label{partial dn map for u_l}
 \Lambda^{\sharp}_{\ell}(\overrightarrow{F)}:=\left(\partial_\nu \overrightarrow{u}^{(\ell)} + 2\left(\nu\cdot A\right) \overrightarrow{u}^{(\ell)}\right)\big|_{(0,T)\times \widetilde{F}(\o_0)}
 \end{align}
 for each $\overrightarrow{F}\in \mathbf{H}^{\frac{3}{2},\frac{3}{4}}(\partial \O_T)$.
Next, we  denote
\begin{equation}\label{simplified notation of all vectors}
	\begin{array}{ r l }
 \overrightarrow{u}(t,x) &:= \overrightarrow{u}^{(1)}(t,x) - \overrightarrow{u}^{(2)}(t,x),\\
     A(t,x) &:= A_{(1)}(t,x) - A_{(2)}(t,x), \\
	\widetilde{Q}_{(\ell)}(t,x) &:= - \nabla_x \cdot A_{(\ell)}(t,x)  - (A_{(\ell)}(t,x))^2 + Q_{(\ell)}(t,x),   \\
	\widetilde{Q}(t,x) &:= \widetilde{Q}_{(2)}(t,x) - \widetilde{Q}_{(1)}(t,x).
	\end{array}
 \end{equation}
Now  notice  that  $\overrightarrow{u}\in  \mathbf{H}^{2,1}(\O_T)$  satisfies the following IBVP
\begin{equation}\label{L A1 Q1 in term of A U2}
  \left\{
	\begin{array}{ r l c }
    \mathcal{L}_{A_{(1)},Q_{(1)}}\overrightarrow{u} &= 2A\cdot\nabla_x \overrightarrow{u}^{(2)} + \widetilde{Q}\overrightarrow{u}^{(2)}, \quad &(t,x)\in \Omega_T, \\
	 \overrightarrow{u}(0,x) &=\overrightarrow{0},  \quad &x \in \Omega,  \\
	\overrightarrow{u}(t,x) &= \overrightarrow{0}, \quad &(t,x) \in \partial\Omega_T.
	\end{array}
	\right.
 \end{equation}
Clearly, $2A\cdot\nabla_x \overrightarrow{u}^{(2)} + \widetilde{Q}\overrightarrow{u}^{(2)} \in \textbf{L}^2(\Omega_T)$, therefore using Theorem \ref{existence in forward pde}, we obtain  that there exists a unique solution  $\displaystyle \overrightarrow{u}\in L^{2}(0,T;\mathbf{H}^2(\Omega))\cap H^1(0,T;\mathbf{L}^2(\Omega))$ of IBVP \eqref{L A1 Q1 in term of A U2} with $\partial_\nu \overrightarrow{u}\in L^2(0,T;\mathbf{H}^{1/2}(\partial\Omega)).$
Next, assume that $\overrightarrow{v}^{(1)} \in \mathbf{H}^{2,1}(\O_T)$ is  a solution to the following backward problem: 
\begin{align} \label{adjoint operator with condition on T=0}
   \left\{\begin{array}{r l c}
    \mathcal{L}^*_{A_{(1)},Q_{(1)}} \overrightarrow{v}^{(1)}(t,x)&= \overrightarrow{0}, &(t,x)\in \Omega_T,  \\
      \overrightarrow{v}^{(1)}(T,x) &=\overrightarrow{0},  & x \in \O.
   \end{array}\right.
\end{align}
Now multiply Equation  \eqref{L A1 Q1 in term of A U2} with  $\overline{\overrightarrow{v}^{(1)}}(t,x)$ and  integrating over $\Omega_T$, we have
\begin{align*}
 \int_{\Omega_T} \left(2A\cdot\nabla_x \overrightarrow{u}^{(2)}(t,x) + \widetilde{Q}\overrightarrow{u}^{(2)}(t,x)\right)\cdot \overline{\overrightarrow{v}^{(1)}}(t,x)\ dx dt= \int_{\Omega_T}  \mathcal{L}_{A_{(1)},Q_{(1)}}\overrightarrow{u}(t,x) \cdot \overline{\overrightarrow{v}^{(1)}}(t,x)\ dx dt .
\end{align*}
After using the  integration by parts, together with the fact that $\overrightarrow{u}|_{\partial\Omega_T} = 0, \overrightarrow{u}(0,\cdot) = \overrightarrow{0}$, $A$ vanishing at ${\partial\Omega_T}$ and $\overrightarrow{v}^{(1)}$ satisfies Equation \eqref{adjoint operator with condition on T=0},  the above equation boils down to
 \begin{align}
 \label{equation 54}
 \displaystyle  \underbrace{\int_{\Omega_T} \left(2A\cdot\nabla_x \overrightarrow{u}^{(2)} + \widetilde{Q}\overrightarrow{u}^{(2)}\right)\cdot\overline{\overrightarrow{v}^{(1)}}(t,x)\ dx dt}_{J_1} = - \underbrace{\int_{\partial\O_T}  \partial_{\nu}\overrightarrow{u}(t,x) \cdot\overline{\overrightarrow{v}^{(1)}}(t,x)\ dS_x dt}_{J_2}.
 \end{align}
 \subsection{Stability result for convection term}
In this subsection, we derive the stability estimate for the convection term $A$ using the integral identity given by \eqref{equation 54} and the geometric optics solutions. In particular, we shall start with 
 choosing the GO solutions 
\begin{align}\label{go solution for u2 and v1}
\begin{aligned}
      \overrightarrow{u}^{(2)} = e^{\phi(t,x)}(\overrightarrow{T_u}^{(2)}(t,x) + \overrightarrow{R_u}^{(2)}(t,x;\lambda))\text{  and }
         \overrightarrow{v}^{(1)} = e^{-\phi(t,x)}(\overrightarrow{T_v}^{(1)}(t,x) + \overrightarrow{R_v}^{(1)}(t,x;\lambda)),
         \end{aligned}
\end{align}
where 
\begin{align}\label{tu and tv}
    \begin{aligned}
    \overrightarrow{T_u}^{(2)}&=\begin{bmatrix}
        \alpha_1\Xi_\varrho(t)\left(\dfrac{\xi}{\lvert \xi\rvert}\cdot \nabla_x\Psi_1 \right)e^{\int_{0}^{\infty} \o\cdot \overrightarrow{A}^1_{(2)}(t,x+s\o)ds}\\\vdots   \\ \alpha_k\Xi_\varrho(t)\left(\dfrac{\xi}{\lvert \xi\rvert}\cdot \nabla_x\Psi_k\right)e^{\int_{0}^{\infty} \o\cdot \overrightarrow{A}^k_{(2)}(t,x+s\o)ds}
        \end{bmatrix},\quad 
        \overrightarrow{T_v}^{(1)} = \begin{bmatrix}
        \widetilde{\alpha}_1\Xi_\varrho(t)e^{-\int_{0}^{\infty} \o\cdot \overrightarrow{A}_{(1)}^1(t,x+s\o)ds}\\ \vdots  \\ \widetilde{\alpha}_k\Xi_\varrho(t)e^{-\int_{0}^{\infty} \o\cdot \overrightarrow{A}_{(1)}^k(t,x+s\o)ds}
        \end{bmatrix},
    \end{aligned}
\end{align}
and \(\Psi_i=e^{-\mathrm{i}(t\tau+x\cdot\xi)}e^{\int_{\mathbb{R}}\o\cdot (\overrightarrow{A}_{(1)}^i-\overrightarrow{A}_{(2)}^i)(t,x+s\o)ds}\) for each $1\leq i\leq k$. Also, the correction terms $\overrightarrow{R_u}^{(2)}(t,x;\lambda)$ and
$\overrightarrow{R_v}^{(1)}(t,x;\lambda)$ satisfies the following inequalities
\begin{align} \label{ru and rv bounds}
\begin{aligned}
\|\overrightarrow{R_u}^{(2)}\|_{L^2(0,T;\mathbf{H}^{k}(\O))}&\leq  C\lambda^{-1+k}\varrho^{-1}\langle \tau,\xi\rangle_*^4\quad\text{ and }\quad \|\overrightarrow{R_v}^{(1)}\|_{L^2(0,T;\mathbf{H}^k(\Omega))}&\leq C\lambda^{-1+k}\varrho^{-1},
\end{aligned}
\end{align}
for $ k\in \{0,1,2\}$ and $\langle \tau, \xi\rangle_*:=\sqrt{1+\tau^2+|\xi|^2}$.
Now, substitute the aforementioned GO solution given by \eqref{go solution for u2 and v1}
in the expression of $J_1$, to get
\begin{align}\label{j1 expression}
\begin{aligned}
    J_1=&
    2\int_{\O_T}\sum_{i=1}^{k} \left( \left(\overrightarrow{A}_{(1)}^{i}- \overrightarrow{A}_{(2)}^{i} \right)
    \cdot\left((\nabla_x {T_u}^{(2)})_i\right)\right)\left(( {T_v}^{(1)})_i +\overline{( R_v^{(1)})_i}\right)\ dx dt
    \\&+
    2\int_{\O_T}\sum_{i=1}^{k} \left(\left(\overrightarrow{A}_{(1)}^{i}- \overrightarrow{A}_{(2)}^{i} \right)\cdot (\nabla_x R_u^{(2)})_i\right) \left(( {T_v}^{(1)})_i +\overline{(R_v^{(1)})_i}\right)~   dx dt
    \\&+2\int_{\O_T}\sum_{i=1}^{k} \left(\lambda\left(\o\cdot (\overrightarrow{A}_{(1)}^{i}- \overrightarrow{A}_{(2)}^{i}) \right)(T_u^{(2)})_i\right) \left(( {T_v}^{(1)})_i \right)\ dx dt  \\&+2\int_{\O_T}\sum_{i=1}^{k} \left(\lambda\left(\o\cdot (\overrightarrow{A}_{(1)}^{i}- \overrightarrow{A}_{(2)}^{i}) \right)(T_u^{(2)})_i\right) \left(( \overline{(R_v^{(1)})_i}\right)\ dx dt
    \\&+2\int_{\O_T}\sum_{i=1}^{k} \left(\lambda\left(\o\cdot (\overrightarrow{A}_{(1)}^{i}- \overrightarrow{A}_{(2)}^{i}) \right)(R_u^{(2)})_i\right) \left(( {T_v}^{(1)})_i +\overline{(R_v^{(1)})_i}\right)\ dx dt
    \\&+\int_{\Omega_T} \sum_{i,j=1}^{k} \widetilde{q}_{ji} \left(
    (T_u^{(2)})_i+(R_u^{(2)})_i
    \right)
    \left(( {T_v}^{(1)})_j +\overline{( R_v^{(1)})_j}\right)\ dx dt
\\:\triangleq &~ J_{1,1}+J_{1,2}+J_{1,3}+J_{1,4}+J_{1,5}+J_{1,6},
\end{aligned}
\end{align}
where $A=\diag(\overrightarrow{A}^1_{(1)}-\overrightarrow{A}^1_{(2)}, \cdots, \overrightarrow{A}^k_{(1)}-\overrightarrow{A}^k_{(2)}):=\diag(\overrightarrow{A}^1,\cdots,\overrightarrow{A}^k)$, 
$\overrightarrow{T_u}^{(2)}:=\left((T_u^{(2)})_1,\cdots, (T_u^{(2)})_k \right)$ and $\overrightarrow{R_u}^{(2)}:=\left((R_u^{(2)})_1,\cdots, (R_u^{(2)})_k \right)$.
Next, we first simplify the term $J_{1,3}$.  We make use of the orthogonal decomposition
   $\mathbb{R}^{n}
=
\mathbb{R}\omega
\oplus
\omega^{\perp}$,
where $\omega^{\perp}$ denotes the hyperplane orthogonal to $\omega$. Accordingly, every point $x\in\mathbb{R}^{n}$ can be uniquely written as
\(
   x=\ell+r\omega,\) 
where \(\ell\in\omega^{\perp},\ 
r\in\mathbb{R}. 
\)
Using this decomposition together with the change of variables
$(\ell,r)\mapsto x$, we obtain

\begin{align}\label{j13 value}
    \begin{aligned}
 J_{1,3}&=
2\int_{\O_T}\sum_{i=1}^{k} \left(\lambda\o\cdot \overrightarrow{A}^{i} \right)(T_u^{(2)})_i ({T_v}^{(1)})_i \ dx dt  \\&=2\int_{\mathbb{R}^{1+n}}\sum_{i=1}^{k} \left(\lambda\o\cdot \overrightarrow{A}^{i} \right)(T_u^{(2)})_i ( {T_v}^{(1)})_i \ dx dt  \\& =2\lambda\int_{\mathbb{R}^{1+n}}\sum_{i=1}^{k} (\o\cdot \overrightarrow{A}^{i}) \alpha_i \widetilde{\alpha}_i \Xi^2_\varrho(t)\dfrac{\xi}{\lvert \xi\rvert}\cdot \nabla_x\left(e^{-\mathrm{i}(t\tau+x\cdot\xi)}e^{\int_{\mathbb{R}}\o\cdot \overrightarrow{A}^i(t,x+s\o)ds}\right)e^{-\int_{0}^{\infty} \o\cdot \overrightarrow{A}^i(t,x+s\o)ds}~dxdt\\&=2\lambda \sum_{i=1}^{k} \alpha_i \widetilde{\alpha}_i \int_{\mathbb{R}} \Xi^2_\varrho(t) e^{-\mathrm{i} t \tau} \int_{\omega^{\perp}} 
  \int_{\mathbb{R}}  (\o\cdot \overrightarrow{A}^{i})(t, \ell+r\o)\dfrac{\xi}{\lvert \xi\rvert}\cdot \nabla_x\left(e^{-\mathrm{i}\ell\cdot\xi}e^{\int_{\mathbb{R}}\o\cdot \overrightarrow{A}^i(t,\ell+s\o)ds}\right)\times \\& \hspace{6cm} \left(e^{-\int_{r}^{\infty} \o\cdot \overrightarrow{A}^i(t,\ell+s\o)ds}\right) ~dr d\ell dt,
    \end{aligned}
\end{align}
where $d\ell$ denotes the Lebesgue measure on the hyperplane $\omega^{\perp}$ and  $\xi\in \o^{\perp}$. To simplify the above expression further, we introduce the notation
$\Phi:=e^{-\int_{r}^{\infty} \omega\cdot \overrightarrow{A}^{i}(t,\ell+s\omega) ds}$.
Differentiating $\Phi$ with respect to the variable $r$, we obtain 
 $\partial_{r}\Phi= \o \cdot \overrightarrow{A}^{i}(t,\ell+r \o) \Phi$. Consequently, we have
\begin{align*}
\int_{\mathbb{R}} \partial_{r}\Phi \ dr=\int_{\mathbb{R}} \omega\cdot \overrightarrow{A}^{i}(t,\ell+r \omega) \Phi dr=1-  e^{ -\int_{\mathbb{R}} \omega\cdot \overrightarrow{A}^{i}(t,\ell+r\omega) dr}.
\end{align*}
The use of above expression in Equation~\eqref{j13 value} yields to
\begin{align}
       J_{1,3}&= 2\lambda \sum_{i=1}^{k} \alpha_i \widetilde{\alpha}_i \int_{\mathbb{R}} \Xi^2_\varrho(t) e^{-\mathrm{i} t \tau} \int_{\omega^{\perp}} 
\dfrac{\xi}{\lvert \xi\rvert}\cdot \nabla_x\left(e^{-\mathrm{i}\ell\cdot\xi}e^{\int_{\mathbb{R}}\o\cdot \overrightarrow{A}^i(t,\ell+s\o)ds}\right)(1-  e^{ -\int_{\mathbb{R}} \omega\cdot \overrightarrow{A}^{i}(t,\ell+r\omega) dr}) \,d\ell dt\\&  =-2\lambda \sum_{i=1}^{k} \alpha_i \widetilde{\alpha}_i \int_{\mathbb{R}} \Xi^2_\varrho(t) e^{-\mathrm{i} t \tau} \int_{\omega^{\perp}}e^{-\mathrm{i} \xi\cdot \ell}\dfrac{\xi}{\lvert \xi\rvert} \cdot \nabla^{\perp}\left(\int_{\mathbb{R}} \o\cdot  \overrightarrow{A}^i(t,\ell+s\o)~ds\right)\,d\ell dt\\&=2\lambda \sum_{i=1}^{k} \alpha_i \widetilde{\alpha}_i \int_{\mathbb{R}} \Xi^2_\varrho(t) e^{-\mathrm{i} t \tau}  \int_{\mathbb{R}^n}e^{-\mathrm{i}\xi\cdot x}\dfrac{\xi}{\lvert \xi\rvert}\cdot \nabla_x\left( \o\cdot  \overrightarrow{A}^i(t,x)\right)\,dx  dt,
\end{align}
where $\nabla g:=\nabla^{\perp} g+\o(\o\cdot \nabla g)$. 
Using the properties of the Fourier transform (space-time variable) along with the above expression, we have
\begin{align}\label{j13 first value}
J_{1,3}=\mathrm{i}2\lambda \sum_{i=1}^{k} \alpha_i \widetilde{\alpha}_i \lvert \xi\rvert  \widehat{\Xi^2_\varrho\o\cdot  \overrightarrow{A}^i}(\tau,\xi), \quad \text{for $\xi\in \o^{\perp}$ and $\alpha_i,\widetilde{\alpha_i}\in \mathbb{R}$ $(1\leq i\leq k)$}.
\end{align}
Next, we estimate the remaining terms appearing in Equation~\eqref{j1 expression}. Using the estimates for the correction terms together with the geometric optics solutions, we find that, for sufficiently large $\lambda$,
\begin{align}\label{bound for jij terms}
    \begin{aligned}
       \lvert J_{1,1}\rvert \leq C \langle \tau, \xi\rangle_*^2(1+\varrho^{-1}\lambda^{-1}), \quad  \lvert J_{1,2}\rvert \leq C \varrho^{-1} \langle \tau, \xi\rangle_*^4(1+\varrho^{-1}\lambda^{-1}),\quad
        \lvert J_{1,4}\rvert \leq C  \varrho^{-1}\langle \tau, \xi\rangle_*,\\  \lvert J_{1,5}\rvert\leq C \varrho^{-1}\langle \tau, \xi\rangle_*^4(1+\lambda^{-1} \varrho^{-1}),\quad \lvert J_{1,6}\rvert \leq C(1+\varrho^{-2}) \langle \tau, \xi\rangle_*^4.
    \end{aligned}
\end{align} 
Also, using  Equation \eqref{DN map} and the fact that $A_1=A_2$ on $\partial\O_T$, we obtain 
\begin{align}\label{j2 in sum of front and back}
  & \left\langle\left(\Lambda_{A_{(1)},Q_{(1)}} - \Lambda_{A_{(2)},Q_{(2)}}\right)(\overrightarrow{F}), \overrightarrow{v}^{(1)}|_{\partial\Omega_T}\right\rangle  =  \int_{\partial\Omega_T} \partial_{\nu}\overrightarrow{u}(t,x) \cdot \overline{\overrightarrow{v}^{(1)}}(t,x)\ dS_x dt~ ( =J_2)  \\&\qquad=  \underbrace{\int_{(0,T)\times \widetilde{F}(\o_0)} \partial_{\nu}\overrightarrow{u}(t,x) \cdot \overline{\overrightarrow{v}^{(1)}}(t,x)\ dS_x dt}_{J_{2,1}} + \underbrace{\int_{(0,T)\times \widetilde{B}(\o_0)} \partial_{\nu}\overrightarrow{u}(t,x) \cdot \overline{\overrightarrow{v}^{(1)}}(t,x)\ dS_x dt }_{J_{2,2}}.
\end{align}
By the definitions of the illuminated and shadowed parts of the boundary, for every $\omega\in\mathbb{S}^{n-1}$ satisfying
$|\omega-\omega_0|\le \frac{\varepsilon}{2},$
we have
\begin{align}\label{front and back face relation}
\begin{aligned}
(0,T)\times\widetilde{F}(\omega_0)
\supseteq
(0,T)\times F(\omega)\quad\text{and} \quad
(0,T)\times B(\omega)
\supseteq
(0,T)\times\widetilde{B}(\omega_0).
\end{aligned}
\end{align}
{After applying trace theorem, there exists $\vartheta_1>0$ such that
\begin{align}
    \begin{aligned}
\lVert\overrightarrow{v}^{(1)}\rVert_{\mathbf{L}^2(\partial\O_T)}&\leq C \lVert\overrightarrow{v}^{(1)}\rVert_{L^2(0,T;\mathbf{H}^{1}(\O))}\leq C\varrho^{-1} e^{\vartheta_1\lambda}
    \end{aligned}
\end{align}
where $|e^{-\lambda^2 t}|\leq 1$ for $t\geq 0$. 
Using the aforementioned calculation along with the Cauchy-Schwarz inequality, the trace theorem and the DN map given by Equation~\eqref{partial dn map for u_l}, the upper bound for $J_{2,1}$ can be obtained as 
\begin{align}\label{j21 bound}
    \begin{aligned}
        \lvert J_{2,1} \rvert&\leq \lVert\overrightarrow{v}^{(1)}\rVert_{\textbf{L}^2(\partial\O_T)} \lVert \partial_{\nu}\overrightarrow{u}\rVert_{\textbf{L}^2((0,T)\times \widetilde{F}(\o_0))}\\&\leq C\varrho^{-1} e^{\vartheta_1\lambda}  \lVert (\Lambda_1^{\sharp}-\Lambda_2^{\sharp})\overrightarrow{F}\rVert_{\textbf{L}^2((0,T)\times \widetilde{F}(\o_0))}\\&
        \leq C\varrho^{-1} e^{\vartheta_1\lambda}  \lVert (\Lambda_1^{\sharp}-\Lambda_2^{\sharp})\overrightarrow{F}\rVert_{\mathbf{H}^{\frac{1}{2},\frac{1}{4}}((0,T)\times \widetilde{F}(\o_0))}
        \\& \leq C\varrho^{-1} e^{\vartheta_1\lambda}  \lVert (\Lambda_1^{\sharp}-\Lambda_2^{\sharp})\rVert \lVert \overrightarrow{F} \rVert_{\mathbf{H}^{\frac{3}{2}, \frac{3}{4}}(\partial\O_T)}\\& \leq C\varrho^{-1} e^{\vartheta_1\lambda}  \lVert (\Lambda_1^{\sharp}-\Lambda_2^{\sharp})\rVert \lVert \overrightarrow{u}^{(2)}\rVert_{\mathbf{H}^{2,1}(\O_T)}\\&\leq C\lambda^2 \varrho^{-2} e^{\vartheta_1\lambda+\vartheta_2\lambda^2} \langle \tau,\xi\rangle_*^4 \lVert (\Lambda_1^{\sharp}-\Lambda_2^{\sharp})\rVert \\& \leq C \lambda^2\varrho^{-2}  \langle \tau,\xi\rangle_*^4 e^{\vartheta_2\lambda^2}  \lVert (\Lambda_1^{\sharp}-\Lambda_2^{\sharp})\rVert 
    \end{aligned}
\end{align}
for some $\vartheta_2>0$ and $\lambda$ large enough.

\noindent We next estimate the term $J_{2,2}$. By applying the boundary Carleman estimate \eqref{boundary carleman estimate1}, we obtain the following upper bound:
\begin{align}
    \begin{aligned}
        \lvert J_{2,2}\rvert &\leq C \varrho^{-1} \lVert e^{-\phi} \partial_{\nu} \overrightarrow{u}\rVert_{\textbf{L}^2((0,T)\times \widetilde{B}(\o_0)}\\&\leq C \varrho^{-1} \lVert e^{-\phi} \partial_{\nu} \overrightarrow{u}\rVert_{\textbf{L}^2((0,T)\times B(\o)}\\&
        \leq  \dfrac{C \varrho^{-1}}{\sqrt{\epsilon}}
       \left\langle(\o\cdot \nu(x))e^{-\phi}\partial_\nu \overrightarrow{u},e^{-\phi}\partial_\nu\overrightarrow{u}\right\rangle_{(0,T)\times B(\o)}^{1/2} 
        \\&\leq  \dfrac{C \varrho^{-1}}{\sqrt{\epsilon \lambda}}\left(\|e^{-\phi}\mathcal{L}_{A_{(1)},Q_{(1)}}\overrightarrow{u}\|_{\textbf{L}^2(\Omega_T)} +\sqrt{\lambda} \left\langle(\o\cdot \nu(x))e^{-\phi}\partial_\nu \overrightarrow{u},e^{-\phi}\partial_\nu\overrightarrow{u}\right\rangle_{(0,T)\times F(\o)}^{1/2} \right).
    \end{aligned}
\end{align}
Using IBVP~\eqref{L A1 Q1 in term of A U2} into the above expression and using the inclusions in Equation~\eqref{front and back face relation}, we obtain
\begin{align}\label{j22 bound used later}
    \begin{aligned}
      \lvert J_{2,2}\rvert &\leq  \dfrac{C \varrho^{-1}}{\sqrt{\epsilon \lambda}}\left(\|e^{-\phi}(2A\cdot\nabla_x \overrightarrow{u}^{(2)} + \widetilde{Q}\overrightarrow{u}^{(2)})\|_{\textbf{L}^2(\Omega_T)} +\sqrt{\lambda} \left\langle(\o\cdot \nu(x))e^{-\phi}\partial_\nu \overrightarrow{u},e^{-\phi}\partial_\nu\overrightarrow{u}\right\rangle_{(0,T)\times \widetilde{F}(\o_0)}^{1/2} \right) \\& \leq \dfrac{C \varrho^{-1}}{\sqrt{\epsilon \lambda}}\left(\|e^{-\phi}(2A\cdot\nabla_x \overrightarrow{u}^{(2)} + \widetilde{Q}\overrightarrow{u}^{(2)})\|_{\textbf{L}^2(\Omega_T)} +\sqrt{\lambda}  e^{\vartheta_3\lambda} \lVert \partial_{\nu}\overrightarrow{u}\rVert_{\textbf{L}^2((0,T)\times \widetilde{F}(\o_0))} \right),
    \end{aligned}
\end{align}
for some $\vartheta_3>0$ and $\lambda$ large enough. 
Next, we substitute $\overrightarrow{u}^{(2)}$ given by Equation~\eqref{go solution for u2 and v1} in the above expression and follow a similar analysis performed to obtain bounds in Equations~\eqref{bound for jij terms} and \eqref{j21 bound}. Thus, there exists $\vartheta_2>0$ such that 
\begin{align}\label{j22 bound}
    \begin{aligned}
         \lvert J_{2,2}\rvert  &\leq  \dfrac{C \varrho^{-1}}{\sqrt{\epsilon \lambda}} \left(
\lambda
\varrho^{-1} \langle\tau,\xi\rangle_*^4
+\lambda^{-1}
\varrho^{-1} \langle\tau,\xi\rangle_*^4+ \lambda^{5/2}\varrho^{-1} e^{\vartheta_3\lambda+\vartheta_2\lambda^2}  \lVert (\Lambda_1^{\sharp}-\Lambda_2^{\sharp})\rVert
         \right)\\&\leq C \varrho^{-2}\langle\tau,\xi\rangle_*^4\left(\sqrt{\lambda}+\lambda^2 e^{\vartheta_2\lambda^2} \lVert (\Lambda_1^{\sharp}-\Lambda_2^{\sharp})\rVert
         \right)
    \end{aligned}
\end{align}
holds for $\lambda$ large enough.
Using Equations \eqref{equation 54} and \eqref{j1 expression}-\eqref{j22 bound}, we have 
\begin{align}\label{j13 bound}
    \begin{aligned}
    \lvert J_{1,3}\rvert&=
\left\lvert\int_{\O_T}2\sum_{i=1}^{k} \left(\lambda\left(\o\cdot (\overrightarrow{A}_{(1)}^{i}- \overrightarrow{A}_{(2)}^{i}) \right)(T_u^{(2)})_i\right) \left(( {T_v}^{(1)})_i \right)\ dx dt \right\rvert\\&=
\left\lvert 
2 \lambda \int_{\O_T} \left((\o\cdot A)\cdot \overrightarrow{T_u}^{(2)} \right)\cdot \overrightarrow{T_v}^{(1)}\ dx dt
\right\rvert
\\&\leq C \varrho^{-2}\langle\tau,\xi\rangle_*^4\left(\sqrt{\lambda}+\lambda^2 e^{\vartheta_2\lambda^2} \lVert (\Lambda_1^{\sharp}-\Lambda_2^{\sharp})\rVert
         \right).
    \end{aligned}
\end{align}
After simplification, estimate \eqref{j13 bound} reduces to
\begin{align}\label{5.15}
    \begin{aligned}
        \left\lvert 
 \int_{\O_T} \left((\o\cdot A)\cdot \overrightarrow{T_u}^{(2)} \right)\cdot \overrightarrow{T_v}^{(1)}\ dx dt
\right\rvert &\leq  C \varrho^{-2}\langle\tau,\xi\rangle_*^4\left(\dfrac{1}{\sqrt{\lambda}}+\lambda e^{\vartheta_2\lambda^2}  \lVert (\Lambda_1^{\sharp}-\Lambda_2^{\sharp})\rVert
         \right)\\&\leq C \varrho^{-2}\langle\tau,\xi\rangle_*^4\left(\dfrac{1}{\sqrt{\lambda}}+e^{2\vartheta_2\lambda^2}  \lVert (\Lambda_1^{\sharp}-\Lambda_2^{\sharp})\rVert
         \right).
    \end{aligned}
\end{align}
Also, on dividing Equation~\eqref{j13 first value} by $2\lambda$, we obtain 
\begin{align}
    \begin{aligned}
         \int_{\O_T} \left((\o\cdot A)\cdot \overrightarrow{T_u}^{(2)} \right)\cdot \overrightarrow{T_v}^{(1)}\ dx dt= \mathrm{i} \sum_{i=1}^{k} \alpha_i \widetilde{\alpha}_i \lvert \xi\rvert  \widehat{\Xi^2_\varrho\o\cdot  \overrightarrow{A}^i}(\tau,\xi).
    \end{aligned}
\end{align}
 In particular, taking
\(
     (\alpha_1,\cdots,\alpha_k)=(\widetilde{\alpha}_1,\cdots, \widetilde{\alpha}_k)= \mathbf e_i,
\)
where $\mathbf e_i$ denotes the $i^{th}$ canonical basis vector of $\mathbb{R}^k$, estimate \eqref{5.15} becomes
\begin{align}\label{5.16} \begin{aligned} \left| \,|\xi|\, \widehat{\Xi_\varrho^{2}\,\omega\cdot\overrightarrow{A}^{\,i}} (\tau,\xi) \right| \le C\varrho^{-2} \langle\tau,\xi\rangle_{*}^{4} \left( \frac{1}{\sqrt{\lambda}} + e^{2\vartheta_{2}\lambda^{2}} \|\Lambda_{1}^{\sharp}-\Lambda_{2}^{\sharp}\| \right), \end{aligned} \end{align}
for every $ i\in \{1,\cdots,k\}$ and every $(\tau,\xi)\in\mathbb{R}\times\mathbb{R}^{n}$ satisfying $\o\cdot \xi= 0$.
Now we define some notations and results used in the forthcoming analysis: For a fixed $\o_0 \in \mathbb{S}^{n-1}$ and $\epsilon>0$, we define a spherical patch $\mathfrak{C}_{\o_0}$ around $\o_0$, given by 
\begin{align}
    \mathfrak{C}_{\o_0}:=\left\{ \o\in \mathbb{S}^{n-1}; \lvert \o-\o_0\rvert< \frac{\epsilon}{2}\right\}.
\end{align}
Also, for $\o\in \mathbb{S}^{n-1}$, we denote  the plane passing through origin and perpendicular to $\o$ by $\mathcal{P}_\o$ and   define $\mathcal{P}:=\displaystyle\cup_{\o\in \mathfrak{C}_{\o_0}} \mathcal{P}_\o$. Next, for $(\tau, \xi)\in  \mathbb{R} \times \mathcal{P}$ in estimate~\eqref{5.16}, we choose $\o \in \mathfrak{C}_{\o_0}$ (depending on $\xi)$ with $\o(\xi)\cdot\xi=0$ so that using the properties of Fourier transform we obtain
\begin{align}\label{5.17}
    \begin{aligned}
      \left\lvert  \widehat{\Xi^2_\varrho\partial_j \o(\xi)\cdot  \overrightarrow{A}^i}(\tau,\xi)\right\rvert \leq C \varrho^{-2}\langle\tau,\xi\rangle_*^4\left(\dfrac{1}{\sqrt{\lambda}}+e^{2\vartheta_2\lambda^2}  \lVert (\Lambda_1^{\sharp}-\Lambda_2^{\sharp})\rVert
         \right),
    \end{aligned}
\end{align}
for $j\in \{1,\cdots,n\}$. 
From the aforementioned expression, our objective is to establish the required stability estimate. To this end, we first derive a uniform bound for the Fourier transform of $\Xi^2_\varrho \partial_j\overrightarrow{A}^{i}$, and subsequently for $\Xi^2_\varrho \overrightarrow{A}^{i}$. 
\begin{lemma}\label{Lemma 5.1}
Let $\mathfrak C\subseteq \mathcal P$ be an open cone in $\mathbb R^n$.
Assume that
\(
    \nabla_x\cdot A=[0]_{k\times k}
 \text{ in } \Omega_T.
\)
Then, for every $(\tau,\xi)\in \mathbb R\times \mathfrak C$,
$i\in\{1,\ldots,k\}$, and $j\in\{1,\ldots,n\}$, there exist constants
$C>0$ and $\vartheta_2>0$ such that
\begin{align}\label{eq:Lemma_5_1_estimate}
\left|
\widehat{\Xi_\varrho^2\partial_j\overrightarrow A^{\,i}}(\tau,\xi)
\right|
\le
C\varrho^{-2}\langle\tau,\xi\rangle_*^4
\left(
\frac{1}{\sqrt{\lambda}}
+
e^{2\vartheta_2\lambda^2}
\|\Lambda_1^\sharp-\Lambda_2^\sharp\|
\right).
\end{align}
\end{lemma}
\begin{proof}
Fix $i\in\{1,\ldots,k\}$, $j\in\{1,\ldots,n\}$, and
$\xi\in \mathcal P\setminus\{0\}$. For this fixed $\xi$, choose
$n-1$ linearly independent vectors
\begin{align*}
  \omega^\ell(\xi)
=
\bigl(\omega_1^\ell(\xi),\ldots,\omega_n^\ell(\xi)\bigr),
\qquad \ell=1,\ldots,n-1,  
\end{align*}
such that
\begin{align*}
  \omega^\ell(\xi)\in \mathfrak C_{\omega_0},
\qquad
\omega^\ell(\xi)\cdot \xi=0,
\qquad
\ell=1,\ldots,n-1.  
\end{align*}
These vectors span the hyperplane orthogonal to $\xi$. From the previously
established estimate for the projections of
$\widehat{\Xi_\varrho^2\partial_j\overrightarrow A^{\,i}}$ in the directions
$\omega^\ell(\xi)$, we have
\begin{align}\label{eq:omega_projection_estimate}
\left|
\omega^\ell(\xi)\cdot
\widehat{\Xi_\varrho^2\partial_j\overrightarrow A^{\,i}}(\tau,\xi)
\right|
\le
C\varrho^{-2}\langle\tau,\xi\rangle_*^4
\left(
\frac{1}{\sqrt{\lambda}}
+
e^{2\vartheta_2\lambda^2}
\|\Lambda_1^\sharp-\Lambda_2^\sharp\|
\right),
\end{align}
for every $\ell=1,\ldots,n-1$.
For simplicity, define
\begin{align*}
  B_\ell(\tau,\xi)
:=
\omega^\ell(\xi)\cdot
\widehat{\Xi_\varrho^2\partial_j\overrightarrow A^{\,i}}(\tau,\xi).  
\end{align*}
Equivalently, we have
\begin{align}\label{eq:B_ell_definition}
B_\ell(\tau,\xi)
=
\sum_{r=1}^{n}
\omega_r^\ell(\xi)
\widehat{\Xi_\varrho^2\partial_j A_r^{\,i}}(\tau,\xi),
\qquad
\ell=1,\ldots,n-1.
\end{align}
Thus estimate~\eqref{eq:omega_projection_estimate} gives an upper bound for each
$B_\ell$.
 We now use the divergence-free condition. Since
   $ \nabla_x\cdot \overrightarrow A^{\,i}=0$
in  $\Omega_T,$
and since $\Xi_\varrho=\Xi_\varrho(t)$ depends only on the time variable, we
also have
\(
   \nabla_x\cdot
\bigl(\Xi_\varrho^2(t)\overrightarrow A^{\,i}(t,x)\bigr)=0. 
\)
Applying $\partial_j$ to this identity gives
\(
    \nabla_x\cdot
\bigl(\Xi_\varrho^2(t)\partial_j\overrightarrow A^{\,i}(t,x)\bigr)=0.
\)
Taking the Fourier transform with respect to $(t,x)$, we obtain
\begin{align}\label{eq:fourier_divergence_identity}
\mathrm{i}\sum_{r=1}^{n}
\xi_r
\widehat{\Xi_\varrho^2\partial_j A_r^{\,i}}(\tau,\xi)
=0.
\end{align}
Since $\xi\neq0$, we may divide by $|\xi|$. Hence, we have
\begin{align}\label{eq:normalized_divergence_identity}
\sum_{r=1}^{n}
\frac{\xi_r}{|\xi|}
\widehat{\Xi_\varrho^2\partial_j A_r^{\,i}}(\tau,\xi)
=0.
\end{align}
Equations \eqref{eq:B_ell_definition} and
\eqref{eq:normalized_divergence_identity} form a linear system for the
$n$ unknown Fourier components
\begin{align*}
  \widehat{\Xi_\varrho^2\partial_j A_1^{\,i}}(\tau,\xi),
\ldots,
\widehat{\Xi_\varrho^2\partial_j A_n^{\,i}}(\tau,\xi).  
\end{align*}
More precisely, we can write this system as
\begin{align*}
    D(\xi)X(\tau,\xi)=B(\tau,\xi),
\end{align*}
where
\begin{align*}
  D(\xi)=
\begin{bmatrix}
\omega_1^1(\xi)&\omega_2^1(\xi)&\cdots&\omega_n^1(\xi)\\
\omega_1^2(\xi)&\omega_2^2(\xi)&\cdots&\omega_n^2(\xi)\\
\vdots&\vdots&\ddots&\vdots\\
\omega_1^{n-1}(\xi)&\omega_2^{n-1}(\xi)&\cdots&\omega_n^{n-1}(\xi)\\
\dfrac{\xi_1}{|\xi|}&\dfrac{\xi_2}{|\xi|}&\cdots&
\dfrac{\xi_n}{|\xi|}
\end{bmatrix},\quad  
   X(\tau,\xi)=
\begin{bmatrix}
\widehat{\Xi_\varrho^2\partial_j A_1^{\,i}}(\tau,\xi)\\
\widehat{\Xi_\varrho^2\partial_j A_2^{\,i}}(\tau,\xi)\\
\vdots\\
\widehat{\Xi_\varrho^2\partial_j A_n^{\,i}}(\tau,\xi)
\end{bmatrix},
\quad
B(\tau,\xi)=
\begin{bmatrix}
B_1(\tau,\xi)\\
B_2(\tau,\xi)\\
\vdots\\
B_{n-1}(\tau,\xi)\\
0
\end{bmatrix}. 
\end{align*}
The rows of $D(\xi)$ consist of the $n-1$ linearly independent vectors
$\omega^\ell(\xi)$, which are orthogonal to $\xi$, together with the
normal direction $\xi/|\xi|$. Hence, these $n$ vectors are linearly
independent in $\mathbb R^n$. Therefore, $D(\xi)$ is invertible.
Moreover, the matrix $D(\xi)$ is homogeneous of degree zero in $\xi$.
That is,
\(
  D(\rho \xi)=D(\xi),
\text{ for } \rho>0,  
\)
after restricting to the same angular direction. Thus, it is enough to study
$D(\xi)$ for $\xi\in \mathbb S^{n-1}\cap \mathcal P$.

\par Fix $\xi_0\in \mathbb S^{n-1}\cap \mathcal P$. Since $D(\xi_0)$ is
invertible, we have
   $\det D(\xi_0)\neq0$.
By continuity of the determinant map, there exists an open neighbourhood
$\widetilde{\mathfrak C}\subset \mathbb S^{n-1}\cap \mathcal P$ of
$\xi_0$ and a constant $c>0$ such that
\(
    |\det D(\xi)|\ge c,
\;
\xi\in \widetilde{\mathfrak C}.
\)
Consequently, $D(\xi)^{-1}$ is uniformly bounded on
$\widetilde{\mathfrak C}$. Hence there exists $C>0$ such that
\begin{align*}
    \|D(\xi)^{-1}\|\le C,
\qquad
\xi\in \widetilde{\mathfrak C}.
\end{align*}
Therefore, we have
\begin{align*}
   |X(\tau,\xi)|
\le
\|D(\xi)^{-1}\|\,|B(\tau,\xi)|
\le
C|B(\tau,\xi)|,
\qquad
\xi\in \widetilde{\mathfrak C}. 
\end{align*}
Using the bounds for $B_\ell$ from
estimate~\eqref{eq:omega_projection_estimate}, we get
\begin{align*}
    |B(\tau,\xi)|
\le
C\varrho^{-2}\langle\tau,\xi\rangle_*^4
\left(
\frac{1}{\sqrt{\lambda}}
+
e^{2\vartheta_2\lambda^2}
\|\Lambda_1^\sharp-\Lambda_2^\sharp\|
\right).
\end{align*}
Thus, we have
\begin{align*}
    |X(\tau,\xi)|
\le
C\varrho^{-2}\langle\tau,\xi\rangle_*^4
\left(
\frac{1}{\sqrt{\lambda}}
+
e^{2\vartheta_2\lambda^2}
\|\Lambda_1^\sharp-\Lambda_2^\sharp\|
\right).
\end{align*}
Since $X(\tau,\xi)$ is precisely the vector
$\widehat{\Xi_\varrho^2\partial_j\overrightarrow A^{\,i}}(\tau,\xi)$,
thus we have shown that
\begin{align}\label{eq:local_conic_estimate}
\left|
\widehat{\Xi_\varrho^2\partial_j\overrightarrow A^{\,i}}(\tau,\xi)
\right|
\le
C\varrho^{-2}\langle\tau,\xi\rangle_*^4
\left(
\frac{1}{\sqrt{\lambda}}
+
e^{2\vartheta_2\lambda^2}
\|\Lambda_1^\sharp-\Lambda_2^\sharp\|
\right), \quad \text{for all $\xi\in \widetilde{\mathfrak C}$}.
\end{align}
Finally, define the open cone generated by
$\widetilde{\mathfrak C}$ as
\(
    \mathfrak C
:=
\bigcup_{\rho>0}\rho\,\widetilde{\mathfrak C}.
\)
Since $D(\xi)$ is homogeneous of degree zero in $\xi$, the uniform
invertibility obtained on $\widetilde{\mathfrak C}$ extends to the whole
cone $\mathfrak C$. Hence estimate~\eqref{eq:local_conic_estimate} holds for every
$(\tau,\xi)\in \mathbb R\times \mathfrak C$.
This proves estimate~\eqref{eq:Lemma_5_1_estimate} and completes the proof.
\end{proof}
The following lemma from an interpolation argument provides an estimate for the Fourier transform of the localized coefficient $\overrightarrow{A}^{\,i}$ in terms of the difference of the corresponding boundary operators, which will play a crucial role in the proof of the stability result.
\begin{lemma}\label{estimate: interpolation}
    Let $B_R :=\{x\in \mathbb{R}^{1+n},\lvert x\rvert< R\}$ be a open ball of radius $R>0$. Then, for every $R\ge 1$ and $\varrho\in \left(0,\frac{T}{4}\right)$, there exist  constants $C>0$ and $\theta \in (0,1)$, independent of $R$, $\varrho$, and $\lambda$, such that the following estimate holds
    \begin{align}\label{estimate; interpolation}
        \lvert| |\xi|\widehat{\Xi_\varrho^2\overrightarrow{A}^i}\rvert|_{\textbf{L}^{\infty}(B_R)} \leq C e^{R(1-\theta)}\varrho^{-2\theta} R^{4\theta}\left(\dfrac{1}{\sqrt{\lambda}}+e^{2\vartheta_2\lambda^2}  \lVert (\Lambda_1^{\sharp}-\Lambda_2^{\sharp})\rVert
         \right)^{\theta}.
    \end{align}
\end{lemma}
Before proving the previous lemma, we recall a stability estimate for analytic continuation established in \cite[Lemma 3.4]{BellassouedBenAicha2017}.

\begin{proposition}\label{real_analytic_function_proposition}
Let $B_r:=\{x\in \mathbb{R}^{d},\lvert x\rvert< r\}$, and $\gamma \in (\mathbb{N}\cup\{0\})^{d},~d\geq 2$ be a multi-index of length $\lvert \gamma\rvert$. Also, let $g$ be a real analytic function in $B_2$. Assume that there exist constants $M,\rho>0$ such that
\begin{align*}
    \|\partial^\gamma g\|_{L^\infty(B(0,2))}
\leq
\frac{M\,|\gamma|!}{(2\rho)^{|\gamma|}},
\qquad
\gamma\in (\mathbb{N}\cup\{0\})^{d}.
\end{align*}
Then, for every non-empty open set $U\subset B_1$, there exist constants $N=N(\rho)>0$ and $\theta\in(0,1)$, depending only on $d$, $\rho$, and $|U|$, such that
\begin{align*}
  \|g\|_{L^\infty(B(0,1))}
\leq
N\,M^{1-\theta}
\|g\|_{L^\infty(U)}^{\theta}.  
\end{align*}
\end{proposition}
\begin{proof}[Proof of Lemma \ref{estimate: interpolation}]
Fix $j\in \{1,\ldots,n\}$. For $R>0$, define
\begin{align*}
    \overrightarrow{f}^{i}_{R,j}(t,x)
:=
\widehat{\Xi_\varrho^2\partial_j\overrightarrow{A}^{\,i}}
(Rt,Rx),
\qquad (t,x)\in \mathbb{R}^{1+n}.
\end{align*}
Since $\Xi_\varrho^2\partial_j\overrightarrow{A}^{\,i}$ is compactly supported, its Fourier transform is an entire function. Hence, $\overrightarrow{f}^{i}_{R,j}$ is real analytic in $\mathbb{R}^{1+n}$.

Let $\gamma\in (\mathbb N\cup\{0\})^{1+n}$ be an arbitrary multi-index. Differentiating under the integral sign, we obtain
\begin{align*}
\partial^\gamma_{(t,x)}\overrightarrow{f}^{i}_{R,j}(t,x)
&=
\int_{\mathbb R^{1+n}}
e^{-\mathrm{i}R(s,y)\cdot (t,x)}
(-\mathrm{i}R)^{|\gamma|}
(s,y)^\gamma
\bigl(\Xi_\varrho^2\partial_j\overrightarrow{A}^{\,i}\bigr)(s,y)
\,ds\,dy.
\end{align*}
Therefore, we have
\begin{align*}
\bigl|\partial^\gamma_{(t,x)}\overrightarrow{f}^{i}_{R,j}(t,x)\bigr|
&\le
R^{|\gamma|}
\int_{\mathbb R^{1+n}}
|(s,y)|^{|\gamma|}
\bigl|
\bigl(\Xi_\varrho^2\partial_j\overrightarrow{A}^{\,i}\bigr)(s,y)
\bigr|
\,ds\,dy.
\end{align*}
Since the support of
$\Xi_\varrho^2\partial_j\overrightarrow{A}^{\,i}$
is contained in $(0,T)\times \Omega$ and
$\operatorname{diam}(\Omega)<T$, we have
$s^2+|y|^2\le 2T^2$
on the support of the integrand. Consequently, we have
\begin{align*}
\bigl|\partial^\gamma_{(t,x)}\overrightarrow{f}^{i}_{R,j}(t,x)\bigr|
&\le
(2T^2)^{\frac{|\gamma|}{2}}
R^{|\gamma|}
\int_{\mathbb R^{1+n}}
\bigl|
\bigl(\Xi_\varrho^2\partial_j\overrightarrow{A}^{\,i}\bigr)(s,y)
\bigr|
\,ds\,dy.
\end{align*}
Using the a priori bound on $\overrightarrow{A}^{\,i}$, there exists a constant
$C_\ast>0$, independent of $R$ and $\gamma$, such that
\begin{align*}
    \bigl|\partial^\gamma_{(t,x)}\overrightarrow{f}^{i}_{R,j}(t,x)\bigr|
\le
C_\ast (\sqrt{2}T)^{|\gamma|}R^{|\gamma|}.
\end{align*}
Since
\(
    \frac{R^{|\gamma|}}{|\gamma|!}\le e^R
\) for \(R>0,
\)
it follows that
\begin{align}\label{eq:estimate_f}
\bigl|\partial^\gamma_{(t,x)}\overrightarrow{f}^{i}_{R,j}(t,x)\bigr|
\le
C_\ast e^R
\frac{|\gamma|!}{(T^{-1})^{|\gamma|}},
\qquad
(t,x)\in \mathbb R^{1+n}.
\end{align}
Thus, $\overrightarrow{f}^{i}_{R,j}$ satisfies the assumptions of Proposition~\ref{real_analytic_function_proposition}. Applying Proposition~\ref{real_analytic_function_proposition} with
\(
   U=(\mathbb R\times\mathfrak C)\cap B_1, 
\)
we deduce that there exists $\theta\in(0,1)$ such that
\begin{align}\label{eq:estimate_f_B1}
\|\overrightarrow{f}^{i}_{R,j}\|_{\textbf{L}^\infty(B_1)}
\le
C e^{R(1-\theta)}
\|\overrightarrow{f}^{i}_{R,j}\|_{\textbf{L}^\infty((\mathbb R\times\mathfrak C)\cap B_1)}^\theta .
\end{align}
By the definition of $\overrightarrow{f}^{i}_{R,j}$, we have
\begin{align*}
    \|\overrightarrow{f}^{i}_{R,j}\|_{\textbf{L}^\infty(B_1)}
=
\|\widehat{\Xi_\varrho^2\partial_j\overrightarrow{A}^{\,i}}\|_{\textbf{L}^\infty(B_R)}.
\end{align*}
Combining inequality~\eqref{eq:estimate_f_B1} with Lemma~\ref{Lemma 5.1}, we obtain
\begin{align*}
    \|\widehat{\Xi_\varrho^2\partial_j\overrightarrow{A}^{\,i}}\|_{\textbf{L}^\infty(B_R)}
\le
C e^{R(1-\theta)}
\varrho^{-2\theta}
(1+R^2)^{2\theta}
\left(
\frac{1}{\sqrt{\lambda}}
+
e^{2\vartheta_2\lambda^2}
\|\Lambda_1^\sharp-\Lambda_2^\sharp\|
\right)^\theta .
\end{align*}
Since $R\ge 1$, we have
\(
    (1+R^2)^{2\theta}
\le C R^{4\theta}.
\)
Therefore, we have
\begin{align*}
    \bigl\||\xi|
\widehat{\Xi_\varrho^2\overrightarrow{A}^{\,i}}
\bigr\|_{\textbf{L}^\infty(B_R)}
\le
C e^{R(1-\theta)}
\varrho^{-2\theta}
R^{4\theta}
\left(
\frac{1}{\sqrt{\lambda}}
+
e^{2\vartheta_2\lambda^2}
\|\Lambda_1^\sharp-\Lambda_2^\sharp\|
\right)^\theta .
\end{align*}
This completes the proof.
\end{proof}
We now combine the interpolation estimate \eqref{estimate; interpolation}
with the a priori assumptions on the coefficients in order to derive a
stability estimate for the first-order coefficient
$\overrightarrow A^{\,i}$ in terms of the partial DN map.
The main idea is to split the Fourier space into low and high frequencies and
then choose the auxiliary parameters $R$, $\lambda$, and $\varrho$ in a
balanced way.

\noindent
Let
\begin{align}\label{mathcal E in term of dn differnce}
    \mathcal E:=\|\Lambda_1^\sharp-\Lambda_2^\sharp\|.
\end{align}
By Plancherel's theorem, we have
\begin{align}\label{estimate:A_fourier_split}
\|\Xi_\varrho^2\overrightarrow A^{\,i}\|_{L^2(\Omega_T)}^{\frac{2}{\theta}}
&=
\left(
\int_{\mathbb R^{1+n}}
\left|
\widehat{\Xi_\varrho^2\overrightarrow A^{\,i}}(\tau,\xi)
\right|^2
\,d\tau\,d\xi
\right)^{\frac1\theta}
\nonumber\\
&=
\left(
\int_{B_R}
\left|
\widehat{\Xi_\varrho^2\overrightarrow A^{\,i}}(\tau,\xi)
\right|^2
\,d\tau\,d\xi
+
\int_{B_R^c}
\left|
\widehat{\Xi_\varrho^2\overrightarrow A^{\,i}}(\tau,\xi)
\right|^2
\,d\tau\,d\xi
\right)^{\frac1\theta}.
\end{align}
Since $0<\theta<1$, we have $1/\theta>1$. Thus, using the convexity of the function
$x\mapsto x^{1/\theta}$ yields
\begin{align}\label{estimate:I1_I2}
\|\Xi_\varrho^2\overrightarrow A^{\,i}\|_{L^2(\Omega_T)}^{\frac{2}{\theta}}
\le
C\left(I_1^{\frac1\theta}+I_2^{\frac1\theta}\right),
\end{align}
where
\begin{align*}
   I_1:=
\int_{B_R}
\left|
\widehat{\Xi_\varrho^2\overrightarrow A^{\,i}}(\tau,\xi)
\right|^2
\,d\tau\,d\xi,
\qquad
I_2:=
\int_{B_R^c}
\left|
\widehat{\Xi_\varrho^2\overrightarrow A^{\,i}}(\tau,\xi)
\right|^2
\,d\tau\,d\xi. 
\end{align*}
We begin with the estimate of the high-frequency contribution $I_2$.
Since $\langle \tau,\xi\rangle\ge R$ on $B_R^c$, we obtain
\begin{align}\label{estimate:I2}
\begin{aligned}
I_2
&\le
\frac1{R^2}
\int_{\mathbb R^{1+n}}
\langle\tau,\xi\rangle^2
\left|
\widehat{\Xi_\varrho^2\overrightarrow A^{\,i}}(\tau,\xi)
\right|^2
\,d\tau\,d\xi
\le
\frac{C}{R^2}
\|\Xi_\varrho^2\overrightarrow A^{\,i}\|_{H^1(\Omega_T)}^2.
\end{aligned}
\end{align}
Using the a priori regularity of $\overrightarrow A^{\,i}$ and the fact
that differentiation of the cut-off $\Xi_\varrho$ produces factors of
order $\varrho^{-1}$ (see Equation \eqref{Xi estimate}), we deduce that
\(
  \|\Xi_\varrho^2\overrightarrow A^{\,i}\|_{H^1(\Omega_T)}
\le
C\varrho^{-1}.  
\)
Consequently, we have
\begin{align}\label{estimate:I2_final}
I_2
\le
\frac{C}{\varrho^2R^2}.
\end{align}
Next, we estimate the low-frequency part $I_1$. We decompose
\begin{align*}
   I_1=I_{11}+I_{12},
\end{align*}
where
\begin{align*}
   I_{11}
=
\int_{B_R\cap\{|\xi|\le R^{-3/n}\}}
\left|
\widehat{\Xi_\varrho^2\overrightarrow A^{\,i}}(\tau,\xi)
\right|^2
\,d\tau\,d\xi, \quad 
   I_{12}
=
\int_{B_R\cap\{|\xi|>R^{-3/n}\}}
\left|
\widehat{\Xi_\varrho^2\overrightarrow A^{\,i}}(\tau,\xi)
\right|^2
\,d\tau\,d\xi. 
\end{align*}
To estimate $I_{11}$, we use the elementary Fourier bound
\begin{align*}
   \left\|
\widehat{\Xi_\varrho^2\overrightarrow A^{\,i}}
\right\|_{L^\infty(\mathbb R^{1+n})}
\le
\|\Xi_\varrho^2\overrightarrow A^{\,i}\|_{L^1(\Omega_T)}
\le C. 
\end{align*}
Hence, we get
\begin{align}\label{estimate:I11}
\begin{aligned}
I_{11}
&\le
\left\|
\widehat{\Xi_\varrho^2\overrightarrow A^{\,i}}
\right\|_{L^\infty(\mathbb R^{1+n})}^{2}
\int_{-R}^{R}
\int_{|\xi|\le R^{-3/n}}
d\xi\,d\tau
\le
CR\left(R^{-3/n}\right)^n
=
CR^{-2}.
\end{aligned}
\end{align}
For the estimate of $I_{12}$, we use the interpolation estimate \eqref{estimate; interpolation} obtained
earlier. Since $|\xi|>R^{-3/n}$, we have
\begin{align*}
 \left|
\widehat{\Xi_\varrho^2\overrightarrow A^{\,i}}(\tau,\xi)
\right|
\le
R^{3/n}
\left|
|\xi|
\widehat{\Xi_\varrho^2\overrightarrow A^{\,i}}(\tau,\xi)
\right|.   
\end{align*}
Therefore, we have
\begin{align}\label{estimate:I12_start}
I_{12}
&\le
R^{6/n}
|B_R|
\left\|
|\xi|
\widehat{\Xi_\varrho^2\overrightarrow A^{\,i}}
\right\|_{L^\infty(B_R)}^2
\le
CR^{n+1+\frac6n}
\left\|
|\xi|
\widehat{\Xi_\varrho^2\overrightarrow A^{\,i}}
\right\|_{L^\infty(B_R)}^2.
\end{align}
Applying Lemma~\ref{estimate: interpolation}, we conclude that
\begin{align}\label{estimate:I12}
I_{12}
&\le
Ce^{2R(1-\theta)}
\varrho^{-4\theta}
R^{8\theta+n+1+\frac6n}
\left(
\frac1{\sqrt{\lambda}}
+
e^{2\vartheta_2\lambda^2}\mathcal E
\right)^{2\theta}.
\end{align}
Since $0<\theta<1$, the above inequality, after adjusting the constant, can be written
\begin{align}\label{estimate:I12_simplified}
I_{12}^{1/\theta}
&\le
C
e^{\frac{2R(1-\theta)}{\theta}}
\varrho^{-4}
R^{8+\frac{n+1+\frac6n}{\theta}}
\left(
\frac1{\lambda}
+
e^{4\vartheta_2\lambda^2}\mathcal E^2
\right).
\end{align}
Define
\begin{align}\label{eq:alpha_definition}
\kappa
:=
8+\frac{n^2+n+6}{n\theta}.
\end{align}
Then \eqref{estimate:I12_simplified} becomes
\begin{align}\label{estimate:I12_alpha}
I_{12}^{1/\theta}
\le
C
\frac{R^\kappa}{\varrho^4}
e^{\frac{2R(1-\theta)}{\theta}}
\left(
\frac1{\lambda}
+
e^{4\vartheta_2\lambda^2}\mathcal E^2
\right).
\end{align}
Combining Equations \eqref{estimate:I1_I2},
\eqref{estimate:I2_final},
\eqref{estimate:I11},
and \eqref{estimate:I12_alpha}, we arrive at
\begin{align}\label{estimate:A_data_before_choice}
\|\Xi_\varrho^2\overrightarrow A^{\,i}\|_{L^2(\Omega_T)}^{\frac{2}{\theta}}
\le
C
\left[
\frac{R^\kappa}{\lambda\varrho^4}
e^{\frac{2R(1-\theta)}{\theta}}
+
\frac{R^\kappa}{\varrho^4}
e^{\frac{2R(1-\theta)}{\theta}}
e^{4\vartheta_2\lambda^2}\mathcal E^2
+
\frac1{\varrho^{\frac2\theta}R^{\frac2\theta}}
+
R^{-\frac2\theta}
\right].
\end{align}
Since $\Xi_\varrho=1$ away from a $\varrho$-neighborhood of the temporal
endpoints and $\overrightarrow A^{\,i}$ is uniformly bounded a priori, we have
\begin{align}\label{estimate:cutoff_removal}
\|\overrightarrow A^{\,i}\|_{L^2(\Omega_T)}^2
\le
\|\Xi_\varrho^2\overrightarrow A^{\,i}\|_{L^2(\Omega_T)}^2
+
C\varrho.
\end{align}
Raising this estimate to the power $1/\theta$ and using again
$(a+b)^{1/\theta}\le C(a^{1/\theta}+b^{1/\theta})$, we obtain
\begin{align}\label{estimate:A_before_balance}
\|\overrightarrow A^{\,i}\|_{L^2(\Omega_T)}^{\frac2\theta}
\le
C
\left[
\frac{R^\kappa}{\lambda\varrho^4}
e^{\frac{2R(1-\theta)}{\theta}}
+
\frac{R^\kappa}{\varrho^4}
e^{\frac{2R(1-\theta)}{\theta}}
e^{4\vartheta_2\lambda^2}\mathcal E^2
+
\frac1{\varrho^{\frac2\theta}R^{\frac2\theta}}
+
\varrho^{\frac1\theta}
\right].
\end{align}
We now choose the parameters so that the first, third, and fourth terms are
balanced. Take
\begin{align}\label{eq:rho_lambda_choice}
\varrho
=
R^{-2/3},
\qquad
\lambda
=
R^{\kappa+\frac83+\frac{2}{3\theta}}
e^{\frac{2R(1-\theta)}{\theta}}.
\end{align}
With this choice,
\begin{align*}
  \frac{R^\kappa}{\lambda\varrho^4}
e^{\frac{2R(1-\theta)}{\theta}}
=
R^{-\frac{2}{3\theta}}  
\quad \text{
and }\quad 
\frac1{\varrho^{\frac2\theta}R^{\frac2\theta}}
=
\varrho^{\frac1\theta}
=
R^{-\frac{2}{3\theta}}. 
\end{align*}
Therefore, the first, third and fourth terms in
estimate~\eqref{estimate:A_before_balance} are all bounded by
$CR^{-2/(3\theta)}$.

For the remaining term, using Equation~\eqref{eq:rho_lambda_choice}, we can find a
constant $\mu>0$, independent of $R$, such that for all sufficiently
large $R$, such that
\(
   \dfrac{R^\kappa}{\varrho^4}
e^{\frac{2R(1-\theta)}{\theta}}
e^{4\vartheta_2\lambda^2}
\le
e^{e^{\mu R}}. 
\)
Consequently, we have
\begin{align}\label{estimate:A_with_R}
\|\overrightarrow A^{\,i}\|_{L^2(\Omega_T)}^{\frac2\theta}
\le
C
\left(
e^{e^{\mu R}}\mathcal E^2
+
R^{-\frac{2}{3\theta}}
\right).
\end{align}
Assume first that $\mathcal E$ is sufficiently small. We choose
\(
R
=
\frac1\mu
\log |\log \mathcal E|.
\)
Then, we have
\(
 e^{e^{\mu R}}\mathcal E^2
=
e^{|\log \mathcal E|}
\mathcal E^2
=
\mathcal E.   
\)
Substituting this into \eqref{estimate:A_with_R} yields
\begin{align}\label{final norm of A with power theta}
 \|\overrightarrow A^{\,i}\|_{L^2(\Omega_T)}^{\frac2\theta}
\le
C
\left(
\mathcal E
+
|\log|\log \mathcal E||^{-\frac{2}{3\theta}}
\right).   
\end{align}
Subsequently, summing over all indices $1 \leq i\leq k$, and substituting \eqref{mathcal E in term of dn differnce} in above inequality, we achieve
\begin{align}\label{in comparing A, at last, 1 form}
 \| A\|_{\textbf{L}^2(\Omega_T)}
\le
C
\left(
\|\Lambda_1^\sharp-\Lambda_2^\sharp\|^{a_1}
+
|\log|\log \|\Lambda_1^\sharp-\Lambda_2^\sharp\|||^{-a_2}
\right) 
\end{align}
for some $C,a_1$ and $a_2>0$. Moreover, if consider the case when $\mathcal E\ge \mathcal E_0$, where
$\mathcal E_0>0$ is fixed. By the a priori assumptions on the admissible
set, there exists a constant $C(m)>0$ such that
\begin{align*}
 \|\overrightarrow A^{\,i}\|_{L^2(\Omega_T)}^{\frac2\theta}
\le C \leq \dfrac{C}{\mathcal{E}_0} \mathcal{E}_0 \leq \dfrac{C}{\mathcal{E}_0} \mathcal{E},
\end{align*}
that further leads to
\begin{align}\label{in comparing A, at last, 2 form}
    \begin{aligned}
        \| A\|_{\textbf{L}^2(\Omega_T)}\leq \dfrac{C}{\mathcal{E}_0} \|\Lambda_1^\sharp-\Lambda_2^\sharp\|^{\theta/2}.
    \end{aligned}
\end{align}
Following bounds \eqref{in comparing A, at last, 1 form} and \eqref{in comparing A, at last, 2 form}, we conclude that 
\begin{align}
 \| A\|_{\mathbf{L}^2(\Omega_T)}
\le
C
\left(
\|\Lambda_1^\sharp-\Lambda_2^\sharp\|^{a_1}
+
|\log|\log \|\Lambda_1^\sharp-\Lambda_2^\sharp\|||^{-a_2}
\right) 
\end{align}
for some $C,a_1$ and $a_2>0$.
Consequently, we obtain the desired stability result for the zeroth-order coefficient.
\subsection{Stability estimate for matrix-valued potential}
To obtain the stability estimate for the potential term, we  
 choose the GO solutions $\overrightarrow{u}^{(2)}$  and $\overrightarrow{v}^{(1)}$ for the operator $\mathcal{L}_{A_{(2)},Q_{(2)}}$ and $\mathcal{L}^*_{A_{(1)},Q_{(1)}}$ respectively,  of the  form
\begin{align}\label{go solution for u2 and v1 in stability of Q}
\begin{aligned}
      \overrightarrow{u}^{(2)} = e^{\phi(t,x)}(\overrightarrow{T_u}^{(2)}(t,x) + \overrightarrow{R_u}^{(2)}(t,x;\lambda)) \quad\text{ and }\quad
         \overrightarrow{v}^{(1)} = e^{-\phi(t,x)}(\overrightarrow{T_v}^{(1)}(t,x) + \overrightarrow{R_v}^{(1)}(t,x;\lambda)),
         \end{aligned}
\end{align}
where 
\begin{align}\label{tu and tv in stability of Q}
    \begin{aligned}
        \overrightarrow{T_u}^{(2)}&=\begin{bmatrix}
            \alpha_1\Xi_\varrho(t)e^{-\mathrm{i}(t\tau+x\cdot\xi)}e^{\int_{0}^{\infty} \o\cdot \overrightarrow{A}^1_{(2)}(t,x+s\o)ds}\\\vdots   \\ \alpha_k\Xi_\varrho(t)e^{-\mathrm{i}(t\tau+x\cdot\xi)}e^{\int_{0}^{\infty} \o\cdot \overrightarrow{A}^k_{(2)}(t,x+s\o)ds}
        \end{bmatrix}\ \text{ and } \ \overrightarrow{T_v}^{(1)}&= \begin{bmatrix}
    \widetilde{\alpha}_1\Xi_\varrho(t)e^{-\int_{0}^{\infty} \o\cdot \overrightarrow{A}_{(1)}^1(t,x+s\o)ds}\\ \vdots  \\ \widetilde{\alpha}_k\Xi_\varrho(t)e^{-\int_{0}^{\infty} \o\cdot \overrightarrow{A}_{(1)}^k(t,x+s\o)ds}
        \end{bmatrix}
    \end{aligned}
\end{align}
with $\alpha_i,\widetilde{\alpha}_i \in \mathbb{R}$ for $1\leq i\leq k$.
 Also, the correction terms $\overrightarrow{R_u}^{(2)}(t,x;\lambda)$ and
$\overrightarrow{R_v}^{(1)}(t,x;\lambda)$ satisfy
\begin{align} \label{ru and rv bounds for Q}
\begin{aligned}
\|\overrightarrow{R_u}^{(2)}\|_{L^2(0,T;\mathbf{H}^{k}(\O))}&\leq  C\lambda^{-1+k}\varrho^{-1}\langle \tau,\xi\rangle_*^3\quad\text{ and }\quad \|\overrightarrow{R_v}^{(1)}\|_{L^2(0,T;\mathbf{H}^k(\Omega))}&\leq C\lambda^{-1+k}\varrho^{-1},
\end{aligned}
\end{align}
for $ k\in \{0,1,2\}$ and $\langle \tau, \xi\rangle_*:=\sqrt{1+\tau^2+|\xi|^2}$. Next, we substitute the aforementioned GO solutions in integral identity \eqref{equation 54}. We observe that
\begin{align}
    \begin{aligned}
        \lvert J_{1,1}\rvert&=\big\lvert
    2\int_{\O_T}\sum_{i=1}^{k}  \left(\overrightarrow{A}^{i} 
    \cdot\left(\nabla_x {T_u}^{(2)}\right)_i\right)\left(( {T_v}^{(1)})_i +\overline{( R_v^{(1)})_i}\right)\ dx dt\big\rvert
     \leq C \lVert A\rVert_{\textbf{L}^2(\O_T)} \langle \tau, \xi\rangle_*(1+\varrho^{-1}\lambda^{-1}),\\ \lvert J_{1,2}\rvert &=\big\lvert 2\int_{\O_T}\sum_{i=1}^{k} \left(\overrightarrow{A}^{i} \cdot \left(\nabla_x R_u^{(2)}\right)_i\right) \left(\left( {T_v}^{(1)}\right)_i +\overline{(R_v^{(1)})_i}\right)~   dx dt\big\rvert\leq C \lVert A\rVert_{\textbf{L}^2(\O_T)} \varrho^{-1} \langle \tau, \xi\rangle_*^3(1+\varrho^{-1}\lambda^{-1}),\\\lvert J_{1,3}\rvert&=\big\lvert
2\int_{\O_T}\sum_{i=1}^{k} \left(\lambda\o\cdot \overrightarrow{A}^{i} \right)(T_u^{(2)})_i ({T_v}^{(1)})_i \ dx dt\big\rvert \leq C\lambda \lVert A\rVert_{\textbf{L}^2(\O_T)},\\\lvert J_{1,4}\rvert&=\big\lvert
        2\lambda\int_{\O_T}\sum_{i=1}^{k} \left(\left(\o\cdot \overrightarrow{A}^{i} \right)\left(T_u^{(2)}\right)_i\right) \overline{\left( R_v^{(1)}\right)_i}\ dx dt\big\rvert\leq C \lVert A\rVert_{\textbf{L}^2(\O_T)} \varrho^{-1},\\\lvert J_{1,5}\rvert&=\big\lvert2\lambda\int_{\O_T}\sum_{i=1}^{k} \left(\o\cdot \overrightarrow{A}^{i} \right)\left(R_u^{(2)}\right)_i\left(\left( {T_v}^{(1)}\right)_i +\overline{\left(R_v^{(1)}\right)_i}\right)\ dx dt\big\rvert\leq C \lVert A\rVert_{\textbf{L}^2(\O_T)} \varrho^{-1} \langle \tau, \xi\rangle_*^3(1+\varrho^{-1}\lambda^{-1}),
    \end{aligned}
\end{align}
for $\lambda$ large enough. Also, we have
\begin{align}\label{J16 in Q}
    \begin{aligned}
 J_{1,6} &= 
\int_{\Omega_T}\sum_{i,j=1}^{k}\widetilde{q}_{ji} 
    \left(T_u^{(2)}\right)_i \left({T_v}^{(1)}\right)_j\ dx dt+     \int_{\Omega_T}\sum_{i,j=1}^{k} \widetilde{q}_{ji} 
    \left(R_u^{(2)}\right)_i \left(\left( {T_v}^{(1)}\right)_j +\overline{\left( R_v^{(1)}\right)_j}\right)\ dx dt \\&\qquad+    \int_{\Omega_T}\sum_{i,j=1}^{k} \widetilde{q}_{ji} 
    \left(T_u^{(2)}\right)_i \overline{\left( R_v^{(1)}\right)_j}\ dx dt\\ :&\triangleq J_{1,6,1}+J_{1,6,2}+J_{1,6,3}.
\end{aligned}
\end{align} 
From the above expression, we observe that 
\begin{align}\label{j161 in Q}
    \begin{aligned}
J_{1,6,1}&=\int_{\Omega_T}\sum_{i,j=1}^{k}\widetilde{q}_{ji} 
    \left(T_u^{(2)}\right)_i \left({T_v}^{(1)}\right)_j\ dx dt\\&=\sum_{i,j=1}^{k}\int_{\Omega_T}\widetilde{q}_{ji} \alpha_i \widetilde{\alpha}_j
    \Xi^2_\varrho(t)e^{-\mathrm{i}(t\tau+x\cdot\xi)}e^{-\int_{0}^{\infty} \o\cdot \left(\overrightarrow{A}^j_{(1)}-\overrightarrow{A}^i_{(2)}\right)(t,x+s\o)ds} \ dxdt.
    \end{aligned}
\end{align}
Next, we derive suitable upper bounds for the last two terms in Equation~\eqref{J16 in Q}. Using \eqref{tu and tv in stability of Q} and \eqref{ru and rv bounds for Q}, we arrive at 
\begin{align}\label{j161 and j162 in Q}
    \begin{aligned}
        \lvert J_{1,6,2}\rvert = \big\lvert \int_{\Omega_T}\sum_{i,j=1}^{k} \widetilde{q}_{ji} 
    \left(R_u^{(2)}\right)_i \left(\left( {T_v}^{(1)}\right)_j +\overline{\left( R_v^{(1)}\right)_j}\right)\ dx dt \big\rvert\leq C \lambda^{-1}\varrho^{-1}\langle \tau,\xi\rangle_*^3(1+\lambda^{-1}\varrho^{-1}), \text{ and } \\
\lvert J_{1,6,3}\rvert = \big\lvert \int_{\Omega_T}\sum_{i,j=1}^{k} \widetilde{q}_{ji} 
    \left(T_u^{(2)}\right)_i \overline{\left( R_v^{(1)}\right)_j}\ dx dt \big\rvert \leq C \lambda^{-1}\varrho^{-1}.
    \end{aligned}
\end{align}
Also, by looking at the right-hand side of integral identity~\eqref{equation 54}, we have
\begin{align}
    \begin{aligned}
    J_2&=  \int_{(0,T)\times \widetilde{F}(\o_0)} \partial_{\nu}\overrightarrow{u}(t,x) \cdot \overline{\overrightarrow{v}^{(1)}}(t,x)\ dS_x dt +\int_{(0,T)\times \widetilde{B}(\o_0)} \partial_{\nu}\overrightarrow{u}(t,x) \cdot \overline{\overrightarrow{v}^{(1)}}(t,x)\ dS_x dt \\:&\triangleq J_{2,1}+J_{2,2}.
    \end{aligned}
\end{align}
Now, to derive an upper bound for $J_{2,1}$ and $J_{2,2}$, we use a similar analysis to that used to derive the bounds \eqref{j21 bound} and \eqref{j22 bound}. Thus, we have 
\begin{align}\label{j21 in Q}
    \begin{aligned}
        \lvert J_{2,1} \rvert&\leq \lVert\overrightarrow{v}^{(1)}\rVert_{\textbf{L}^2(\partial\O_T)} \lVert \partial_{\nu}\overrightarrow{u}\rVert_{\textbf{L}^2((0,T)\times \widetilde{F}(\o_0))}
       \\&
        \leq C\varrho^{-1} e^{\vartheta_1\lambda}  \mathcal E \lVert \overrightarrow{F} \rVert_{\mathbf{H}^{\frac{3}{2}, \frac{3}{4}}(\partial\O_T)}\\& \leq C\varrho^{-1} e^{\vartheta_1\lambda}  \mathcal E \lVert \overrightarrow{u}^{(2)}\rVert_{\mathbf{H}^{2,1}(\O_T)}\\&\leq C\lambda^2 \varrho^{-2} e^{\vartheta_1\lambda+\vartheta_2\lambda^2} \langle \tau,\xi\rangle_*^3 \mathcal E  \\& \leq C \lambda^2\varrho^{-2}  \langle \tau,\xi\rangle_*^3 e^{\vartheta_2\lambda^2}  \mathcal E,
    \end{aligned}
\end{align}
 and using 
 Carleman estimate \eqref{boundary carleman estimate1} and IBVP~\eqref{L A1 Q1 in term of A U2}, we obtained (see \eqref{j22 bound used later})
\begin{align}\label{j22 bound in Q}
    \begin{aligned}
        \lvert J_{2,2}\rvert &\leq \dfrac{C \varrho^{-1}}{\sqrt{\epsilon \lambda}}\left(\|e^{-\phi}(2A\cdot\nabla_x \overrightarrow{u}^{(2)} + \widetilde{Q}\overrightarrow{u}^{(2)})\|_{\textbf{L}^2(\Omega_T)} +\sqrt{\lambda}  e^{\vartheta_3\lambda} \lVert \partial_{\nu}\overrightarrow{u}\rVert_{\textbf{L}^2((0,T)\times \widetilde{F}(\o_0))} \right)\\& \leq C \varrho^{-2} \langle \tau, \xi\rangle_*^3 \left(\sqrt{\lambda}
         \lVert A\rVert_{\textbf{L}^2(\O_T)}+\dfrac{1}{\sqrt{\lambda}}+ e^{\vartheta_2\lambda^2} \mathcal E 
        \right)
    \end{aligned}
\end{align}
for some $\vartheta_1,\vartheta_2,\vartheta_3>0$ and $\lambda$ large enough. 
From Equations \eqref{J16 in Q} to \eqref{j22 bound in Q}, we have
\begin{align}\label{j161 bound in Q}
    \begin{aligned}
        \lvert J_{1,6,1}\rvert  \leq C \varrho^{-2} \langle \tau, \xi\rangle_*^3 \left(\sqrt{\lambda}
         \lVert A\rVert_{\textbf{L}^2(\O_T)}+\dfrac{1}{\sqrt{\lambda}}+ e^{\vartheta_2\lambda^2} \mathcal E 
        \right).
    \end{aligned}
\end{align}
Now choose  $\alpha=\mathbf e_i $ and $\widetilde{\alpha}= \mathbf e_j$, then using 
\eqref{j161 in Q} and \eqref{j161 bound in Q}, we obtain 
\begin{align}\label{requied in uniqueness}
    \begin{aligned}
     &\left \lvert   \int_{\Omega_T}\widetilde{q}_{ji} 
    \Xi^2_\varrho(t)e^{-\mathrm{i}(t\tau+x\cdot\xi)}e^{-\int_{0}^{\infty} \o\cdot \left(\overrightarrow{A}^j_{(1)}-\overrightarrow{A}^i_{(2)}\right)(t,x+s\o)ds}\ dx dt \right\rvert \\& \qquad \qquad\leq C \varrho^{-2} \langle \tau, \xi\rangle_*^3 \left(\sqrt{\lambda}
         \lVert A\rVert_{\textbf{L}^2(\O_T)}+\dfrac{1}{\sqrt{\lambda}}+ e^{\vartheta_2\lambda^2} \mathcal E 
        \right),\quad \text{for $(\tau,\xi)\in \mathbb{R}\times \mathcal{P}$}.
    \end{aligned}
\end{align}
Moreover, for every $(\tau,\xi)\in \mathbb{R}\times \mathcal{P}$, we observe that
\begin{align}\label{condition on matrix A required for estimating Q}
    \begin{aligned}
 \lvert\mathcal{F}(\Xi_\varrho^2 \widetilde{q}_{ji})(\tau,\xi)\rvert&= \left\lvert \int_{\Omega_T}\widetilde{q}_{ji} 
    \Xi^2_\varrho(t)e^{-\mathrm{i}(t\tau+x\cdot\xi)}\ dx dt\right\rvert\\&\leq    \left\lvert   \int_{\Omega_T}\widetilde{q}_{ji} 
    \Xi^2_\varrho(t)e^{-\mathrm{i}(t\tau+x\cdot\xi)}e^{-\int_{0}^{\infty} \o\cdot \left(\overrightarrow{A}^j_{(1)}-\overrightarrow{A}^i_{(2)}\right)(t,x+s\o)ds}\ dx dt \right\rvert\\&\qquad+ \left\lvert   \int_{\Omega_T}\widetilde{q}_{ji} 
    \Xi^2_\varrho(t)e^{-\mathrm{i}(t\tau+x\cdot\xi)}\left(e^{-\int_{0}^{\infty} \o\cdot \left(\overrightarrow{A}^j_{(1)}-\overrightarrow{A}^i_{(2)}\right)(t,x+s\o)ds}-1\right)\ dx dt\right\rvert\\&\leq C \varrho^{-2} \langle \tau, \xi\rangle_*^3 \left(\sqrt{\lambda}
         \lVert A\rVert_{\textbf{L}^2(\O_T)}+\dfrac{1}{\sqrt{\lambda}}+ e^{\vartheta_2\lambda^2} \mathcal E 
        \right)\\&\qquad +C \int_{\O_T} \int_{0}^\infty 
        \lvert \left(\overrightarrow{A}^j_{(1)}-\overrightarrow{A}^i_{(2)} \pm 
        \overrightarrow{A}^j_{(2)}
        \right)\rvert(t,x+s\o)\ ds dxdt
        \\&\leq C \varrho^{-2} \langle \tau, \xi\rangle_*^3 \left(\sqrt{\lambda}
         \lVert A\rVert_{\textbf{L}^2(\O_T)}+\dfrac{1}{\sqrt{\lambda}}+ e^{\vartheta_2\lambda^2} \mathcal E 
        \right)+ C\lVert A \rVert_{\textbf{L}^2(\O_T)}\\&\leq C \varrho^{-2} \langle \tau, \xi\rangle_*^3 \left(\sqrt{\lambda}
         \lVert A\rVert_{\textbf{L}^2(\O_T)}+\dfrac{1}{\sqrt{\lambda}}+ e^{\vartheta_2\lambda^2} \mathcal E 
        \right).
    \end{aligned}
\end{align}
Now if we 
implement Lemma \ref{estimate; interpolation}
to the function $ \mathcal{F}(\Xi_\varrho^2 \widetilde{q}_{ji})$, then for $R\geq 1$ and $\varrho\in(0,T/4)$, there exists $\theta \in(0,1)$ such that
\begin{align}\label{estimate; qji l infinity}
    \begin{aligned}
        \left\lVert \mathcal{F}(\Xi_\varrho^2 \widetilde{q}_{ji})\right\rVert_{L^{\infty}(B_R)} \leq Ce^{R(1-\theta)}\varrho^{-2\theta} R^{3\theta} \left(\sqrt{\lambda}
         \lVert A\rVert_{\textbf{L}^2(\O_T)}+\dfrac{1}{\sqrt{\lambda}}+ e^{\vartheta_2\lambda^2} \mathcal E 
        \right)^{\theta}.
    \end{aligned}
\end{align}
To derive the stability estimate for the matrix-valued potential, we proceed in a manner analogous to the proof for the convection coefficient. By Plancherel's theorem, we obtain
\begin{align}\label{estimate:A_fourier_split in Q}
\|\Xi_\varrho^2 \widetilde{q}_{ji}\|_{L^2(\Omega_T)}^{\frac{2}{\theta}}
&=
\left(
\int_{\mathbb R^{1+n}}
\left|
\mathcal{F} \left(\Xi_\varrho^2 \widetilde{q}_{ji}\right) (\tau,\xi)
\right|^2
\,d\tau\,d\xi
\right)^{\frac1\theta}
\nonumber\\
&=
\left(
\int_{B_R}
\left|
\mathcal{F} \left(\Xi_\varrho^2 \widetilde{q}_{ji}\right) (\tau,\xi)
\right|^2
\,d\tau\,d\xi
+
\int_{B_R^c}
\left|
\mathcal{F} \left(\Xi_\varrho^2 \widetilde{q}_{ji}\right) (\tau,\xi)
\right|^2
\,d\tau\,d\xi
\right)^{\frac1\theta}\\&\leq  C\left(
\left(\int_{B_R}
\left|
\mathcal{F} \left(\Xi_\varrho^2 \widetilde{q}_{ji}\right)(\tau,\xi)
\right|^2
\,d\tau\,d\xi\right)^{\frac{1}{\theta}}+\left(\int_{B_R^c}
\left|
\mathcal{F} \left(\Xi_\varrho^2 \widetilde{q}_{ji}\right)(\tau,\xi)
\right|^2
\,d\tau\,d\xi
\right)^{\frac{1}{\theta}}\right).
\end{align}
Using the same argument as in the derivation of \eqref{estimate:I2}, we obtain
\begin{align}\label{estimate:I2_final in Q}
\int_{B_R^c}
\left|
\mathcal{F} \left(\Xi_\varrho^2 \widetilde{q}_{ji}\right)(\tau,\xi)
\right|^2
\,d\tau\,d\xi
\le
\frac{C}{\varrho^2R^2}.
\end{align}
Also, using estimate~\eqref{estimate; qji l infinity}, we have 
\begin{align}
\int_{B_R}
\left|
\mathcal{F} \left(\Xi_\varrho^2 \widetilde{q}_{ji}\right)(\tau,\xi)
\right|^2
\,d\tau\,d\xi
&\le
Ce^{2R(1-\theta)}\varrho^{-4\theta} R^{1+n+6\theta} \left(\sqrt{\lambda}
         \lVert A\rVert_{\textbf{L}^2(\O_T)}+\dfrac{1}{\sqrt{\lambda}}+ e^{\vartheta_2\lambda^2} \mathcal E 
        \right)^{2\theta}.
\end{align}
Using the fact that $0<\theta<1$ and adjusting the constant if necessary, the above estimate becomes
\begin{align}\label{estimate:I12_simplified in Q}
\begin{aligned}
\left(\int_{B_R}
\left|
\mathcal{F} \left(\Xi_\varrho^2 \widetilde{q}_{ji}\right)(\tau,\xi)
\right|^2
\,d\tau\,d\xi\right)^{1/\theta}
&\le
C
\left(e^{\frac{2R(1-\theta)}{\theta}}
\varrho^{-4}
R^{6+\frac{n+1}{\theta}}\right)
\left(
\lambda \lVert A\rVert_{\textbf{L}^2(\O_T)}^2+\dfrac{1}{\lambda}
+
e^{2\vartheta_2\lambda^2}\mathcal E^2
\right).
\end{aligned}
\end{align}
Now, if we denote
\(
\kappa
:=
6+\frac{n+1}{\theta},
\)
then inequality \eqref{estimate:I12_simplified in Q} becomes
\begin{align}\label{estimate:I12_alpha in Q}
\left(\int_{B_R}
\left|
\mathcal{F} \left(\Xi_\varrho^2 \widetilde{q}_{ji}\right)(\tau,\xi)
\right|^2
\,d\tau\,d\xi\right)^{1/\theta}
\le
C
\frac{R^\kappa}{\varrho^4}
e^{\frac{2R(1-\theta)}{\theta}}
\left(
\lambda \lVert A\rVert_{\textbf{L}^2(\O_T)}^2+\dfrac{1}{\lambda}
+
e^{2\vartheta_2\lambda^2}\mathcal E^2
\right).
\end{align}
Combining Equations~\eqref{estimate:A_fourier_split in Q}, \eqref{estimate:I2_final in Q},
and \eqref{estimate:I12_alpha in Q}, we conclude that
\begin{align}\label{estimate:A_data_before_choice in Q}
\|\Xi_\varrho^2\widetilde{q}_{ji}\|_{L^2(\Omega_T)}^{\frac{2}{\theta}}
\le
C
\left[
\frac{\lambda R^\kappa}{\varrho^4}
e^{\frac{2R(1-\theta)}{\theta}}\lVert A\rVert_{\textbf{L}^2(\O_T)}^2
+
\frac{R^\kappa}{\varrho^4}
e^{\frac{2R(1-\theta)}{\theta}}
e^{2\vartheta_2\lambda^2}\mathcal E^2+\frac{R^\kappa}{\lambda\varrho^4}
e^{\frac{2R(1-\theta)}{\theta}}
+
\frac1{\varrho^{\frac2\theta}R^{\frac2\theta}}
\right].
\end{align}
Thus, a upper bound for $\widetilde{q}_{ji}$ in $L^2$-norm is given by 
\begin{align}\label{estimate:A_before_balance in Q}
\hspace{-.7cm}\|\widetilde{q}_{ji} \|_{L^2(\Omega_T)}^{\frac2\theta}
\le
C
\left[
\frac{\lambda R^\kappa}{\varrho^4}
e^{\frac{2R(1-\theta)}{\theta}}\lVert A\rVert_{\textbf{L}^2(\O_T)}^2
+
\frac{R^\kappa}{ \varrho^4}
e^{\frac{2R(1-\theta)}{\theta}}
e^{2\vartheta_2\lambda^2}\mathcal E^2+\frac{R^\kappa}{\lambda\varrho^4}
e^{\frac{2R(1-\theta)}{\theta}}
+
\frac1{\varrho^{\frac2\theta}R^{\frac2\theta}}+\varrho^{\frac{1}{\theta}}
\right].
\end{align}
We now choose the parameters so that the last three terms in the above inequality are comparable. Take
\(
\varrho
=
R^{-2/3}
\) and \(
\lambda
=
R^{\kappa+\frac83+\frac{2}{3\theta}}
e^{\frac{2R(1-\theta)}{\theta}}.
\)
With this choice, we have
\begin{align*}
  \frac{R^\kappa}{\lambda\varrho^4}
e^{\frac{2R(1-\theta)}{\theta}}
=
R^{-\frac{2}{3\theta}}  
  \text{ and } \frac1{\varrho^{\frac2\theta}R^{\frac2\theta}}
=
\varrho^{\frac1\theta}
=
R^{-\frac{2}{3\theta}}. 
\end{align*}
Therefore, the last three terms in
estimate~\eqref{estimate:A_before_balance in Q} are bounded by
$CR^{-2/(3\theta)}$. Now, on combining the above expression with Equation \eqref{final norm of A with power theta}, we obtain
\begin{align}
\|\widetilde{q}_{ji} \|_{L^2(\Omega_T)}^{\frac2\theta}
\le
C
\left(e ^{\mu_2 R}\mathcal E^{\theta}
+
e ^{\mu_2 R}|\log|\log \mathcal E||^{-\frac{2}{3}}+e^{e^{\mu_2 R}}\mathcal{E}^{2}+R^{-\frac{2}{3 \theta}}
\right),
\end{align}
for $\mathcal{E}$ small enough and for some constant $\mu_2>0$. Now, if we choose $R= \frac{1}{3\mu_2}\log \log \lvert \log \mathcal{E}\rvert$, the foregoing expression reduces to
\begin{align}
        \|\widetilde{q}_{ji} \|_{L^2(\Omega_T)}^{\frac2\theta}
&\le
C
\left(\mathcal E^{\theta} (\log \lvert \log \mathcal{E}\rvert)^{\frac{1}{3}} 
+
(\log|\log \mathcal E|)^{-\frac{1}{3}}+\mathcal{E}^{2}( \lvert \log \mathcal{E}\rvert)^{\frac{1}{3}} +(3 \mu_2)^{\frac{2}{3 \theta}}(\log \log \lvert \log \mathcal{E}\rvert)^{-\frac{2}{3 \theta}}
\right)\nonumber\\&\leq C
\left(\mathcal E^{\theta} (\log \lvert \log \mathcal{E}\rvert)^{\frac{1}{3}} 
+
(\log|\log \mathcal E|)^{-\frac{1}{3}}+\mathcal{E}^{2} \lvert \log \mathcal{E}\rvert +(\log \log \lvert \log \mathcal{E}\rvert)^{-\frac{2}{3 \theta}}
\right),\label{5.48}
\end{align}
 for $\mathcal E$ small enough. Now, to reduce the right-hand side of the above expression, we further impose more smallness on $\mathcal{E}$. In particular, we choose $\mathcal{E}$ small enough such that following holds: 
 \begin{itemize}
 \item[(a)] $\mathcal E^{\theta} \leq C (\log \lvert \log \mathcal{E}\rvert)^{-2/3}$ ,
  \item[(b)] $(\log \lvert \log \mathcal{E}\rvert)^{-1/3} \leq  (\log \log \lvert \log \mathcal{E}\rvert)^{-1/3} $,
     \item[(c)] $\mathcal{E} \lvert \log \mathcal{E}\rvert \leq C$.
 \end{itemize}
The use of above assumptions in Equation \eqref{5.48} yields
\begin{align}
    \begin{aligned}
            \|\widetilde{q}_{ji} \|_{L^2(\Omega_T)}^{\frac2\theta}
&\le
C 
\left( (\log \log \lvert \log \mathcal{E}\rvert)^{-\frac{1}{3}} 
+\mathcal{E} +(\log \log \lvert \log \mathcal{E}\rvert)^{-\frac{2}{3 \theta}}
\right).
    \end{aligned}
\end{align}
After simplification, we get
\begin{align}\label{final norm for q tilde}
    \begin{aligned}
            \|\widetilde{q}_{ji} \|_{L^2(\Omega_T)}
&\le
C 
\left( \mathcal{E}^{\frac{\theta}{2}}+(\log \log \lvert \log \mathcal{E}\rvert)^{-\frac{\theta}{6}} 
\right)
    \end{aligned}
\end{align}
for all $1\leq i,j\leq k$.
Now, if we denote $Q=Q_{(2)}-Q_{(1)}$, then from the notations given by \eqref{simplified notation of all vectors}, we have
\begin{align*}
    Q=\widetilde{Q}+A^2_{(1)}-A^2_{(2)}+\nabla_x\cdot A.
\end{align*}
	The component-wise comparison will give us 
    \begin{align}
        q_{ij}= \left\{
	\begin{array}{ r l  }
\widetilde{q}_{ij},& \text{ if }  \ i\neq j,\\
    \widetilde{q}_{ii}+\nabla_x\cdot \overrightarrow{A}^{i}+ (\overrightarrow{A}^i_{(1)})^2-(\overrightarrow{A}^i_{(2)})^2,& \text{ if } \ i=j.
\end{array}
\right.
    \end{align}
    Thus, we have 
    \begin{align}\label{qji norm in term of qji tilde and A}
        \begin{aligned}
            \lVert q_{ij}\rVert_{L^2(\O_T)}\leq    \lVert \widetilde{q}_{ij}\rVert_{L^2(\O_T)}+C\lVert \overrightarrow{A}^i\rVert_{L^2(0,T;H^1(\O))}.
        \end{aligned}
    \end{align}
    The higher regularity corresponds to the first order perturbation along with the logarithmic convexity for Sobolev norm leads to the following: There exist a constant $C>0$ and $\sigma \in(0,1)$ such that 
    \begin{align*}
        \lVert\overrightarrow{A}^i\rVert_{H^1(\O_T)}\leq C \lVert\overrightarrow{A}^i\rVert^\sigma_{L^2(\O_T)}.
    \end{align*}
  By combining Equations \eqref{final norm of A with power theta}, \eqref{final norm for q tilde}, and \eqref{qji norm in term of qji tilde and A} with the preceding inequality, and subsequently summing over all indices $1 \leq i,j \leq k$, we obtain
  \begin{align}
      \begin{aligned}
           \lVert Q\rVert_{\textbf{L}^2(\O_T)}\leq C\left(
\mathcal{E}^{a_3}+\lvert \log \lvert\log\lvert \log \mathcal{E}\rvert\rvert\rvert^{-a_4}
           \right),
      \end{aligned}
  \end{align}
  for some $C,a_3$ and $a_4>0$. Thus, we conclude that  \begin{align}
      \begin{aligned}
           \lVert Q\rVert_{\textbf{L}^2(\O_T)}\leq C\left(
\lVert\Lambda_1^\sharp-\Lambda_2^\sharp\rVert^{a_3}+\lvert \log \lvert\log\lvert \log \lVert\Lambda_1^\sharp-\Lambda_2^\sharp\rVert\rvert\rvert\rvert^{-a_4}
           \right),
      \end{aligned}
  \end{align}
  for some $C,a_3$ and $a_4>0$. This completes the proof.
  \begin{proof}[Proof of Corollary \ref{Uniqueness for main PDE}]
   Through this corollary, we claim that if we drop the additional assumption \eqref{extra condition imposed on A} from the hypotheses of Theorem \ref{main theorem}, even then the uniqueness is obtained. 
To see this, observe first that the condition \eqref{extra condition imposed on A} was not used in the derivation of the stability estimate for $A:=A_{(1)}-A_{(2)}$. Therefore, inserting $\Lambda^\sharp_1=\Lambda^\sharp_2$ into Equation \eqref{in comparing A, at last, 1 form} yields $A_{(1)}=A_{(2)}$. The condition \eqref{extra condition imposed on A} is invoked for the first time only in Equation \eqref{condition on matrix A required for estimating Q}. Hence, prior to that, in \eqref{requied in uniqueness} we may substitute $A_{(1)}=A_{(2)}$ together with $\Lambda^\sharp_1=\Lambda^\sharp_2$, which gives
\begin{align*}
  \left \lvert   \int_{\Omega_T}\widetilde{q}_{ji} 
    \Xi^2_\varrho(t)e^{-\mathrm{i}(t\tau+x\cdot\xi)}e^{-\int_{0}^{\infty} \o\cdot \left(\overrightarrow{A}^j_{(1)}-\overrightarrow{A}^i_{(2)}\right)(t,x+s\o)ds}\, dx\, dt \right\rvert  \leq C \varrho^{-2} \langle \tau, \xi\rangle_*^3 \left(\dfrac{1}{\sqrt{\lambda}}
        \right).  
\end{align*}
Letting $\lambda\to\infty$ we deduce that
\begin{align*}
    \mathcal{F}\!\left(
q_{ji}(t,x) 
    \Xi^2_\varrho(t)e^{-\int_{0}^{\infty} \o\cdot \left(\overrightarrow{A}^j_{(1)}-\overrightarrow{A}^i_{(2)}\right)(t,x+s\o)ds}
      \right)=0
\quad\text{for all }\tau \in \mathbb{R},\ \xi\in \o^\perp,
\end{align*}
where $\o\in \mathbb{S}^{n-1}$ satisfies $\lvert \o -\o_0 \rvert\leq\dfrac{\epsilon}{2}$. Here
\begin{align*}
    q_{ji}(t,x) 
    \Xi^2_\varrho(t)e^{-\int_{0}^{\infty} \o\cdot \left(\overrightarrow{A}^j_{(1)}-\overrightarrow{A}^i_{(2)}\right)(t,x+s\o)ds}
    \in L^{\infty}(\O_T),
\end{align*}
and this function is extended by zero outside $\O_T$. By the Paley-Wiener theorem, it follows that
\begin{align*}
   q_{ji}(t,x) 
    \Xi^2_\varrho(t)e^{-\int_{0}^{\infty} \o\cdot \left(\overrightarrow{A}^j_{(1)}-\overrightarrow{A}^i_{(2)}\right)(t,x+s\o)ds}=0, 
\end{align*}
for $(t,x)\in \O_T$. 
 This holds for every choice of $\Xi(t)\in C_c^{\infty}(0,T)$, and the exponential factor is non-vanishing; hence we conclude that $q_{ji}=0$ in $\O_T$. Since this argument is valid for all $1\leq i,j\leq k$, we obtain $q_{ji}=0$ for all $1\leq i,j\leq k$, and therefore $Q_{(1)}=Q_{(2)}$ in $\O_T$.
  \end{proof}

\section*{Acknowledgments}\label{sec:acknowledgements}
Parveen Kumar acknowledges the financial support provided by the Council of Scientific and Industrial Research (CSIR), India, through the Senior Research Fellowship (SRF), Grant No. 09/1005(19269)/2024-EMR-I. The authors acknowledge the Department of Mathematics, Indian Institute of Technology Ropar, for providing research facilities supported under the DST-FIST programme (Reference No. SR/FST/MS-I/2018/22(C)).  The authors express their sincere gratitude to Dr.~Manmohan Vashisth for his numerous insightful discussions and constructive suggestions, which have substantially enhanced the overall quality of this work.

\bibliography{references}
\bibliographystyle{alpha}

\end{document}